\documentclass[11pt]{amsart}
\usepackage{amssymb}
\usepackage{amsmath}
\usepackage{mathtools}
\usepackage{mdwtab}
\usepackage[alphabetic]{amsrefs}
\usepackage{ytableau}
\usepackage{tabularray}
\usepackage{caption}
\usepackage{subcaption}
\usepackage{eucal}
\usepackage{bm}
\usepackage{stmaryrd}
\usepackage{tikz}
\usepackage{tikz-cd}
\usetikzlibrary{positioning,fit}
\usepackage[indent]{parskip}
\usepackage{url}

\usepackage{fullpage}
\theoremstyle{plain}
\newtheorem{theorem}{Theorem}[subsection]
\newtheorem{prop}[theorem]{Proposition}
\newtheorem{lemma}[theorem]{Lemma}
\newtheorem{corollary}[theorem]{Corollary}

\theoremstyle{definition}
\newtheorem{dfn}[theorem]{Definition}
\newtheorem{ex}[theorem]{Example}

\newtheorem{rem}[theorem]{Remark}
\numberwithin{equation}{section}

\DeclareMathAlphabet\mathsf{OT1}{cmss}{m}{n}

\newcommand{\ssA}{\mathsf{A}}
\newcommand{\ssC}{\mathsf{C}}
\newcommand{\ssD}{\mathsf{D}}
\newcommand{\ssE}{\mathsf{E}}
\newcommand{\C}{\mathbb{C}}
\newcommand{\scrO}{\mathcal{O}}
\newcommand{\R}{\mathbb{R}}
\DeclareMathOperator{\GL}{GL}
\DeclareMathOperator{\U}{U}
\DeclareMathOperator{\SL}{SL}

\newcommand{\gl}{\mathfrak{gl}}
\DeclareMathOperator{\Sp}{Sp}
\DeclareMathOperator{\RSK}{RSK}
\renewcommand{\sp}{\mathfrak{sp}}
\renewcommand{\O}{\operatorname{O}}
\DeclareMathOperator{\SO}{SO}
\newcommand{\so}{\mathfrak{so}}

\newcommand{\g}{\mathfrak{g}}
\renewcommand{\k}{\mathfrak{k}}
\newcommand{\h}{\mathfrak{h}}
\renewcommand{\q}{\mathfrak{q}}
\newcommand{\p}{\mathfrak{p}}
\newcommand{\F}[2]{F^{(#2)}_{#1}}

\newcommand{\Cleq}[1]{\mathcal{C}^{#1}_{\leqslant k}}
\newcommand{\Ceq}[1]{\mathcal{C}^{#1}_{k}}
\renewcommand{\phi}{\varphi}
\renewcommand{\P}{\mathcal{P}}
\DeclareMathOperator{\supp}{supp}
\DeclareMathOperator{\msupp}{msupp}
\newcommand{\bp}{\mathbf{p}}
\newcommand{\f}{\mathbf{f}}
\newcommand{\scrF}{\mathcal{F}}
\DeclareMathOperator{\Par}{Par}
\DeclareMathOperator{\M}{M}
\DeclareMathOperator{\SM}{SM}
\DeclareMathOperator{\AM}{AM}
\DeclareMathOperator{\SSYT}{SSYT}
\DeclareMathOperator{\SYT}{SYT}
\newcommand{\la}{\lambda}
\newcommand{\ep}{\varepsilon}
\newcommand{\al}{\alpha}
\newcommand{\be}{\beta}
\newcommand{\cor}{{\rm cor}}
\newcommand{\s}{{\rm s}}
\newcommand{\m}{{\rm m}}
\newcommand{\0}{\textcolor{gray!50}{0}}
\newcommand{\sm}[1]{\left[\begin{smallmatrix} #1 \end{smallmatrix}\right]}
\newcommand{\Lplusk}{\Lambda^{\!+\!}(\k)}
\newcommand{\ind}{\ell}
\newcommand{\Wedge}{\scalebox{.85}{$\bigwedge$}}

\allowdisplaybreaks

\begin{document}

\title{Stanley decompositions of rings of invariants and certain  highest weight Harish-Chandra modules}

\author{William Q.~Erickson}
\address{
William Q.~Erickson\\
Department of Mathematics\\
Baylor University \\ 
One Bear Place \#97328\\
Waco, TX 76798} 
\email{Will\_Erickson@baylor.edu}

\author{Markus Hunziker}
\address{
Markus Hunziker\\
Department of Mathematics\\
Baylor University \\ 
One Bear Place \#97328\\
Waco, TX 76798} 
\email{Markus\_Hunziker@baylor.edu}

\begin{abstract}
The first half of this paper is largely expository, wherein we present a systematic combinatorial approach to the theory of polynomial (semi)invariants and multilinear invariants of several vectors and covectors, for the classical groups.
This culminates in a graphical description of graded linear bases.
By applying well-known results of lattice-path combinatorics to Weyl's fundamental theorems of classical invariant theory, we write down Stanley decompositions and Hilbert--Poincar\'e series in terms of families of non-intersecting lattice paths, enumerated with respect to certain corners.

In the second half of the paper, we revisit the (semi)invariants in the first half as a special case of a much broader phenomenon.
On one hand, polynomial invariants of a group $H$ can be generalized to modules of covariants, i.e., $H$-equivariant polynomial functions between $H$-modules.
On the other hand, from the perspective of Roger Howe's theory of dual pairs, these modules of covariants can be viewed as infinite-dimensional simple $(\g, K)$-modules.
This suggests an expanded program in which our goal is to apply combinatorial techniques involving lattice paths in order to write down Hilbert series for arbitrary unitarizable highest-weight $(\g,K)$-modules.
As a preview of future work in this program, we present examples showing how modules of covariants --- even those which are not Cohen--Macaulay, and therefore which we would not expect to be combinatorially nice --- can be decomposed in terms of lattice paths.
We also extend these methods beyond the classical groups.
\end{abstract}




\maketitle

\tableofcontents

\section{Introduction}

We have written this paper with two goals.
In the first half (Sections~\ref{sec:prelim}--\ref{sec:Stanley decomp and Hilbert series}), we present a unified exposition outlining some major results of classical invariant theory through the combinatorial lens of lattice paths.
The connections we mention here are, for the most part, scattered in the literature of the 1990s, and often only implicitly.
For this reason, we believe it is worthwhile to describe this framework in detail for all of the complex classical groups, i.e., the general/special linear, the (special) orthogonal, and the symplectic groups.
Incidental to this first goal, we have included certain dimension formulas which seem to be new.
The primary contribution of this paper, however, and one which is entirely new, is the second half (Sections~\ref{sec:Howe duality and Enright reduction}--\ref{sec:ADE}), where we extend the lattice path approach from (semi)invariants to \emph{co}variants, and from the classical groups to the exceptional groups.
In the overview below, we point out previous work in the literature and we highlight the new results contained in this paper.

\subsection{Classical invariant theory}

Let $H$ be one of the classical groups over the complex numbers: the general/special linear group, the (special) orthogonal group, or the symplectic group.
Let $W$ be a representation of $H$, i.e., a complex vector space along with a group homomorphism $\rho: H \longrightarrow \GL(W)$.
Thus $H$ acts linearly on $W$, and for brevity we write $hw \coloneqq \rho(h)w$.
Then $H$ also acts linearly on $\C[W]$, the space of complex-valued polynomial functions on $W$.
The primary problem in classical invariant theory is to determine generators and relations for the algebra of (polynomial) invariants, denoted by
\[
\C[W]^H:=\{ f\in \C[W] \mid f(hw)=f(w)\ \mbox{for all $h\in H, \: w\in W$}\}.
\]

In the latter half of the 19th century, nearly every paper in invariant theory contained explicit computations of generators and relations for some algebra $\C[W]^H$.
See the survey~\cite{DieudonneCarrell} for a comprehensive overview of the subject.
Central in classical invariant theory is the case where $W$ is a direct sum of vectors and covectors; i.e., if $V$ is the defining representation of $H$, we let $W$ be the space
\[
V^{*p} \oplus V^q:=\underbrace{V^*\oplus \cdots \oplus V^*}_{\mbox{$p$ copies}} \oplus \underbrace{V\oplus \cdots \oplus V}_{\mbox{$q$ copies}}.
\]
(The distinction between vectors and covectors is needed only when $H$ is the general or special linear group.)
Hermann Weyl~\cite{Weyl} determined generators and relations for $\C[V^{*p} \oplus V^q]^H$, in his first and second fundamental theorems of classical invariant theory.
When $H$ is $\GL(V)$, $\O(V)$, or $\Sp(V)$, the generators are certain quadratics $f_{ij}$, and the relations among them are given by the vanishing of determinants (or Pfaffians) of all minors of size $\dim V + 1$.
Hence for these three groups, $\C[W]^H$ is isomorphic to the coordinate ring of a determinantal variety.

Weyl's theorems are called ``fundamental'' in the sense that many problems in invariant theory can be reduced to the case of several vectors and covectors, via polarization and restitution; this is the essence of the ``symbolic'' (or ``umbral'') method that flourished in the 19th century.  (See~\cite{Kraft}.)
One such instance is the classical problem of \emph{multilinear invariants}, i.e., the $H$-invariant tensors in the space
\[
V^{\otimes p} \otimes V^*{}^{\otimes q}:=
\underbrace{V\otimes \cdots \otimes V}_{\mbox{$p$ copies}} \otimes \underbrace{V^*\otimes \cdots \otimes V^*}_{\mbox{$q$ copies}},
\]
where $H$ acts naturally on each tensor factor.
Via the canonical isomorphism $V^{**} \cong V$, one can identify $V^{\otimes p} \otimes V^*{}^{\otimes q}$ with the space of multilinear forms on $V^{*p} \oplus V^q$, which in turn is a single multigraded component of $\C[V^{*p} \oplus V^q]$.
Hence the multilinear invariants are a special case of the polynomial invariants.
We will see that the combinatorial analogue is the relationship between standard and semistandard Young tableaux.

\subsection{Linear bases}

In contrast with the historical goal of finding generators and relations (as an algebra), our main interest in this paper is in describing linear bases.
In particular, we focus on writing down an explicit basis for each graded component of $\C[V^{*p} \oplus V^q]^H$ as a vector space.
If $H$ is $\GL(V)$, $\O(V)$, or $\Sp(V)$, in light of Weyl's fundamental theorems, this problem is equivalent to writing down a graded linear basis for a determinantal ring.
This can be done via standard monomial theory (SMT), as in \cites{DeConciniProcesi,Lakshmibai,Procesi}, where the standard ``monomials'' are products of minors encoded by certain semistandard tableaux.
Then following Sturmfels~\cite{Sturmfels} and Conca~\cite{Conca94}, one can exploit the Robinson--Schensted--Knuth (RSK) correspondence to obtain a bijection between standard monomials and ordinary monomials.
More specifically, the strategy employs the bijections below:
\[
\text{standard monomials} \xrightarrow{{\rm SMT}} \text{tableaux} \xrightarrow{{\rm RSK}} \text{$\mathbb{N}$-matrices} \longrightarrow \text{ordinary monomials and graphs},
\]
where the final arrow denotes viewing a matrix either as the degree matrix of a monomial in variables $z_{ij}$, or as the adjacency matrix of a graph.
Our graphical interpretation in terms of arc diagrams seems to be new, although graphical methods are nearly as old as invariant theory itself; see Sylvester's ``algebro-chemical'' theory outlined in~\cite{Sylvester}, but also Olver's elaboration in~\cite{Olver}*{Ch.~7}, along with modern innovations such as webs and spiders~\cite{Kuperberg}.
See also the work by Stanley~\cite{StanleyCombInvThy} and Proctor~\cite{Proctor} on the interplay of combinatorics and invariant theory.

\noindent \emph{\textbf{New results:}} In the five Propositions~4.$*$.1 (one for each classical group $H$), we describe a linear basis for each graded component of $\C[V^{*p} \oplus V^q]^H$.
We represent the basis elements by arc diagrams, where each edge represents a fundamental invariant $f_{ij}$.
Hence for $H = \GL(V)$, $\O(V)$, and $\Sp(V)$, our graphs are essentially visualizations of the SMT results described above.
In the cases where $H = \SL(V)$ or $\SO(V)$, however, the $f_{ij}$'s are not enough to generate the ring of invariants, and so we attach \emph{hyperedges} to our graphs according to certain rules.
In Propositions~4.$*$.2, by restricting our graphs to obey certain degree conditions, we give a linear basis for the space of multlinear invariants.
As a consequence, in Corollaries~4.$*$.3 we are able to formulate the dimension of this space by enumerating the standard Young tableaux of a certain shape.
Although some of these dimension formulas are presumably known (see~\cite{Smith}*{p.~654} in the case of the orthogonal group, along with~\cite{Westbury} for an approach via crystal bases), we believe that at least the $\SL(V)$ and $\SO(V)$ formulas may be new.

\subsection{Stanley decompositions and Hilbert--Poincar\'e series}

Once we understand linear bases consisting of ordinary monomials, our next step is to write down a \emph{Stanley decomposition}
\[
\C[V^{*p} \oplus V^q]^H = \bigoplus \C[\text{certain subset of the $f_{ij}$'s}] \cdot (\text{monomial in certain $f_{ij}$'s}).
\]
Sturmfels~\cite{Sturmfels}, expanding upon~\cite{Billera}, showed how to find a Stanley decomposition of the coordinate ring of the determinantal varieties of generic $p \times q$ matrices, via shellings of the $k$th order complexes on the poset of matrix coordinates (which we will call $\P_{\GL}$); in this way, the coordinate ring can be identified with a Stanley--Reisner ring.
Sturmfels's technique was extended to symmetric and alternating matrices by~\cite{Conca94} and~\cite{Herzog}, respectively (where in this paper we call the underlying posets $\P_{\O}$ and $\P_{\Sp}$).
See also~\cite{BrunsVetter} concerning determinantal rings.
In all three cases, previous results of Krattenthaler~\cite{Krattenthaler} (generalizing Viennot's ``light and shadow '' \cites{Viennot,HerzogTrung} and Fulton's ``matrix balls'' \cite{Fulton}) yield a combinatorial description of the $k$th order complexes, wherein the facets are the families of $k$ non-intersecting lattice paths in the poset $\P_H$, and the restrictions are certain ``corners'' in the paths.
Thus, when $H = \GL(V)$, $\O(V)$, or $\Sp(V)$, after combining these results with Weyl's fundamental theorems, we obtain Stanley decompositions of the form
\[
    \C[V^{*p} \oplus V^q]^H = \bigoplus_{\text{facets $\f$}} \C[f_{ij} \mid (i,j) \in \f] \cdot f_{\cor(\f)},
\]
where $f_{\cor(\f)}$ denotes the product of all $f_{ij}$'s such that $(i,j)$ is a corner of $\f$.
The Stanley decomposition can then be used to write down a rational expression for the Hilbert--Poincar\'e series; see~\cite{ConcaHerzog94}, for example.
With this we conclude the first half of the paper.

\noindent \emph{\textbf{New results:}} In Propositions~5.$*$.1 we write down the Stanley decomposition and Hilbert--Poincar\'e series of the invariant ring.  (The lone exception is $\SO(V)$, which is addressed in Corollary~\ref{cor:Hilbert series SO_k invariants}, after treating the $\O(V)$ semiinvariants in Proposition~\ref{prop:Hilbert series O_k semis}.)
As with linear bases above, for the groups $\GL(V)$, $\O(V)$, and $\Sp(V)$, these results are mostly a matter of applying Weyl's fundamental theorems to previous results on determinantal rings and lattice paths.
Although these Hilbert series were already well known (see, e.g.,~\cites{EW,EnrightHunziker04}),
the proofs in our paper now explain the numerators via lattice paths.
For the groups $\SL(V)$ and $\SO(V)$, on the other hand, the techniques above must be adapted, and we believe that our Stanley decompositions and Hilbert series are genuinely new results.
As an illustration of our techniques, we refer the reader ahead to Example~\ref{ex:SL_k invariants k3p3q4}, where $H = \SL(V)$ with $\dim V = 3$; merely by counting the corners of certain lattice paths, we are able to write down the Hilbert--Poincar\'e series of the invariant ring as follows:
\begin{align*}
       P(\C[V^{*3} \oplus V^4]^{\SL(V)}; t) &= \frac{1}{(1-t^2)^{12}} + \frac{2t^3}{(1-t^3)(1-t^2)^{12}} + \frac{t^3 (1 +t^2 + t^4)}{(1-t^3)^2 (1-t^2)^{11}} \\
       & \qquad + \frac
        {t^3(1+2t^2)}{(1-t^3)^3 (1-t^2)^{10}} + \frac{t^3}{(1-t^3)^4 (1-t^2)^9} \\[2ex]
        &= \frac{1 + 3 t^2 + 2 t^3 + 6 t^4 + 3 t^5 + 8 t^6 + 3 t^7 + 6 t^8 + 2 t^9 +  3 t^{10} + t^{12}}{(1 - t^2)^9 (1 - t^3)^3 (1 - t^6)}.
    \end{align*}

 \subsection{Howe duality and modules of covariants}  

 The success of the lattice path approach to the theory of invariants motivated us to extend these methods to the broader setting of  \emph{modules of covariants}.
 In classical invariant theory, a \emph{covariant} function on $W$, of type $U$ (where $U$ is an $H$-module), is an $H$-equivariant polynomial function $\varphi:W \longrightarrow U$.
 Note that the invariants are the special case where $U$ is the trivial representation.
 The space of covariants of type $U$ is canonically isomorphic to $(\C[W] \otimes U)^H$, which is a $\C[W]^H$-module in the obvious way: the action of $f \in \C[W]^H$ is given by $f \cdot (\sum_i f_i \otimes u_i) = \sum_i (f f_i) \otimes u_i$.
 The 1990s saw renewed interest in modules of covariants, in particular those with the Cohen--Macaulay property; see especially the work of Van den Bergh~\cites{Vandenberg91,VandenBergh}, Brion~\cite{Brion}, and Broer~\cite{Broer}.

 From the perspective of Roger Howe's theory of dual pairs~\cite{Howe89}, one can view modules of covariants as Harish-Chandra modules, i.e., as $(\g,K)$-modules where $\g$ is a certain Lie algebra which is ``dual'' to $H$ (in the sense of Howe duality).
See also~\cite{Leung} concerning invariant theory in the context of dual pairs.
 In this framework, Jackson~\cite{Jackson} developed a standard monomial theory for modules of covariants of the groups $\GL(V)$ and $\Sp(V)$, along with partial results for $\O(V)$.
 In particular, Jackson described the ring $\C[V^{*p} \oplus V^q]^N$ as a Gr\"obner algebra, where $N$ is the maximal unipotent subgroup of $H$
; this ring is isomorphic to the direct sum of the modules of covariants of types $U$ appearing in the decomposition of $\C[W]$. 
Jackson describes a basis of standard monomials in terms of what he calls $H$-sequences, and defines the generators of the idea of non-standard monomials in terms of what he calls $H$-splits.

\noindent \emph{\textbf{New results:}} 
In Theorem~\ref{thm:Hilbert series GL_k covariants}, we write down a Stanley decomposition and Hilbert--Poincar\'e series for $(\C[V^{*p} \otimes V^q] \otimes U)^{\GL(V)}$, where $U$ is an irreducible polynomial representation of $\GL(V)$ or the dual thereof.
(The case of generic rational representations involves some subtlety, and will be described fully in a forthcoming paper.)
Our essential tool is a new type of ``hybrid'' family of lattice paths, combined with Howe duality and Jackson's standard monomial theory.
Strikingly, we are always able to express the Hilbert series as a positive combination of sums over lattice paths, even when the module of covariants is not Cohen--Macaulay.
For an instance of this, see Example~\ref{ex:GL covariants}, where $\dim V = 3$, and where we obtain the Hilbert--Poincar\'e series by summing the following rational expressions:
\begin{align*}
    &t^3 \left(8 \cdot\frac{1+t^2+t^4+t^6}{(1-t^2)^{15}} + 8 \cdot \frac{1+2t^2 + 3t^4}{(1-t^2)^{14}} + 4 \cdot \frac{1 + 3t^2}{(1-t^2)^{13}} + 0 \cdot \frac{1}{(1-t^2)^{12}}\right) \\[2ex]
    = \hspace{1ex}& \frac{20t^3 + 20 t^5 - 4 t^7 - 4 t^9}{(1 - t^2)^{15}}.
\end{align*}
The covariants of the orthogonal group are combinatorially more subtle, but in this paper we obtain a Stanley decomposition and Hilbert--Poincar\'e series for the special case $(\C[V^n] \otimes \Wedge^m V)^{\O(V)}$, for $m \leq \dim V$; see Example~\ref{ex:O covariants}.
Finally, in Theorem~\ref{thm:Sp covariants}, we give a Stanley decomposition and Hilbert series for $(\C[V^n] \otimes U)^{\Sp(V)}$, where $U$ is any finite-dimensional irreducible representation of $\Sp(V)$.

 \subsection{Beyond classical groups}

Recall the dual pairs $(H,\g)$ in the setting of Howe duality above.
As a $\g$-module, the invariant algebra $\C[V^{*p} \oplus V^q]^H$ is known as the \emph{$k$th Wallach representation}, where $k = \dim V$ (or $\frac{1}{2} \dim V$ if $H = \Sp(V)$).
Outside the setting of classical groups, there exist analogous Wallach representations for each Hermitian symmetric pair $(\g,\k)$.
A process called \emph{Enright--Shelton reduction} allows one to interpret the numerator of the Hilbert series of the Wallach representations as the Hilbert series of a finite-dimensional representation of a certain Hermitian symmetric pair $(\g',\k')$, which has lesser rank than $(\g, \k)$.
The reduction process itself maps singular $\k$-dominant weights to regular $\k'$-dominant weights; see~\cites{ES87,ES89,EW,EricksonHunziker23} for a detailed treatment.

\noindent \emph{\textbf{New results:}}
We show, case by case, that for all Hermitian symmetric pairs $(\g,\k)$ with $\g$ simply laced, the Hilbert series of the Wallach representations can be understood via families of non-intersecting lattice paths in the poset of positive noncompact roots, in exactly the same way as the three classical cases above.
We conclude the paper by observing a surprising connection with Enright--Shelton reduction.
In particular, the key to our lattice path approach for the classical groups was the notion of the \emph{corners} of a path.
When considering all families of $k$ non-intersecting lattice paths inside a poset $\P$, we paid special attention to (translations of) the \emph{corner poset} $\operatorname{Cor}_k(\P)$, which is the subset of $\P$ (but endowed with a different poset structure than that inherited from $\P$) in which the corners of an individual path can lie.
See Figure~\ref{fig:facet examples}, where we shade a corner poset $\operatorname{Cor}_k(\P)$ inside each of the three classical posets.
If we rename each poset $\P_H$ as $\P(\g,\k)$ to emphasize the Lie algebra in the dual pair, then we observe the poset isomorphisms
\[
\operatorname{Cor}_k \P(\g,\k) \cong \P(\g',\k').
\]
This extends even to the exceptional groups associated to the simply laced Hermitian symmetric pairs.
In future work, we hope to more fully understand this combinatorial connection.

\section{RSK correspondences and order complexes}
\label{sec:prelim}

\subsection{The ``classical'' posets}
\label{section:posets}

We define the following partially ordered sets (posets), which are planar distributive lattices:
\begin{alignat}{4}
    \P_{\GL} &= \P_{\GL}(p,q) &&\coloneqq \{ (i,j) \mid 1 \leq i \leq p, \: 1 \leq j \leq q \}, && \quad (i,j) \leq (i', j') \iff i \leq i' \text{ and } j \leq j', \\
    \P_{\O} &= \P_{\O}(n) &&\coloneqq \{ (i,j) \mid 1 \leq i \leq j \leq n \}, && \quad (i,j) \leq (i', j') \iff i \leq i' \text{ and } j \geq j', \\
    \P_{\Sp} &= \P_{\Sp}(n) &&\coloneqq \{ (i,j) \mid 1 \leq i < j \leq n \}, && \quad (i,j) \leq (i', j') \iff i \leq i' \text{ and } j \leq j'.
\end{alignat}
Note that $\P_{\Sp}$ is a subposet of $\P_{\GL}$ for sufficiently large $p$ and $q$, but the partial order on $\P_{\O}$ is different than the partial order in the other two.  
The support of a $p \times q$ matrix is a subset of $\P_{\GL}$, the support of an upper-triangular $n \times n$ matrix is a subset of $\P_{\O}$, and the support of an $n \times n$ strictly upper-triangular matrix is a subset of $\P_{\Sp}$.  
For this reason, we will depict these posets using matrix coordinates, with $(1,1)$ in the upper-left; hence our pictures are rotations of the Hasse diagrams. 

As is standard in order theory, we define a \emph{chain} to be a totally ordered subset, and an \emph{antichain} to be a subset whose elements are pairwise incomparable.  
The \emph{height} of a poset is the size of its largest chain; equivalently, by Mirsky's theorem, the height is the minimum number of antichains into which the poset can be partitioned.  
Dually, the \emph{width} of a poset is the size of the largest antichain, which by Dilworth's theorem equals the minimum number of chains into which the poset can be partitioned.  
We will use the term \emph{strict chain} for a chain in which $i < i'$ and $j < j'$ for each pair of distinct elements $(i,j) \leq (i', j')$.  
Given a nonnegative integer matrix $M$, we define its \emph{support} $\supp(M) \coloneqq \{ (i,j) \mid M_{ij} \neq 0\}$, and its \emph{multisupport} $\msupp(M) \coloneqq \{(i,j)^{M_{ij}}\}$, where the superscript denotes the multiplicity of each element in the multiset.  We will refer to the \emph{height} of $\msupp(M)$, meaning the size of its largest chain including multiplicities (equivalently, the minimum number of antichains in a multiset decomposition). 

\subsection{RSK correspondences} By generalizing an algorithm due to Robinson and (independently) Schensted, Knuth \cite{Knuth} defined the celebrated Robinson--Schensted--Knuth (RSK) correspondence, which is a bijection sending each pair of semistandard Young tableaux (SSYT's) of the same shape to a matrix with entries in $\mathbb{N} \coloneqq \mathbb{Z}_{\geq 0}$.

Let $[n] \coloneqq \{1, \ldots, n\}$.
We define the set
\[
\SSYT(\mu,n) \coloneqq \{ \text{SSYT's with shape $\mu$ and entries in $[n]$}\},
\]
where $\mu$ is the partition giving the row lengths.
We write $\Par(a \times b)$ to denote the set of partitions whose Young diagram fits inside a rectangle with $a$ rows and $b$ columns; if we do not wish to restrict the number of columns, then we write $\Par(a \times \infty)$.  
Below we summarize the RSK correspondence, along with two variations we will need.  
We adapt the presentation slightly for the purposes of this paper.  
The subscripts $\GL$, $\O$, and $\Sp$ refer to the classical groups we will introduce in Section~\ref{sec:invariant theory}.
We write $\M_{p,q}(\mathbb{N})$, $\SM_{n}(\mathbb{N})$, and $\AM_{n}(\mathbb{N})$ to denote generic matrices, symmetric matrices, and alternating matrices, respectively.
Explicit details of the following correspondence can be found in~\cite{Knuth}*{\S3}.

\begin{prop}[Knuth]
    \label{prop:RSK_GL}  There is a bijection
\[
\RSK_{\GL} : \bigcup_{\mathclap{\mu \in \Par(\min\{p,q\} \times \infty)}} \: \SSYT(\mu,p) \times\SSYT(\mu,q) \longrightarrow \M_{p,q}(\mathbb{N})
\]
with the following properties. Let $M = \RSK_{\GL}(T,U)$, where $T$ and $U$ have shape $\mu$: 
\begin{enumerate}
     \item  We have $\sum_{i,j} M_{ij} = |\mu|$.


    \item In $\P_{\GL}$, the width of $\supp(M)$ equals the number of rows in the Young diagram of $\mu$, while the height of $\msupp(M)$ equals the number of columns.

 \end{enumerate}

\end{prop}

By applying Knuth's ``dual insertion'' algorithm on ``dual tableaux'' (transposes of SSYT's), we obtain a similar bijection.  
Burge \cite{Burge}*{p.~22} uses an argument similar to Knuth's to spell out this bijection explicitly, thus associating each dual tableau to a matrix in $\SM_n(\mathbb N)$ whose diagonal entries are even.  
(The same construction is described in \cite{Conca94}*{p.~410}.) For this paper, we will replace the codomain by the set of upper-triangular matrices; clearly any such matrix $M$ corresponds uniquely to one of Burge's matrices $M+M^t$.  
We write the correspondence in terms of SSYT's rather than dual tableaux:

\begin{prop}[Burge, Conca]
    \label{prop:RSK_O}  There is a bijection
\[
 \RSK_{\O} : \bigcup_{\mathclap{\substack{\mu \in \Par(n \times \infty)\\ \textup{with even row lengths}}}} \: \SSYT(\mu,n) \longrightarrow 
 \{M \in \M_n(\mathbb N) \mid \textup{$M$ upper-triangular}\}
 \]
with the following properties. Let $M = \RSK_{\O}(T)$, where $T$ has shape $\mu$: 
\begin{enumerate}
     \item  We have $\sum_{i,j} M_{ij} = |\mu|/2$.


    \item In $\P_{\O}$, the width of $\supp(M)$ equals the number of rows in the Young diagram of $\mu$, while the height of $\msupp(M)$ equals half the number of columns.
    
    \end{enumerate}
\end{prop}

Knuth's observation that $\RSK_{\GL}(T,U) = \left[\RSK_{\GL}(U,T)\right]^t$ leads to a bijection between single SSYT's and symmetric matrices.  
Upon setting $p = q = n$, we have the bijection
 \begin{align}
     \label{RSK_GL on twins}
     \begin{split}
     \bigcup_{\mathclap{\mu \in \Par( n \times \infty)}} \SSYT(\mu,n) & \longrightarrow \{M \in \SM_n(\mathbb N) \mid \text{$M$ upper-triangular}\},\\
     T &\longmapsto \RSK_{\GL}(T,T),\\   
     \text{\# odd-length columns in $T$} & \;\; = \;\; \text{trace of $\RSK_{\GL}(T,T)$}.
     \end{split}
 \end{align}
Therefore, by restricting to those shapes with even-length columns, and replacing symmetric matrices by their upper-triangular parts, we obtain our final RSK variant (see  \cite{Knuth}*{\S4}, or the equivalent construction in~\cite{Burge}*{\S2} on even-column tableaux):

\begin{prop}[Knuth]
    \label{prop:RSK_Sp}  There is a bijection
\[
\RSK_{\Sp} : \bigcup_{\mathclap{\substack{\mu \in \Par(n \times \infty)\\ \textup{with even column lengths}}}} \: \SSYT(\mu,n) \longrightarrow 
 \{M \in \M_n(\mathbb N) \mid \textup{$M$ strictly upper-triangular}\}
\]
with the following properties. Let $M = \RSK_{\Sp}(T)$, where $T$ has shape $\mu$: 
\begin{enumerate}
     
     \item  We have $\sum_{i,j} M_{ij} = |\mu|/2$.


    \item In $\P_{\Sp}$, the width of $\supp(M)$ equals half the number of rows in the Young diagram of $\mu$, while the height of $\msupp(M)$ equals the number of columns.
 \end{enumerate}
\end{prop}

\subsection{Multichain and antichain decompositions}

Suppose that $M$ is a matrix obtained from one of the RSK correspondences above.  
In this subsection, we recursively define two important decompositions of the multisupport of $M$, into chains and antichains.  
The number of chains (resp., antichains) determines the number of rows (resp., columns) in the corresponding tableau, as detailed in part (2) of the RSK propositions in the previous subsection.
This provides a two-dimensional visualization of the original context of RSK, in terms of extracting weakly increasing and strongly decreasing subsequences from two-row arrays.
The methods below are all variations on constructions of Krattenthaler~\cite{Krattenthaler}*{Fig.~8}, Fulton~\cite{Fulton}*{\S4.2}, and Herzog--Trung~\cite{HerzogTrung}*{pp.~14, 27}, who generalized the ``light and shadow'' technique of Viennot~\cite{Viennot}.

First we regard $\msupp(M)$ as a multiset with elements in $\P_{\GL}$.  Define the multiset
\begin{align}
\label{C*1 and C1 for GL_k}
\begin{split}
   C^*_1 = C^*_1(M) &\coloneqq \{ (i,j) \in \msupp(M) \mid i' \geq i\text{ or } j' \leq j \text{ for all }(i',j') \in \supp(M)\},\\
    C_1 = C_1(M) &\coloneqq \{ (i,j) \in \msupp(M) \mid i' \leq i\text{ or } j' \geq j \text{ for all }(i',j') \in \supp(M)\}.
    \end{split}
\end{align}
Then for $i > 1$, define $C^*_{i+1} \coloneqq C^*_1(M \setminus \bigcup_{j = 1}^{i} C^*_j)$, and likewise for $C_{i+1}$.  
The process terminates once $\msupp(M)$ has been exhausted. 
One can visualize $C^*_1$ (resp., $C_1$) as the ``northeast (resp., southwest) border'' of $\msupp(M)$; then each successive chain is obtained by peeling off the previous chain and taking the border of the remaining multisupport.  

We construct antichains $D_i$ in a similar way:
\begin{equation}
    \label{D1 for GL_k}
    D_1 = D_1(M) \coloneqq \{ (i,j) \in \supp(M) \mid i' > i \text{ or } j' > j \text{ for all }(i',j') \in \supp(M)\},
\end{equation}
and $D_{i+1} \coloneqq D_1(M \setminus \bigcup_{j=1}^i D_j)$, with the process terminating once $\msupp(M)$ has been exhausted.  
One can visualize $D_1$ as the corners of the northwest border of $\msupp(M)$, and iterate for each $D_i$ by removing the previous corners.  
(We will use the term \emph{corner} in a technical sense in the next subsection, in a seemingly different context; there is actually a close connection, however, which will be explained in the proof of Proposition~\ref{prop:Hilbert series GL_k}.)  See Figure~\ref{fig:Ci Di for GL_k} for a full example of the $C^*_i$, $C_i$, and $D_i$.
Item (1) in the following Lemma will be an especially important ingredient in the new results of this paper.

\begin{lemma}
\label{lemma:Ci and Di for GL_k}

    Let $M = \RSK_{\GL}(T,U)$, with $C^*_i$, $C_i$, and $D_i$ as in~\eqref{C*1 and C1 for GL_k} and~\eqref{D1 for GL_k}.
    
    \begin{enumerate} 

    \item Let $(a_1, \ldots, a_w)$ be the first column of $T$, and $(b_1, \ldots, b_w)$ the first column of $U$.  
    Then $a_i$ is the smallest row index in $C^*_i$, and $b_i$ is the smallest column index in $C_i$.

    \item The underlying set of each $C^*_i$ is a maximal chain in $\supp(M) \setminus \bigcup_{j=1}^{i-1} C^*_j$, and the number of $C^*_i$'s equals the number of rows in $T$ (and in $U$).
    The same holds for the $C_i$'s.
    
    \item Each $D_i$ is a maximal antichain in $\msupp(M) \backslash \bigcup_{j=1}^{i-1} D_j$, and the number of $D_i$ equals the number of columns in $T$ (and in $U$). 
    Moreover, in $\P_{\GL}$, the upper-order ideal generated by $D_i$ contains the upper-order ideal generated by $D_{i+1}$.
    
    \end{enumerate}
\end{lemma}

\begin{proof}\

\begin{enumerate}
   \item This is a direct consequence of the $\RSK_{\GL}$ construction defined in~\cite{Knuth}*{\S3}.  
    Note that in this paper, our convention is that $T$ is the recording tableau, while $U$ is the insertion tableau.  
    Clearly $a_1$ is the northwestern-most element of $C^*_1$. 
    Recall that $T$ and $U$ are constructed following the lexicographical order on $\msupp(M)$.
    Therefore, the second row is created in $T$ exactly when the next element in $\msupp(M)$ increases the cumulative width of the multisupport, and the entry $a_2$ in this new box is the row index of this element.  
    This element is the northwestern-most element in $C^*_2$, and so the lemma holds for $a_2$.  
    Proceeding in this way, we see that the first entry in row $i$ of $T$ is the row index of the northwestern-most element in $C^*_i$.  
    This proves the lemma for the $a_i$.  
    Since $\RSK_{\GL}(U,T) = M^t$, and since the constructions of $C^*_i$ and $C_i$ are transposes of each other, we automatically have the proof for the $b_i$ as well.
   
   \item Suppose that $C^*_1$ contained two incomparable elements $(i,j)$ and $(i',j')$, with $i'<i$.  
   Then the definition of $C^*_1$ forces  $j' \leq j$, which is a contradiction.  
   The identical argument holds for all $C^*_i$, which are therefore chains. 
   We obtain a similar contradiction if we assume that an element of $\supp(M) \setminus \bigcup_{j=1}^i C^*_j$ is comparable with each element inside $C^*_i$; therefore we have maximal chains, and the number of such chains is the width of $\supp(M)$.
   The result now follows from part (2) of Proposition~\ref{prop:RSK_GL}.

    \item Our $D_i$ decomposition is nothing other than Fulton's matrix-ball construction~\cite{Fulton}*{\S4.2}, where $D_i$ consists of all matrix coordinates containing a ball labeled $i$.  
    The maximal antichain property, the number of $D_i$, and the inclusion of upper-order ideals all follow directly from our definition of $D_i$, or, more graphically, from Fulton's visualization. \qedhere
    \end{enumerate}
\end{proof}

\begin{figure}[t]
    \centering
    \input{Matrix_decomp_figures/CD_GL_fig}
    
    \caption{Example of the decompositions in Lemma~\ref{lemma:Ci and Di for GL_k}, where we represent multisets by the matrices they support.  
    We observe the three points of the lemma: 
    (1) the first column of $T$ encodes the smallest row index in each $C^*_i$, and the first column of $U$ encodes the smallest column index in each $C_i$; (2) the number of $C^*_i$ and the number of $C_i$ equals the number of rows in $T$ and $U$; 
    (3) the number of $D_i$ equals the number of columns in $T$ and $U$, and the upper-order ideal generated by each $D_i$ contains the upper-order ideal generated by $D_{i+1}$.}
    \label{fig:Ci Di for GL_k}
\end{figure}

Next we define the analogous decompositions where we regard $\msupp(M)$ as a multiset with elements of $\P_{\O}$.  
Define the multiset
\begin{equation}
    \label{C1 for O_k}
    C_1 = C_1(M) \coloneqq \{ (i,j) \in \msupp(M) \mid i' \geq i \text{ or } j' \geq j \text{ for all } (i',j') \in \supp(M)\},
\end{equation}
and iterate so that $C_{i+1} \coloneqq C_1(M \setminus \bigcup_{j=1}^i C_j)$ until $\msupp(M)$ is exhausted.  
One can visualize each $C_i$ as the northwest border of the remaining part of $\msupp(M)$.  
Likewise, we define 
\begin{equation}
    \label{D1 for O_k}
    D_1 = D_1(M) \coloneqq \{(i,j) \in \supp(M) \mid i' > i \text{ or } j' < j \text{ for all }(i', j') \in \supp(M)\},
\end{equation}
iterating for each $D_i$ exactly as before.  
One can visualize each $D_i$ as the corners of the northeast border of the remaining part of $\msupp(M)$.

\begin{lemma}
\label{lemma:Ci and Di for O_k}

    Let $M = \RSK_{\O}(T)$, with $C_i$ and $D_i$ as in~\eqref{C1 for O_k} and~\eqref{D1 for O_k}.
    
    \begin{enumerate} 

    \item Let $(a_1, \ldots, a_w)$ be the entries in the first column of $T$.  
    Then $a_i$ is the smallest row index in $C_i$.
    
    \item The underlying set of each $C_i$ is a maximal chain in $\supp(M) \setminus \bigcup_{j=1}^{i-1} C_j$, and the number of $C_i$ equals the number of rows in $T$.
    
    \item Each $D_i$ is a maximal antichain in $\msupp(M) \backslash \bigcup_{j=1}^{i-1} D_j$, and the number of $D_i$ equals half the number of columns in $T$. 
    Moreover, in $\P_{\O}$, the upper-order ideal generated by $D_i$ contains the upper-order ideal generated by $D_{i+1}$.
    
    \end{enumerate}
\end{lemma}

\begin{proof}
    The proofs of all three parts are entirely analogous to those in Lemma~\ref{lemma:Ci and Di for GL_k}, the only substantial difference being that part (1) follows directly from Burge's construction of $\RSK_{\O}$.
\end{proof}

Finally, regarding $\supp(M)$ as a subset of $\P_{\Sp}$, we define $C_i$ and $D_i$ exactly as we did in~\eqref{C*1 and C1 for GL_k} and~\eqref{D1 for GL_k} for the $\P_{\GL}$ case.  

The proof of the following lemma mimics those of the lemmas above.

\begin{lemma}
\label{lemma:Ci and Di for Sp_2k}

    Let $M = \RSK_{\Sp}(T)$, with $C_i$ and $D_i$ as in~\eqref{C*1 and C1 for GL_k} and~\eqref{D1 for GL_k}.
    
    \begin{enumerate} 

    \item Let $(a_1, \ldots, a_w,b_1,\ldots,b_w)$ be the first column of $T$.
     Then $a_i$ (resp., $b_i$) is the smallest row (resp., column) index in $C_i$.
     
    \item The underlying set of each $C_i$ is a maximal chain in $\supp(M) \setminus \bigcup_{j=1}^{i-1} C_j$, and the number of $C_i$ equals half the number of rows in $T$. 
    
    \item Each $D_i$ is a maximal antichain in $\msupp(M) \setminus \bigcup_{j=1}^{i-1} D_j$, and the number of $D_i$ equals the number of columns in $T$. 
    Moreover, in $\P_{\Sp}$, the upper-order ideal generated by $D_i$ contains the upper-order ideal generated by $D_{i+1}$.
    
    \end{enumerate}
\end{lemma}

\subsection{Shellings of the order complex}
\label{sec:shellings}
We follow Stanley's exposition in~\cite{StanleyAC}*{Ch.~12}.
Let $\Delta$ be a finite abstract simplicial complex.  
Recall that the maximal faces of $\Delta$, with respect to inclusion, are called \emph{facets}. 
 A simplicial complex is said to be \emph{pure} if all facets have the same cardinality.  
 Furthermore, a pure simplicial complex is said to be \emph{shellable} if there exists an ordering $\f_1, \ldots, \f_r$ of its facets with the following property: for all $i = 1, \ldots, r$, the subcomplex generated by $\f_i$ has a unique minimal element not belonging to the subcomplex generated by $\f_1, \ldots, \f_{i-1}$.  
 Such an ordering is called a \emph{shelling}, and the unique minimal element associated with each facet $\f_i$ is called its \emph{restriction}.  
 Shellings are in general not unique, but each choice of shelling decomposes the complex in a canonical way: each face is associated to a unique facet depending on the restrictions it contains. 
 
 Given a poset $\P$, the \emph{$k$th order complex} $\Delta_k = \Delta_k(\P)$ is the simplicial complex on $\P$ whose faces are the  subsets of width $\leq k$.  
 Let $\scrF_k \subset \Delta_k$ be the set of facets of $\Delta_k$.  
 For the three classical posets $\P$ introduced in Section~\ref{section:posets}, the facets can be described as families of $k$ non-intersecting lattice paths on $\P$.  
 By a \emph{(lattice) path}, we mean the union of the points $(i,j)$ lying inside a saturated chain, which can be depicted as a sequence of horizontal and vertical steps in $\P$.  
 For each of our posets $\P$, there is a shelling such that the restrictions of the facets are the points at certain \emph{corners} in the paths. 
 Details follow, and are depicted in Figure~\ref{fig:facet examples}:
 \begin{itemize}
     \item For $\P_{\GL}$, each facet $\f \in \scrF_k$ is the union of non-intersecting paths $\bp_i: (i,1) \rightarrow (p, \: q-i+1)$, for $1 \leq i \leq k$.  
     If we imagine $(1,1)$ in the northwest corner, then each step in a path advances either south or east.  
     The non-intersecting condition forces $\bp_i$ to pass through the point $s_i$ which is $k-i$ steps to the east of its starting point; likewise, $\bp_i$ must pass through the point $t_i$, which is $k-i$ steps north of its endpoint.  
     A \emph{corner} of a path $\bp$ is a point $(i,j)$ such that both $(i-1, \: j)$ and $(i, \: j+1)$ are in $\bp$. 
     (In other words, a corner occurs at each \scalebox{2}{$\llcorner$}-pattern in a path.)  
     Note that in each path $\bp_i$, its corners must form a strict chain within the $(p-k) \times (q-k)$ region whose upper-left corner is immediately below $s_i$ and whose southeast corner is immediately to the west of $t_i$.
     (We shade this region for $\bp_3$ in Figure~\ref{subfig:facet example GL}.)
     Note also that each facet has size $k(p+q-k)$.

     \item For $\P_{\O}$, each facet $\f \in \scrF_k$ is the union of non-intersecting paths $\bp_i$ starting at $(i,n)$, for $1 \leq i \leq k$, where each step in a path advances either south or west.  
     Hence each path $\bp_i$ has its endpoint along the diagonal of $\P_{\O}$, that is, some point of the form $(j,j)$.  
     Each $\bp_i$ must pass through the point $s_i$ which is $k-i$ steps to the west of its starting point.
     A \emph{corner} of a path $\bp$ is a point $(i,j)$ such that both $(i-1, \: j)$ and $(i, \: j-1)$ are in $\bp$, or a point $(j,j)$ such that $(j-1, \: j)$ is in $\bp$.  
     (In other words, a corner occurs at each \scalebox{2}{$\lrcorner$}-pattern in a path, and wherever the path ends with a vertical step.)  
     Note that the corners of $\bp_i$ form a strict chain inside the $(n-k) \times (n-k)$ right triangle whose northeast vertex is immediately south of $s_i$.
     (We shade this region for $\bp_3$ in Figure~\ref{subfig:facet example O}.)
     Note also that each facet has size $\frac{k}{2}(2n-k+1)$.

     \item For $\P_{\Sp}$, each facet $\f \in \scrF_k$ is the union of non-intersecting paths $\bp_i: (1,\: 2i) \rightarrow (n-2i+1, \: n)$, for $1 \leq i \leq k$.  
     Each step in a path advances either south or east.
     The non-intersecting condition forces $\bp_i$ to begin as an alternating east/south path for $2(k-i)+1$ steps until it passes through $s_i$; likewise, it must pass through the point $t_i$ which is the reflection of $s_i$ about the anti-diagonal.    
     A \emph{corner} of $\bp_i$, is an \scalebox{2}{$\llcorner$}-pattern of a path $\bp$ which lies strictly between $s_i$ and $t_i$.  
     The corners of $\bp_i$ must therefore form a strict chain within the $(n-2k-1) \times (n-2k-1)$ right triangle whose northeast vertex is immediately south of $s_i$.
     (We shade this region for $\bp_1$ in Figure~\ref{subfig:facet example Sp}.)
     Note that each facet has size $k(2n-2k-1)$.
 \end{itemize}

Given a facet $\f$ and its decomposition into paths $\bp_i$, we write $\cor(\bp_i)$ for the set of corners of each path, and we set $\cor(\f) \coloneqq \bigcup_{i=1}^k \cor(\bp_i)$.

\begin{rem}
    The shaded regions in Figure~\ref{fig:facet examples}, i.e., the subsets in which the corners of an individual path must lie, play a leading role in the proofs of Propositions~5.$*$.1.
    We also revisit these regions in Section~\ref{sub:ES reduction}, where we call them the \emph{corner posets} ${\rm Cor}_k(\P)$ associated to each clasical poset $\P$.
\end{rem}

\begin{figure}[t]
     \centering
     \begin{subfigure}[t]{0.3\textwidth}
         \centering
         \tikzstyle{corner}=[rectangle,draw=black,fill=red, minimum size = 4pt, inner sep=0pt]
\tikzstyle{smallend}=[circle,fill=black, minimum size = 5pt, inner sep=0pt]

\begin{tikzpicture}[scale=.6, baseline]

\fill[blue, very nearly transparent] (0.5,0.5) rectangle (4.5,3.5);

\draw[gray!50, very thick] (1,1) grid (7,6);

\draw[ultra thick] (1,6)--(4,6)--(4,5) -- (5,5) -- (5,3) -- (7,3) -- (7,1);
\draw[ultra thick] (1,5) -- (2,5) -- (2,4) -- (4,4) -- (4,2) -- (6,2) -- (6,1);
\draw[ultra thick] (1,4)--(1,3)--(2,3)--(2,2)--(3,2)--(3,1)--(5,1);

\node at (1,6) [label={west:$\bp_1$}] {};
\node at (1,5) [label={west:$\bp_2$}] {};
\node at (1,4) [label={west:$\bp_3$}] {};
\node at (7,1) [label={east:$(6,7)$}] {};
\node at (3,6) [label={[label distance=-5pt]north east:$s_1$},smallend] {};
\node at (2,5) [label={[label distance=-5pt]north east:$s_2$},smallend] {};
\node at (1,4) [label={[label distance=-5pt]north east:$s_3$},smallend] {};
\node at (5,1) [label={[label distance=-5pt]north east:$t_3$},smallend] {};
\node at (6,2) [label={[label distance=-5pt]north east:$t_2$},smallend] {};
\node at (7,3) [label={[label distance=-5pt]north east:$t_1$},smallend] {};

\node at (1,3) [corner] {};
\node at (2,2) [corner] {};
\node at (3,1) [corner] {};
\node at (2,4) [corner] {};
\node at (4,2) [corner] {};
\node at (4,5) [corner] {};
\node at (5,3) [corner] {};

\end{tikzpicture}



         \caption{$\Delta_k(\P_{\GL})$: $k=3$, $p=6$, $q=7$.}
         \label{subfig:facet example GL}
     \end{subfigure}
     \hfill
     \begin{subfigure}[t]{0.3\textwidth}
         \centering
         \tikzstyle{corner}=[rectangle,draw=black,fill=red, minimum size = 4pt, inner sep=0pt]
\tikzstyle{smallend}=[circle,fill=black, minimum size = 5pt, inner sep=0pt]

\begin{tikzpicture}[scale=.6, baseline]

\draw[gray!50, very thick] (1,1) grid (6,6);
\draw[white,fill=white] (0.95,0.95) -- (6,0.95) -- (0.95,6) -- cycle;
\draw[gray!50, very thick] (1,6)--(2,6)--(2,5)--(3,5)--(3,4)--(4,4)--(4,3)--(5,3)--(5,2)--(6,2)--(6,1);
\fill[blue, very nearly transparent] (3.5,3.5) -- (6.5,3.5) -- (6.5, 0.5) -- (5.5,0.5) -- (5.5,1.5) -- (4.5,1.5) -- (4.5, 2.5) -- (3.5, 2.5) -- cycle;

\draw[ultra thick] (6,4) -- (6,3) -- (5,3) -- (5,2);
\draw[ultra thick] (6,5) -- (4,5) -- (4,3);
\draw[ultra thick] (6,6) -- (3,6) -- (3,5) -- (2,5);

\node at (6,6) [label={east:$\bp_1$}] {};
\node at (6,5) [label={east:$\bp_2$}] {};
\node at (6,4) [label={east:$\bp_3$}] {};
\node at (6,1) [label={east:$(6,6)$}] {};
\node at (4,6) [label={[label distance=-5pt]north west:$s_1$},smallend] {};
\node at (5,5) [label={[label distance=-5pt]north west:$s_2$},smallend] {};
\node at (6,4) [label={[label distance=-5pt]north west:$s_3$},smallend] {};

\node at (5,2) [corner] {};
\node at (6,3) [corner] {};
\node at (4,3) [corner] {};
\node at (3,5) [corner] {};

\end{tikzpicture}
         \caption{$\Delta_k(\P_{\O})$: $k=3$, $n = 6$.}
         \label{subfig:facet example O}
     \end{subfigure}
     \hfill
     \begin{subfigure}[t]{0.35\textwidth}
         \centering
         \tikzstyle{corner}=[rectangle,draw=black,fill=red, minimum size = 4pt, inner sep=0pt]
\tikzstyle{smallend}=[circle,fill=black, minimum size = 5pt, inner sep=0pt]

\begin{tikzpicture}[scale=.6, baseline]

\draw[gray!50, very thick] (1,1) grid (7,7);
\draw[white,ultra thick,fill=white] (0.95,0.95) -- (7,0.95) -- (0.95,7) -- cycle;
\draw[gray!50, very thick] (3,6) -- (3,5) -- (4,5);
\fill[blue, very nearly transparent] (2.5,5.5) -- (5.5,5.5) -- (5.5,2.5) -- (4.5,2.5) -- (4.5,3.5) -- (3.5,3.5) -- (3.5,4.5) -- (2.5,4.5) -- cycle;

\draw[ultra thick] (3,7) -- (5,7) -- (5,6) -- (6,6) -- (6,4) -- (7,4) -- (7,3);
\draw[ultra thick] (1,7) -- (2,7) -- (2,6) -- (4,6) -- (4,4) -- (5,4) -- (5,3) -- (6,3) -- (6,2) -- (7,2) -- (7,1);

\node at (1,7) [label={north:$\bp_1$}] {};
\node at (3,7) [label={north:$\bp_2$}] {};
\node at (7,1) [label={east:$(7,8)$}] {};
\node at (3,6) [label={[label distance=-5pt]north east:$s_1$},smallend] {};
\node at (4,7) [label={[label distance=-5pt]north east:$s_2$},smallend] {};
\node at (6,3) [label={[label distance=-5pt]north east:$t_1$},smallend] {};
\node at (7,4) [label={[label distance=-5pt]north east:$t_2$},smallend] {};

\node at (4,4) [corner] {};
\node at (5,3) [corner] {};
\node at (5,6) [corner] {};
\node at (6,4) [corner] {};

\end{tikzpicture}
         \caption{$\Delta_k(\P_{\Sp})$: $k=2$, $n=8$.}
         \label{subfig:facet example Sp}
     \end{subfigure}
        \caption{An example of a facet $\f \in \scrF_k$ in the order complex of each classical poset.  
        The red squares indicate the elements in $\cor(\f)$.  In each diagram, for one of the $\bp_i$, we shade the region in which $\cor(\bp_i)$ must be contained.
        The regions for the other $\bp_i$ are found by translating this shaded region diagonally upwards, parallel to the points $s_i$.
        (This shaded region, later called the \emph{corner poset}, will be the subject of Section~\ref{sub:ES reduction}.)}
        \label{fig:facet examples}
\end{figure}
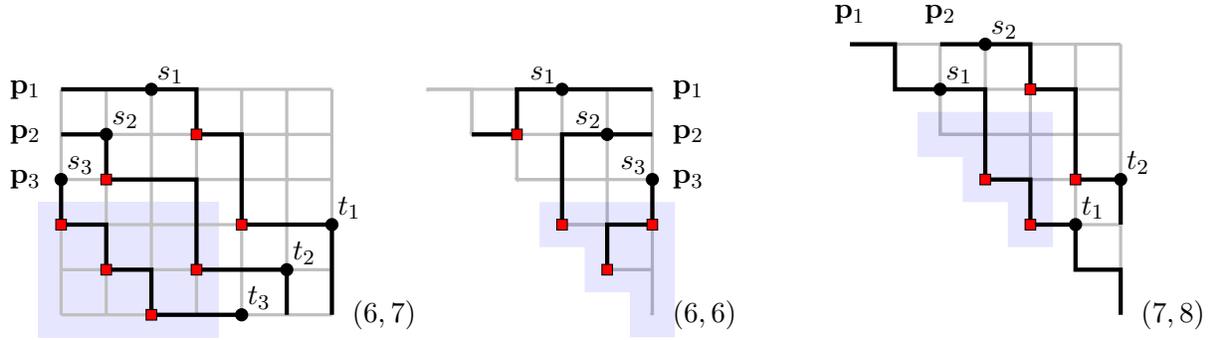
 
Let $\P$ be one of the three classical posets from Section~\ref{section:posets}, and let $\C[z_{ij}] \coloneqq \C[z_{ij} \mid (i,j) \in \P]$.  
Given a subset $S \subset \P$, we define the monomial
\[z_S \coloneqq \prod_{\mathclap{(i,j) \in S}} z_{ij}.
\]
Let $I_{\Delta_k}$ be the ideal generated by all monomials $z_S$ where $S \subset \P$ has width $k+1$.   
The \emph{Stanley--Reisner ring} of $\Delta_k$ is then defined to be the quotient $\C[\Delta_k] \coloneqq \C[z_{ij}]/I_{\Delta_k}$.  
Our shelling of $\Delta_k$ induces a \emph{Stanley decomposition} of the following form~\cite{Stanley82}*{p.~191}:
\begin{equation}
    \label{Stanley decomp}
    \C[\Delta_k] = \bigoplus_{\f \in \scrF_k} z_{\cor(\f)} \C[z_{ij} \mid (i,j) \in \f].
\end{equation}
 
\section{Fundamental theorems of classical invariant theory}
\label{sec:invariant theory}

Let $H$ be one of the complex classical groups, to be defined in the subsections below.
Let $W$ be a finite-dimensional representation of $H$.
We define the \emph{ring of invariants}
\[
\C[W]^H \coloneqq \{f \in \C[W] \mid f(hw) = f(w) \text{ for all } h \in H, \: w \in W\}.
\]
In his first and second fundamental theorems (FFT and SFT) for each classical group, Hermann Weyl \cite{Weyl} determined generators and relations, respectively, for $\C[W]^H$.  
In each case, the generators are quadratic, and the relations among them are given by the vanishing of determinants or Pfaffians.  
We present the details below for each classical group.  
(We will use the name $\pi^*$ for a certain surjective homomorphism of algebras in the following discussions; the reason for this notation will become clear in Section~\ref{sec:Howe duality and Enright reduction}, where we introduce a map $\pi : W \rightarrow \p^+$ in the context of Howe duality.)

\subsection{The general linear group}
\label{sub:GL_k classical inv th}

Let $\dim V = k$.  
The \emph{general linear group} $\GL(V)$ is the group of all invertible linear operators on $V$.
Let $(v^*_1, \ldots, v^*_p, v_1, \ldots, v_q) \in V^{*p} \oplus V^q$.  
The FFT states that $\C[V^{*p} \oplus V^q]^{\GL(V)}$ is generated by the contractions
\begin{equation}
    \label{f_ij for GL_k}
    f_{ij}(v^*_1, \ldots, v^*_p, v_1, \ldots, v_q) \coloneqq v^*_i(v_j), \qquad 1 \leq i \leq p, \: 1 \leq j \leq q.
\end{equation}
Note that each $f_{ij}$ has degree 2.
Now let $\C[z_{ij}] \coloneqq \C[z_{ij} \mid 1 \leq i \leq p, \: 1 \leq j \leq q]$, and define the algebra homomorphism
\begin{align*}
\pi^*: \C[z_{ij}] &\longrightarrow \C[V^{*p} \oplus V^q]^{\GL(V)},\\
z_{ij} &\longmapsto f_{ij}.
\end{align*}
The SFT states that $\ker \pi^*$ is the ideal generated by the determinants of the $(k+1)$-minors of the matrix of indeterminates $[z_{ij}]$.  
It follows that $\C[z_{ij}] / \ker \pi^* \cong \C[\M_{p,q}^{\leqslant k}]$, namely the coordinate ring of the \emph{determinantal variety} $\M_{p,q}^{\leqslant k} \coloneqq \{M \in \M_{p,q}(\C) \mid \operatorname{rank}M \leq k\}$.

\subsection{The special linear group}

Again let $\dim V = k$.
The \emph{special linear group} $\SL(V) \subset \GL(V)$ is the group of all linear operators on $V$ with determinant 1.
In order to state the fundamental theorems for $\SL(V)$, we introduce the poset
\begin{equation}
    \label{C^n_leq k definition}
    \Cleq{n} \coloneqq \{\text{subsequences $A \subseteq (1, \ldots, n)$ such that $0 \leq \#A \leq k$}\},
\end{equation}
where
\[
(a_1, \ldots, a_s) \leq (b_1, \ldots, b_t) \iff s \geq t \text{ and } a_i \leq b_i \text{ for all }1 \leq i \leq s.
\]
Equivalently, $\Cleq{n}$ can be regarded as the set of all possible columns in an SSYT with entries in $[n]$ and with at most $k$ rows; we have $A \leq B$ if and only if the column $A$ can appear to the left of the column $B$.  
We also define the subposet of full columns, namely
\begin{equation}
    \label{C^n_k definition}
    \Ceq{n} \coloneqq \left\{ I \in \Cleq{n} \text{ such that }\#I = k\right\}.
\end{equation}
Throughout the paper, we will use $A$, $B$, $X$, and $Y$ for elements of $\Cleq{n}$, while we use $I$ and $J$ for elements of $\Ceq{n}$.

The FFT states that $\C[V^{*p} \oplus V^q]^{\SL(V)}$ is generated by the following elements:
\begin{alignat}{3}
    f_{ij} &\text{ as in }\eqref{f_ij for GL_k}, && \nonumber \\
    {\rm det}_{I^*}(v^*_1, \ldots, v^*_p, v_1, \ldots, v_q) & \coloneqq \det(v^*_{i_1}, \ldots, v^*_{i_k}), &&\qquad I = (i_1, \ldots, i_k) \in \Ceq{p}, \label{SL generators}\\
    {\rm det}_J(v^*_1, \ldots, v^*_p, v_1, \ldots, v_q) &\coloneqq \det(v_{j_1}, \ldots, v_{j_k}), &&\qquad J = (j_1, \ldots, j_k) \in \Ceq{q}. \nonumber
\end{alignat}
Note that the generators $\det_{I^*}$ exist only if $k \leq p$, and the $\det_J$ exist only if $k \leq q$.

Now consider $\C[z_{ij}, u_{_{\!I}}, \upsilon_{_{\!J}} \mid I \in \Ceq{p}, \: J \in \Ceq{q}]$, and define the algebra homomorphism
\begin{align*}
    \psi^* : \C[z_{ij}, u_{_{\!I}}, \upsilon_{_{\!J}}] & \longrightarrow \C[V^{*p} \oplus V^q]^{\SL(V)},\\
    z_{ij} & \longmapsto f_{ij},\\
    u_{_{\!I}} & \longmapsto {\rm det}_{I^*},\\
    \upsilon_{_{\!J}} & \longmapsto {\rm det}_J.
\end{align*}
The relations among the $f_{ij}$ are the same as in the $\GL(V)$ case.  
As for the complete SFT, Weyl~\cite{Weyl}*{p.~70} lists the relations involving the other generators, which are somewhat more complicated.

We observe that for $h \in \GL(V)$, we have $h \cdot \det_{I^*} = (\det h) \det_{I^*}$ and $h \cdot \det_J = (\det h)^{-1} \det_J$.  
This leads us to view the $\SL(V)$-invariants as semiinvariants for $\GL(V)$, as follows.  
Given a character $\chi: H \rightarrow \C^\times$, we define the \emph{module of semiinvariants} with respect to $\chi$:
\[
\C[W]^{H, \chi} \coloneqq \{ f \in \C[W] \mid f(hw) = \chi(h)  f(w) \text{ for all } h \in H, \: w \in W\}.
\]
Taking $H = \GL(V)$ and $\chi = \det^m$ for each $m \in \mathbb{Z}$, we have 
\begin{equation}
    \label{decomp SL invariants}
    \C[V^{*p} \oplus V^q]^{\SL(V)} \cong \bigoplus_{m \in \mathbb{Z}} \C[V^{*p} \oplus V^q]^{\GL(V), \: \det^m}.
\end{equation}
as modules over $\C[V^{*p} \oplus V^q]^{\SL(V)}$.  

\subsection{The orthogonal group}

Let $\dim V = k$, and let $b( \;, \;)$ be a nondegenerate symmetric bilinear form on $V$.
Then the \emph{orthogonal group} $\O(V) = \O(V,b)$ is the group of all invertible linear operators on $V$ that preserve $b$.
That is to say, $b(hu,hv) = b(u,v)$ for all $u,v \in V$ and $h \in \O(V)$.
The FFT states that $\C[V^{n}]^{\O(V)}$ is generated by the contractions
\begin{equation}
    \label{f_ij for O_k}
    f_{ij}(v_1, \ldots, v_n) \coloneqq b(v_i, v_j), \qquad 1 \leq i \leq j \leq n.
\end{equation}
Now let $\C[z_{ij}] = \C[z_{ij} \mid 1 \leq i \leq j \leq n]$, and define the algebra homomorphism
\begin{align*}
\pi^*: \C[z_{ij}] &\longrightarrow \C[V^{n}]^{\O(V)},\\
z_{ij} &\longmapsto f_{ij}.
\end{align*}
The SFT states that $\ker \pi^*$ is the ideal generated by the determinants of the $(k+1)$-minors of the matrix of indeterminates $[z_{ij}]$.  
It follows that $\C[z_{ij}] / \ker \pi^* \cong \C[\SM_n^{\leqslant k}]$, where $\SM_n^{\leqslant k} \coloneqq \{ M \in \SM_n(\C) \mid \operatorname{rank}M \leq k\}$ is the determinantal variety of symmetric matrices.

\subsection{The special orthogonal group}

Let $\dim V = k$.
The \emph{special orthogonal group} $\SO(V)$ is the subgroup of $\O(V)$ whose elements have determinant 1.
The FFT states that $\C[V^{n}]^{\SO(V)}$ is generated by the functions
\begin{align}
    \label{SO_k generators}
    \begin{split}
        f_{ij} &\text{ as in }\eqref{f_ij for O_k},\\
        {\rm det}_I(v_1, \ldots, v_n) & \coloneqq \det(v_{i_1}, \ldots, v_{i_k}), \qquad I = (i_1, \ldots, i_k) \in \Ceq{n}.
    \end{split}
\end{align}
Note that the generators $\det_I$ exist only if $k \leq n$.  

Now consider $\C[z_{ij}, u_{_{\!I}} \mid I \in \Ceq{n}]$, and define the algebra homomorphism
\begin{align}
\label{SO_k phi}
\begin{split}
    \psi^* : \C[z_{ij}, u_{_{\!I}}] & \longrightarrow \C[V^{n}]^{\SO(V)}, \\
    z_{ij} & \longmapsto f_{ij},\\
    u_{_{\!I}} & \longmapsto {\rm det}_I.
\end{split}
\end{align}
The relations among the $f_{ij}$ are the same as in the $\O(V)$ case.  
As for the complete SFT, the relations (and even an explicit Gr\"obner basis) are recorded in \cite{Domokos}*{Thm.~2.1}.

Just as we did for $\SL(V) \subset \GL(V)$, we can view the $\SO(V)$-invariants as semiinvariants for $\O(V)$.  
In this case the decomposition is much simpler:
\begin{equation}
    \label{decomp SO_k invariants}
    \C[V^{n}]^{\SO(V)} \cong \C[V^{n}]^{\O(V)} \oplus \C[V^{n}]^{\O(V), \: \det}.
\end{equation}
(This is the decomposition $R = R^0 + R^1$ in \cite{Domokos}, expressed in the language of semiinvariants.)

\subsection{The symplectic group}

Let $\dim V = 2k$, and let $\omega(\;,\;)$ be a nondegenerate skew-symmetric bilinear form on $V$. 
Then the \emph{symplectic group} $\Sp(V) = \Sp(V,\omega)$ is the group of all invertible linear operators on $V$ that preserve $\omega$.
That is to say, $\omega(hu,hv) = \omega(u,v)$ for all $u,v \in V$ and $h \in \Sp(V)$.
The FFT states that $\C[V^{n}]^{\Sp(V)}$ is generated by the contractions
\begin{equation}
    \label{f_ij for Sp_2k}
    f_{ij}(v_1, \ldots, v_n) \coloneqq \omega(v_i, v_j), \qquad 1 \leq i < j \leq n.
\end{equation}
Now let $\C[z_{ij}] = \C[z_{ij} \mid 1 \leq i < j \leq n]$, and define the algebra homomorphism
\begin{align*}
\pi^* : \C[z_{ij}] &\longrightarrow \C[V^{n}]^{\Sp(V)},\\
z_{ij} &\longmapsto f_{ij}.
\end{align*}
The SFT states that $\ker \pi^*$ is the ideal generated by the $2(k+1)$-Pfaffians of the matrix of indeterminates $[z_{ij}]$.  
It follows that $\C[z_{ij}] / \ker \pi^* \cong \C[\AM_n^{\leqslant 2k}]$, where $\AM_n^{\leqslant 2k} \coloneqq \{ M \in \AM_n(\C) \mid \operatorname{rank}M \leq 2k\}$ is the determinantal variety of alternating  matrices.

\begin{rem}
    We have defined the classical groups above in a coordinate-free manner.
    In concrete examples where $k$ assumes a specific value, we will use the usual subscript notation $\GL_k$, $\SL_k$, $\O_k$, $\SO_k$, or $\Sp_{2k}$.
\end{rem}

\section{Linear bases for classical invariants and semiinvariants}
\label{sec:linear bases}

For each classical group, we now exhibit a linear basis for the algebra of invariants, in which basis elements are depicted as graphs.  
The total degree of each graph encodes the actual degree of the corresponding invariant polynomial function.  
(In forthcoming work, we will refine this by showing that the degree sequence yields the weight under the action of $\g$ by differential operators, as described in Section~\ref{sec:Howe duality and Enright reduction}.)
Each graph can be interpreted either as a \emph{standard monomial} (to be defined below) or as an ordinary monomial in the contractions $f_{ij}$ described above, and the RSK correspondence gives the transformation between these two monomial bases.  
For $\O(V)$ and $\Sp(V)$, the passage from the graphs to the standard monomials is essentially an application of the Burge correspondence \cite{Burge}.

\subsection{The general linear group}

\label{sub:graphs GL_k}
For $A = (a_1, \ldots, a_t) \in \Cleq{p}$ and $B = (b_1, \ldots, b_t) \in  \Cleq{q}$, let $\det(A,B)$ be the determinant of the $t$-minor in $[z_{ij}]$ determined by the row indices $a_i$ and the column indices $b_i$.  
Then for $A_1 \leq \cdots \leq A_\ell \in \Cleq{p}$ and $B_1 \leq \cdots \leq B_\ell \in \Cleq{q}$, with each $\#A_i = \#B_i$, we define the \emph{standard monomial}
\begin{equation}
    \label{def:standard monomial GL_k}
    \s(A_1, \ldots, A_\ell \mid B_1, \ldots, B_\ell) \coloneqq \prod_{i=1}^\ell \det(A_i,B_i) \in \C[z_{ij}].
\end{equation}
We call $\#A_1$ the \emph{width} of the standard monomial above; clearly this width is at most $k$.  
The key fact \cite{Sturmfels}*{Prop.~3} is that the standard monomials form a linear basis for $\C[z_{ij}]/\ker \pi^* \cong \C[\M_{p,q}^{\leqslant k}]$.  
As is typical in the literature, we abuse notation slightly by identifying standard monomials with their images in the quotient ring.
We will also call their images (under $\pi^*$) \emph{standard monomials} in $\C[V^{*p} \oplus V^q]^{\GL(V)}$.
By regarding the $A_i$ and $B_i$ as the columns of two SSYT's, we have a natural bijection
\begin{equation}
    \label{GL_k bijection SSYT}
    \{\text{standard monomials of width $\leq k$}\} \longleftrightarrow 
\bigcup_{\mathclap{\mu \in \Par(k \times \infty)}} \; \SSYT(\mu,p) \times \SSYT(\mu,q).
\end{equation}
The elements on the right-hand side --- and by some authors, also those on the left --- are called \emph{bitableaux} in the literature, where they are typically depicted with the $A_i$ and $B_i$ as rows rather than columns; see the original treatment in~\cite{DRS}.
For the sake of comparison,
\[
\begin{array}{r|l}
    \longleftarrow A_1 & B_1 \longrightarrow \\
    \vdots & \vdots \\
    \longleftarrow A_\ell & B_\ell \longrightarrow
\end{array}
\text{ is the bitableau corresponding to the tableau pair }
(A_1 \cdots A_\ell, \: B_1 \cdots B_\ell),
\]
where the arrows indicate the direction in which the entries increase.  
For example,
\[
\begin{array}{r|l}
    65421 & 12345 \\
    76432 & 12345 \\
    7544 & 2345\\
    65 & 35\\
    5 & 3
\end{array}
\text{ corresponds to }
\left(
\ytableausetup{centertableaux,smalltableaux}
\ytableaushort{12455,2346,445,567,67}_{\textstyle{,}} \: \ytableaushort{11233,2235,334,445,55}
\right).
\]

When we apply $\RSK_{\GL}$ to the pair $(A,B)$ of tableaux whose columns are the $A_i$ and $B_i$, the result is a matrix in $\M_{p,q}(\mathbb{N})$ whose entries sum to the size (i.e., number of boxes) $|A| = |B|$ of either tableau, and whose support in $\P_{\GL}$ has width $\#A_1 = \#B_1$.  
By viewing this matrix as the degree matrix of a monomial in the $z_{ij}$, we arrive at an ordinary monomial in $\C[z_{ij}]$ which we write as $\m(A_1, \ldots, A_\ell \mid B_1, \ldots, B_\ell)$.
We will be mainly concerned with the images under $\pi^*$, and so we set the shorthand
\begin{equation}
\label{f_AB}
f_{AB} \coloneqq \pi^*(\m(A_1, \ldots, A_\ell \mid B_1, \ldots, B_\ell)) = \prod_{\substack{(i,j) \in \msupp \RSK_{\GL}(A,B)}} f_{ij}.
\end{equation}
Since $\pi^*: z_{ij} \mapsto f_{ij}$ doubles the degree, $f_{AB}$ has degree $2|A| = 2|B|$.
Following \cite{Sturmfels}, we say that $\#A_1 = \#B_1$ is the \emph{width} of $f_{AB}$.  
In summary, we have a bijective correspondence
\begin{equation}
    \label{RSK_GL s to m}
\pi^*(\s(A_1, \ldots, A_\ell \mid B_1, \ldots, B_\ell)) \xmapsto{\RSK_{\GL}} f_{AB},
\end{equation}
preserving both degree and width.
The set of all monomials on either side of~\eqref{RSK_GL s to m} furnishes a basis for $\C[V^{*p} \oplus V^q]^{\GL(V)}$.

We will depict each of these basis monomials $f_{AB}$ as a graph --- more specifically, as a \emph{labeled $(p,q)$-bipartite arc diagram} --- by viewing its degree matrix as a biadjacency matrix.  
In particular, arrange the vertices $1^*, \ldots, p^*$ (the ``starred part'' of the graph) in a horizontal line, followed by the vertices $1, \ldots, q$ (the ``unstarred part'').  
See Figure~\ref{fig:GL_k graph example}, where we draw a vertical line to separate the two parts.  
Given a monomial $f_{AB}$, each arc $(i^*, j)$ represents a factor $f_{ij}$ in $f_{AB}$, including multiplicities.

Since each $f_{ij}$ has degree 2, the total degree of a graph (i.e., twice the number of arcs) encodes the polynomial degree of the associated invariant monomial $f_{AB}$.
Our graphs also nicely encode the width of the associated monomial via nested arcs.  Note that a maximal antichain in $\P_{\GL}$ corresponds to a maximal family of pairwise nested arcs, which we will call a \emph{strict nesting}.  
(By ``strict'' we mean that no two nested arcs share a vertex). 
Hence the width of $f_{AB}$ equals the number of arcs in the largest strict nesting in the graph of $f_{AB}$.

Recall that a graph is \emph{1-regular} if every vertex has degree 1.
We write $\C[W]_d$ to denote the graded component consisting of homogeneous polynomials of degree $d$.
Let $\SYT(\mu)$ denote the set of standard Young tableaux whose shape is given by the partition $\mu$. 
As usual, $\mu \vdash p$ means that $\mu$ is a partition of $p$, and $\ell(\mu)$ denotes the number of parts in $\mu$.

 \begin{prop}[Polynomial invariants for $\GL(V)$]
    \label{prop:GL_k graphs}
    Let $\dim V = k$.
    A basis for $\C[V^{*p} \oplus V^q]^{\GL(V)}_d$ is given by the set of all $(p,q)$-bipartite arc diagrams with total degree $d$, such that no strict nesting contains more than $k$ arcs.
    In particular, we interpret each such arc diagram as the product of contractions in which $f_{ij}$ appears once for each arc $(i^*, j)$.
 \end{prop}

\begin{prop}[Tensor invariants for $\GL(V)$]
    \label{prop:GL_k tensor invariants}
    The space $(V^{\otimes p} \otimes V^*{}^{\otimes q})^{\GL(V)}$ is nonzero if and only if $p=q$.
    In this case, a basis for $(V^{\otimes p} \otimes V^{*\otimes p})^{\GL(V)}$ is given by the set of all 1-regular $(p,p)$-bipartite arc diagrams, such that no strict nesting contains more than $k$ arcs.
    In particular, identifying $V^{\otimes p} \otimes V^{*\otimes p}$ with the space of multilinear forms on $V^{*p} \oplus V^p$, we interpret each arc $(i^*, j)$ as the contraction $f_{ij}$.
\end{prop}

\begin{proof}
    Let $\mathbf{d} = (d_1, \ldots, d_p) \in \mathbb{N}^p$ and $\mathbf{e} = (e_1, \ldots, e_q) \in \mathbb{N}^q$. Write $\mathbf{1} \coloneqq (1, \ldots, 1)$.
    Let $\C[V^{*p} \oplus V^q]_{\mathbf{d},\mathbf{e}}$ denote the multigraded component consisting of polynomial functions which are degree $d_i$ in $v^*_i$ and degree $e_j$ in $v_j$, for all $i,j$.
    (See~\cite{GW}*{p.~256}.)
    This multigraded decomposition is $\GL(V)$-invariant.
    A basis for $\C[V^{*p} \oplus V^q]^{\GL(V)}_{\mathbf{d},\mathbf{e}}$ is given by the set of arc diagrams in Proposition~\ref{prop:GL_k graphs} such that each vertex $i^*$ has degree $d_i$ and each vertex $j$ has degree $e_j$.
    Now, there is a canonical linear isomorphism
    \begin{align*}
    V^{\otimes p} \otimes V^*{}^{\otimes q} &\cong \C[V^{*p} \oplus V^q]_{\mathbf{1}, \mathbf{1}}
    \end{align*}
    given by contractions between vectors and covectors in the natural way; therefore a basis for $(V^{\otimes p} \otimes V^*{}^{\otimes q})^{\GL(V)}$ is given by the set of 1-regular graphs in Proposition~\ref{prop:GL_k graphs}.  
    Clearly there are no such graphs unless $p=q$.
\end{proof}

\begin{corollary}
    We have $\displaystyle\dim {\rm End}_{\GL(V)}(V^{\otimes p}) = \sum_{\substack{\mu \vdash p, \\ \ell(\mu) \leq k}} \#\SYT(\mu)^2$.
\end{corollary}

\begin{proof}
    It follows from the construction of $\RSK_{\GL}$ that $T$ contains one copy of $i$ and $U$ contains one copy of $j$ for each occurrence of $(i,j)$ in the multisupport of $\RSK_{\GL}(T,U)$.  
    The biadjacency matrix of each basis graph in Proposition~\ref{prop:GL_k tensor invariants} is a $p \times p$ permutation matrix whose support has width $\leq k$, and hence each of its corresponding tableaux contains each entry $1, \ldots, p$ exactly once; thus they are both SYT's.
    The result follows from Proposition~\ref{prop:RSK_GL}.
\end{proof}

\begin{ex}
    If $k=2$, then $\dim {\rm End}_{\GL(V)} (V^{\otimes p})$ equals the $p$th Catalan number $\frac{1}{p+1}\binom{2p}{p}$.
\end{ex}

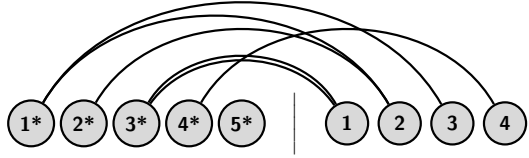
\begin{figure}[t]
    \centering
    \scalebox{0.7}{
\begin{tikzpicture}[-,auto,node distance=1cm,
  very thick,plain node/.style={minimum size=0.8cm,circle,draw,font=\sffamily\bfseries, fill=gray!30}, bend left = 60]

\node[plain node] (1*) {1*};
\node[plain node] (2*) [right of=1*] {2*};
\node[plain node] (3*) [right of=2*] {3*};
\node[plain node] (4*) [right of=3*] {4*};
\node[plain node] (5*) [right of=4*] {5*};
\node (divider) [right of=5*] {$\Bigg|$};
\node[plain node] (1) [right of=divider] {1};
\node[plain node] (2) [right of=1] {2};
\node[plain node] (3) [right of=2] {3};
\node[plain node] (4) [right of=3] {4};

\draw (1*) to (2);
\draw (1*) to (3);
\draw (2*) to (2);
\draw (3*) to (1);
\draw [bend left=55] (3*) to (1);
\draw (4*) to (4);

\end{tikzpicture}
}
    \caption{Example of a basis graph for the polynomial $\GL(V)$-invariants, with $p = 5$ and $q = 4$. There are $6$ edges, for a total degree of $12$. 
    Explicitly, the graph represents the product $f_{12}f_{13}f_{22}f_{31}^2 f_{44}$, with $f_{ij}$ as defined in~\eqref{f_ij for GL_k}.
    Since the largest strict nesting $\{(1^*,3), (2^*, 2), (3^*, 1)\}$ has size $3$, this graph is included in the basis of $\C[V^{*5} \oplus V^4]^{\GL(V)}_{12}$ whenever $\dim V \geq 3$.}
    \label{fig:GL_k graph example}
\end{figure}
 
\subsection{The special linear group}  
\label{sub:graphs SL_k}

For $\SL(V)$, we expand our definition of the standard monomials to include the following elements of $\C[z_{ij}, u_{_{\!I}}, \upsilon_{_{\!J}}]$:
\begin{align}
    &\s(A_1, \ldots, A_\ell \mid B_1, \ldots, B_\ell)\phantom{,} \quad \text{as in \eqref{def:standard monomial GL_k}}, \nonumber \\
    u_{_{\!I_1}} \cdots \, u_{_{\!I_m}}  &\s(A_1, \ldots, A_\ell \mid B_1, \ldots, B_\ell), \quad I_1 \leq \cdots \leq I_m \leq A_1, \label{s-monomials SL_k}\\
    \upsilon_{_{\!J_1}} \cdots \, \upsilon_{_{\!J_m}} &\s(A_1, \ldots, A_\ell \mid B_1, \ldots, B_\ell), \quad J_1 \leq \cdots \leq J_m \leq B_1, \nonumber
\end{align}
where $m$ ranges over all positive integers. 
As explained in~\cite{Lakshmibai}*{Ch.~11}, the set of standard monomials above furnishes a linear basis for the quotient $\C[z_{ij}, u_{_{\!I}}, \upsilon_{_{\!J}}] / \ker \psi^*$.
For our purposes in this paper, we give the following interpretation.
Each standard monomial can be still be identified with a unique pair of SSYT's (but not of the same shape), constructed by taking the pair $(A,B)$ determined by the $A_i$ and $B_i$, and either prepending the columns $I_1, \ldots, I_m$ to the tableau $A$, or prepending the columns $J_1, \ldots, J_m$ to the tableau $B$.
Upon passing to the $\SL(V)$-invariant ring via $\psi^*$, the $\RSK_{\GL}$ correspondence~\eqref{RSK_GL s to m} yields a bijection between standard monomials and ordinary monomials, where the latter take one of the three forms
\begin{equation}
\label{fAB three forms SL}
f_{AB} \qquad \text{or} \qquad f_{AB} \prod_{\ell = 1}^m \textstyle \det_{I_\ell^*} \qquad \text{or} \qquad \displaystyle f_{AB} \prod_{\ell=1}^m \textstyle\det_{J_\ell}.
\end{equation}
Note that the two latter forms have degree  $2|A| + m$.

Due to the presence of the determinantal factors, the graphical analogue to the $\GL(V)$-graphs must now involve not only arcs, but also \emph{hyperedges} that connect $k$ vertices rather than only two. 
We will denote these hyperedges by the $I$'s or $J$'s which label the vertices they connect; in this way, we identify a hyperedge with an element of either $\Ceq{p}$ or $\Ceq{q}$.  
Since the degree of a vertex is defined to be the number of edges incident to it, each arc still contributes 2 to the total degree, while each hyperedge contributes $k$.  
Hence the total degree equals the degree of the asociated monomial in~\eqref{fAB three forms SL}.

We depict each hyperedge as a row of $k$ dots directly beneath $k$ vertices (see Figure~\ref{fig:SL_k graph example}). 
The hyperedges must appear exclusively either in the starred part or in the unstarred part. 
 As a visual aid, we connect the dots from level to level to obtain a family of $k$ non-intersecting trails which float weakly to the left as they move downward.  
 This is a consequence of the condition in~\eqref{s-monomials SL_k} that the $I$'s or $J$'s must form a chain in $\Ceq{p}$ or $\Ceq{q}$.
 (Looking ahead to Example~\ref{ex:SL invariants k2p3q4}, we see that these families of trails will also appear as extensions of lattice paths in $\P_{\GL}$.)

It remains to specify how the topmost hyperedge is constrained by the arcs, and vice versa.  
This is governed by the conditions $I_m \leq A_1$ and $J_m \leq B_1$ in~\eqref{s-monomials SL_k}.  
Hence by part (1) of Lemma~\ref{lemma:Ci and Di for GL_k}, it suffices to understand how the chain decompositions into the $C^*_i$ and $C_i$, defined in~\eqref{C*1 and C1 for GL_k}, interact with $I_m$.  
These decompositions have an especially nice graphical interpretation: $C^*_1$ (resp., $C_1$) consists of all arcs which are not strictly nested inside (resp., outside) any other arcs, and we iterate this to obtain each successive $C^*_i$ or $C_i$.  
This leads to the following rule:
\begin{equation}
\label{SL_k painted vertex rules}
    \parbox{0.8\linewidth}{If the hyperedges are on the starred side of the graph, then the $i$th dot in the topmost hyperedge must lie weakly outside all arcs in $C^*_i$.  
    If the hyperedges are on the unstarred side, then the $i$th dot in the topmost hyperedge must lie weakly inside all arcs in $C_i$.}
    \end{equation}
For indices $i$ where $C_i$ does not exist, the rule above is fulfilled vacuously.  
As a non-example of the rule~\eqref{SL_k painted vertex rules}, the graph below is \emph{not} included in the basis for the $\SL_2$-invariants on $V^{*3} \oplus V^2$:
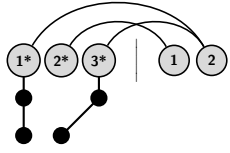
\begin{center}\scalebox{0.5}{
\begin{tikzpicture}[-,auto,node distance=1cm,
  very thick,plain node/.style={minimum size=0.8cm,circle,draw,font=\sffamily\bfseries, fill = gray!30},
  plain hypernode/.style={minimum size=0.4cm,circle,draw=gray},
  painted hypernode/.style={minimum size=0.4cm,circle,draw,fill=black},bend left = 60]

\draw[ultra thick] (2,0) \foreach \x/\y in {0/-1,-1/-1}{
-- ++(\x,\y) node[painted hypernode]{}
};
\draw[ultra thick] (0,0) \foreach \x/\y in {0/-1,0/-1}{
-- ++(\x,\y) node[painted hypernode]{}
};

\node[plain node] (1*) {1*};
\node[plain node] (2*) [right of=1*] {2*};
\node[plain node] (3*) [right of=2*] {3*};
\node (divider) [right of=3*] {$\Bigg|$};
\node[plain node] (1) [right of=divider] {1};
\node[plain node] (2) [right of=1] {2};

\draw (1*) to (2);
\draw (2*) to (1);
\draw (3*) to (2);

\end{tikzpicture}
}
\end{center}
This is because the $2$nd dot in the topmost hyperedge does not lie weakly outside $C^*_2 = \{(2^*, 1)\}$.

We summarize the characterization of the $\SL(V)$-graphs in the following propositions.  
The proof of Proposition~\ref{prop:SL_k tensor invariants} is the same as that of Proposition~\ref{prop:GL_k tensor invariants}, given the fact that $\det_{I^*}$ (resp., $\det_J$) is linear in each index in $I$ (resp., $J$).

\begin{prop}[Polynomial invariants for $\SL(V)$] Let $\dim V = k$.
\label{prop:SL_k graphs}

\begin{enumerate}
    \item A basis for $\C[V^{*p} \oplus V^q]^{\SL(V)}_d$ is given by graphs of total degree $d$, obtained from those in Proposition~\ref{prop:GL_k graphs} by adjoining pairwise comparable hyperedges of size $k$ to either the starred or unstarred side of the graph, such that~\eqref{SL_k painted vertex rules} is satisfied. 

    \item The corresponding $\SL(V)$-invariant is the product of all arcs and hyperedges, where each arc $(i^*, j)$ contributes the contraction $f_{ij}$, each hyperedge $I$ on the starred vertices contributes $\det_{I^*}$, and each hyperedge $J$ on the unstarred vertices contributes $\det_J$, as defined in~\eqref{SL generators}.
    
    \item The graphs with $m$ hyperedges on the starred side furnish a basis for $\C[V^{*p} \oplus V^q]^{\GL(V), \: \det^m}$.
 
    \item The graphs with $m$ hyperedges on the unstarred side furnish a basis for $\C[V^{*p} \oplus V^q]^{\GL(V), \: \det^{-m}}$.
 \end{enumerate}
\end{prop}

\begin{prop}[Tensor invariants for $\SL(V)$]
  \label{prop:SL_k tensor invariants}

Let $\dim V = k$.
The space $(V^{\otimes p} \otimes V^*{}^{\otimes q})^{\SL(V)}$ is nonzero if and only if $k \mid (p -q)$.  
In this case, a basis is given by the 1-regular graphs in Proposition~\ref{prop:GL_k tensor invariants}.
In particular, each basis graph contains exactly $\min\{p,q\}$ arcs, and exactly $\frac{|p-q|}{k}$ hyperedges on one side, such that the rule~\eqref{SL_k painted vertex rules} is satisfied, and each vertex has degree 1.  
    
\end{prop}

\begin{corollary}
\label{cor:SL_k dim tensor invariants}
Assuming that $|p-q| = mk$, we have 
\[
\dim (V^{\otimes p} \otimes V^*{}^{\otimes q})^{\SL(V)} = \sum_{\mu} \#\SYT(\mu + m^k) \cdot \#\SYT(\mu),
\]
where $\mu \vdash \min\{p,q\}$ with $\ell(\mu) \leq k$, and where $\mu + m^k$ is the shape obtained from $\mu$ by adding $m$ boxes to all $k$ rows.
\end{corollary}

\begin{proof}
     
   Without loss of generality, assume that $p = q+mk$.
   Then there are exactly $m$ disjoint hyperedges $I_1, \ldots, I_m$ on the starred side of the graph.
   Hence the graph corresponds to some monomial $f_{AB}\det_{I_1^*} \cdots \det_{I_m^*}$.
   Consider the tableau obtained by prepending the columns $I_1, \ldots, I_m$ to the tableau $A$.
   By the 1-regularity of the graph, each entry appears exactly once in this tableau, as well as in $B$.
   Hence the SSYT's are actually SYT's, and this procedure is clearly invertible.
\end{proof}

\begin{ex}[$\SL_2$-invariants and Catalan numbers]
    Recall that one of the many descriptions for the Catalan numbers $C_n = \frac{1}{n+1} \binom{2n}{n}$ is the fact that $C_n = \#\SYT(n,n)$.
    In the special case $k\coloneqq \dim V = 2$, we have $V \cong V^*$, and so Corollary~\ref{cor:SL_k dim tensor invariants} leads to families of identities for the Catalan numbers in terms of SYT's.
    For example, let $n=4$, so that $p = 2n = 8$, and consider the $\SL_2$-invariants in
    \[
    V^{\otimes 8} \cong V^{\otimes 7} \otimes V^* \cong V^{\otimes 6} \otimes V^{*\otimes 2} \cong V^{\otimes 5} \otimes V^{*\otimes 3} \cong V^{\otimes 4} \otimes V^{*\otimes 4}.
    \]
    By Corollary~\ref{cor:SL_k dim tensor invariants}, the dimension of the space of invariants equals
    \ytableausetup{boxsize=.6em}
    \begin{align*}
    C_4 &= \#\SYT\left(\ydiagram{4,4}\right)\#\SYT(\varnothing) = 14 \cdot 1\\
    &= \#\SYT\left(\ydiagram{4,3}\right) \#\SYT\left(\ydiagram{1}\right) = 14 \cdot 1\\
    &= \#\SYT\left(\ydiagram{4,2}\right) \#\SYT(\ydiagram{2}) + \#\SYT\left(\ydiagram{3,3}\right) \#\SYT\left(\ydiagram{1,1}\right) = 9 \cdot 1 + 5 \cdot 1\\
    &= \#\SYT\left(\ydiagram{4,1}\right) \#\SYT(\ydiagram{3}) + \#\SYT\left(\ydiagram{3,2}\right) \#\SYT\left(\ydiagram{2,1}\right) = 4 \cdot 1 + 5 \cdot 2\\
     &= \#\SYT\left(\ydiagram{4}\right)^2 + \#\SYT\left(\ydiagram{3,1}\right)^2 + \#\SYT\left(\ydiagram{2,2}\right)^2 = 1^2 + 3^2 + 2^2\\
     &= 14.
    \end{align*}
    For $k > 2$, rather than a family of identities, we have only $\dim(V^{\otimes nk})^{\SL_k} = \#\SYT(n^k)$ for $n = 1, 2, \ldots$.  
    These are the ``$k$-dimensional Catalan numbers'' observed in~\cite{BostanEtAl}.
\end{ex}

\begin{ex}[Tensor invariants for $\SL_4$]

As another example of Corollary~\ref{cor:SL_k dim tensor invariants}, we take $k=4$, with $p = 11$ and $q = 3$.
Since $(p-q)/k = 2$ is an integer, the space of invariants is nonzero, and we can count its dimension by counting SYT's of size 3, paired with SYT's obtained by prepending $m=2$ columns of length $k=4$.
(Since $\min\{p,q\} \leq k$, the condition $\ell(\mu) \leq k$ is automatic.)
We conclude that the dimension of $(V^{\otimes 11} \otimes V^{*\otimes 3})^{\SL_4}$ equals
\begin{align*}
    &\phantom{=} \#\SYT\left(\ydiagram{5,2,2,2}\right)\#\SYT\left(\ydiagram{3}\right) + \#\SYT\left(\ydiagram{4,3,2,2}\right)\#\SYT\left(\ydiagram{2,1}\right) + \#\SYT\left(\ydiagram{3,3,3,2}\right)\#\SYT\left(\ydiagram{1,1,1}\right)\\
    & = 825 \cdot 1 + 1320 \cdot 2 + 462 \cdot 1 = 3927.
\end{align*}
\ytableausetup{boxsize=normal}
\end{ex}

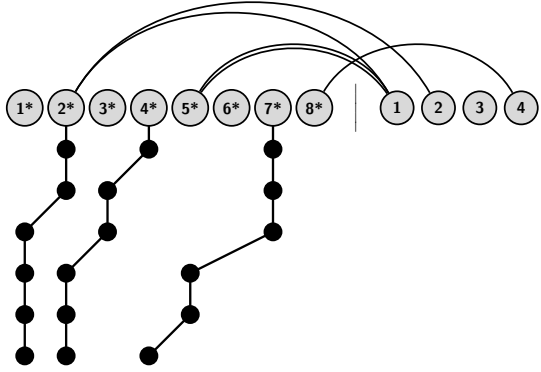
\begin{figure}[h]
    \centering
    \scalebox{.55}{
\begin{tikzpicture}[-,auto,node distance=1cm,
  very thick,plain node/.style={minimum size=0.8cm,circle,draw,font=\sffamily\bfseries, fill = gray!30}, painted node/.style={minimum size=0.8cm,circle,draw,font=\sffamily\bfseries,fill=black, text=white}, 
  plain hypernode/.style={minimum size=0.4cm,circle,draw=gray},
  painted hypernode/.style={minimum size=0.4cm,circle,draw,fill=black},bend left = 60]

\draw[ultra thick] (6,0) \foreach \x/\y in {0/-1,0/-1,0/-1,-2/-1,0/-1,-1/-1}{
-- ++(\x,\y) node[painted hypernode]{}
};
\draw[ultra thick] (3,0) \foreach \x/\y in {0/-1,-1/-1,0/-1,-1/-1,0/-1,0/-1}{
-- ++(\x,\y) node[painted hypernode]{}
};
\draw[ultra thick] (1,0) \foreach \x/\y in {0/-1,0/-1,-1/-1,0/-1,0/-1,0/-1}{
-- ++(\x,\y) node[painted hypernode]{}
};

\node[plain node] (1*) {1*};
\node[plain node] (2*) [right of=1*] {2*};
\node[plain node] (3*) [right of=2*] {3*};
\node[plain node] (4*) [right of=3*] {4*};
\node[plain node] (5*) [right of=4*] {5*};
\node[plain node] (6*) [right of=5*] {6*};
\node[plain node] (7*) [right of=6*] {7*};
\node[plain node] (8*) [right of=7*] {8*};
\node (divider) [right of=8*] {$\Bigg|$};
\node[plain node] (1) [right of=divider] {1};
\node[plain node] (2) [right of=1] {2};
\node[plain node] (3) [right of=2] {3};
\node[plain node] (4) [right of=3] {4};

\draw (2*) to (1);
\draw (2*) to (2);
\draw (5*) to (1);
\draw [bend left=55] (5*) to (1);
\draw (8*) to (4);

\end{tikzpicture}
}
    \caption{Example of a basis graph for the polynomial $\SL_3$-invariants, where $p = 8$ and $q = 4$. 
    Each arc $(i^*, j)$ represents the contraction $f_{ij}$, defined in~\eqref{f_ij for GL_k}.  
    The three painted circles in each row represent a hyperedge connecting the three vertices directly above them. 
    (The fact that $k=3$ is therefore immediate from the number of painted circles in each row.)  
    Each hyperedge $I = (i^*_1, i^*_2, i^*_3)$ represents the determinantal factor $\det_{I^*}$.  
    The chain decomposition is given by $C^*_1 = \{(2^*, 1), (2^*, 2), (8^*, 4)\}$ and $C^*_2 = \{(5^*, 1), (5^*, 1)\}$; therefore the corresponding monomial has width $2$, which is also the size of the largest strict nesting. 
    Note that the rule~\eqref{SL_k painted vertex rules} is satisfied: $2^*$ lies weakly outside all of the arcs in $C^*_1$, while $4^*$ lies outside all of the arcs in $C^*_2$. 
    Since $C^*_3$ does not exist, the third painted vertex could have been any of $5^*, \ldots, 8^*$.  
    The total degree of the graph equals twice the number of arcs, plus the number of painted circles, which is $2(5) + 18 = 28$; this is also the degree of the corresponding monomial, which is the product of all arcs and hyperedges. 
    Since there are 6 hyperedges (i.e., the length of the trails), this function lies in the semiinvariant component $\C[W]^{\GL_3, \: \det^{6}}_{28}$.}
    \label{fig:SL_k graph example}
\end{figure}

\subsection{The orthogonal group} 
\label{sec:graphs O}

For $A = (a_1, \ldots, a_t) \leq B = (b_1, \ldots, b_t) \in \Cleq{n}$, let $\det(A,B)$ denote the determinant of the $t$-minor in the symmetric matrix $[z_{ij}]$ determined by the row indices $a_i$ and column indices $b_i$.  
Then for $A_1 \leq B_1 \leq \cdots \leq A_{\ell} \leq B_{\ell} \in \Cleq{n}$, with each $\#A_i = \#B_i$, we define the standard monomial
\begin{equation}
    \label{def:standard monomial O_k}
    \s(A_1, B_1, \ldots, A_\ell, B_\ell) \coloneqq \prod_{i=1}^\ell \det(A_i, B_i).
\end{equation}
The set of standard monomials furnishes a basis for $\C[z_{ij}]/\ker \pi^*$; see~\cite{Conca94}*{Thm.~5.1} or~\cite{Procesi}*{\S8.3}.  
Again we call $\#A_1$ the \emph{width} of the standard monomial.  
Note that each standard monomial can be viewed as the even-rowed SSYT whose columns are the $A_i$ and $B_i$, interlaced as above.  
This gives a natural bijection
\begin{equation}
    \label{O_k bijection SSYT}
    \{\text{standard monomials of width $\leq k$}\} \longleftrightarrow 
\bigcup_{\mathclap{\substack{\mu \in \Par(k \times \infty)\\ \text{with even row lengths}}}} \; \SSYT(\mu,n).
\end{equation}
(The elements on either side are sometimes called $d$-tableaux in the literature, and are displayed with the $A_i$ and $B_i$ as interlacing rows rather than as columns.)  
We define $\m(A_1, B_1, \ldots, A_\ell, B_\ell) \in \C[z_{ij}]$ to be the ordinary monomial whose degree matrix equals $\RSK_{\O}(A_1, B_1, \ldots, A_\ell, B_\ell)$.  
By Proposition~\ref{prop:RSK_O}, we thus have the analogue of the degree- and width-preserving bijection~\eqref{RSK_GL s to m} between standard and ordinary monomials.
The image of either set of monomials under $\pi^*$ likewise is a basis for $\C[V^n]^{\O(V)}$.

Let $f_{AB} \coloneqq \pi^*(\m(A_1, B_1, \ldots, A_\ell, B_\ell))$.
Then $f_{AB}$ can be represented by an arc diagram, with loops and multiple arcs allowed, on the vertices $1, \ldots, n$; see \cite{Burge}*{Fig.~2}.
Viewed as an $\O(V)$-invariant, each graph is the product of its arcs, where an arc $(i,j)$ represents $f_{ij}$; Figure~\ref{subfig:O_k graph example} shows an example.
This time, a chain (rather than an antichain) in $\P_{\O}$ corresponds to a \emph{weak} nesting (i.e., two nested arcs may share one or both vertices).  
The width of $f_{AB}$ therefore equals the minimum number of weak nestings into which the arcs can be decomposed.  
One way to determine this number is to use the graphical equivalent of the chains $C_i$ defined in~\eqref{C1 for O_k}.  
Note that the term ``leftmost arc'' is well-defined (up to multiple arcs connecting the same two vertices, which are interchangeable): in a properly drawn arc diagram, the leftmost arc is the arc that extends furthest left, and (in case of a tie) is the longest of all such arcs.  
We can construct $C_i$ by repeatedly adding the leftmost arc which is weakly nested inside its predecessor; then we delete $C_i$ and construct $C_{i+1}$ from the remaining graph, until the arcs are exhausted.  
We arrive at the following characterizations:

 \begin{prop}[Polynomial invariants for $\O(V)$]
     \label{prop:O_k graphs}
     Let $\dim V = k$.
     A basis for $\C[V^{n}]^{\O(V)}_d$ is given by the set of all arc diagrams (with multiple edges and loops) on $n$ vertices, with total degree $d$, whose arcs can be decomposed into $\leq k$ weak nestings.  
     In particular, we interpret each such arc diagram as the product in which $f_{ij}$ appears once for each arc $(i,j)$.
 \end{prop}

 \begin{prop}[Tensor invariants for $\O(V)$]
     \label{prop:O_k tensor invariants}
    The space $(V^{\otimes n})^{\O(V)}$ is nonzero if and only if $n$ is even.
     In this case, a basis for $(V^{\otimes n})^{\O(V)}$ is given by the set of 1-regular arc diagrams on $n$ vertices, whose arcs can be decomposed into $\leq k$ weak nestings.
     In particular, identifying $V^{\otimes n}$ as the space of multilinear forms on $V^{n}$, we interpret each arc $(i,j)$ as the contraction $f_{ij}$.  
 \end{prop}

 \begin{proof}
     The argument is identical to that for $\GL(V)$ in Proposition~\ref{prop:GL_k tensor invariants}, upon identifying $V^{\otimes n}$ with the multigraded component $\C[V^{n}]_{\mathbf{1}}$.  Note that this forces $(V^{n})^{\O(V)} = 0$ when $n$ is odd.
 \end{proof}

\begin{corollary}[\cite{Smith}*{p.~654}]
    \label{cor:O_k tensor dim}
    We have $\displaystyle\dim (V^{\otimes n})^{\O(V)} = \sum_{\mathclap{\substack{\mu \vdash n, \\ \ell(\mu) \leq k, \\ \textup{even row lengths}}}} \#\SYT(\mu)$ if $n$ is even, and $0$ if $n$ is odd.
\end{corollary}

\begin{proof}
    By the construction in~\cite{Burge}*{p.~22}, a tableau $T$ contains one copy of $i$ and $j$ for each occurrence of $(i,j)$ in the multisupport of $\RSK_{\O}(T)$. 
    The adjacency matrix of each arc diagram in Proposition~\ref{prop:O_k tensor invariants} is an $n \times n$ permutation matrix with $0$'s on the diagonal, whose support has width $\leq k$.
    Taking the upper-triangular part of such a matrix, we see that its corresponding tableau contains each entry $1, \ldots, n$ exactly once, and is therefore an SYT.
    The result follows from Proposition~\ref{prop:RSK_O}.
\end{proof}

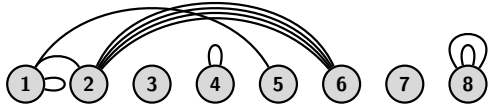
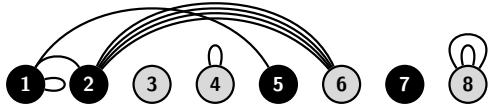
\begin{figure}[t]
    \centering
    \begin{subfigure}[b]{\textwidth}
    \centering
    \scalebox{0.7}{
\begin{tikzpicture}[-,auto,node distance=1.2cm,
  very thick,plain node/.style={circle,draw,font=\sffamily\bfseries,fill=gray!30}, painted node/.style={circle,draw,font=\sffamily\bfseries,fill=black, text-white}]
\tikzstyle{every loop}=[-, distance = 15]

\node[plain node] (1) {1};
\node[plain node] (2) [right of=1] {2};
\node[plain node] (3) [right of=2] {3};
\node[plain node] (4) [right of=3] {4};
\node[plain node] (5) [right of=4] {5};
\node[plain node] (6) [right of=5] {6};
\node[plain node] (7) [right of=6] {7};
\node[plain node] (8) [right of=7] {8};


\draw [bend left = 60] (1) to (2);
\draw [bend left = 60] (1) to (5);
\draw [bend left = 60] (2) to (6);
\draw [bend left = 55] (2) to (6);
\draw [bend left = 50] (2) to (6);
\draw [bend left = 66] (2) to (6);


\path (1) edge [loop right] node {} (1);
\path (4) edge [loop above] node {} (4);
\path (8) edge [loop above] node {} (8);
\path (8) edge [loop above, in=55, out=125,distance=30] node {} (8);

\end{tikzpicture}
}
    \caption{The graph corresponding to the polynomial $\O(V)$-invariant $(f_{11} f_{12} f_{15})(f_{26}^4 f_{44})(f_{88}^2)$, where the parentheses designate the weak nestings $C_1$, $C_2$, and $C_3$.  
    The fact that there are three nestings means that this graph is included in the basis for $\C[V^8]^{\O(V)}_{18}$ as long as $\dim V \geq 3$.}
    \label{subfig:O_k graph example}
    \end{subfigure}

    \begin{subfigure}[b]{\textwidth}
    \centering
        \scalebox{0.7}{
\begin{tikzpicture}[-,auto,node distance=1.2cm,
  very thick,plain node/.style={circle,draw,font=\sffamily\bfseries,fill=gray!30}, painted node/.style={circle,draw,font=\sffamily\bfseries,fill=black, text=white}]
\tikzstyle{every loop}=[-, distance = 15]

\node[painted node] (1) {1};
\node[painted node] (2) [right of=1] {2};
\node[plain node] (3) [right of=2] {3};
\node[plain node] (4) [right of=3] {4};
\node[painted node] (5) [right of=4] {5};
\node[plain node] (6) [right of=5] {6};
\node[painted node] (7) [right of=6] {7};
\node[plain node] (8) [right of=7] {8};


\draw [bend left = 60] (1) to (2);
\draw [bend left = 60] (1) to (5);
\draw [bend left = 60] (2) to (6);
\draw [bend left = 55] (2) to (6);
\draw [bend left = 50] (2) to (6);
\draw [bend left = 66] (2) to (6);


\path (1) edge [loop right] node {} (1);
\path (4) edge [loop above] node {} (4);
\path (8) edge [loop above] node {} (8);
\path (8) edge [loop above, in=55, out=125,distance=30] node {} (8);

\end{tikzpicture}
}
        \caption{The graph corresponding to the product of the invariant above with $\det(v_1, v_2, v_5, v_7)$.  
        This means that $\dim V=4$ in this example.  
        Note that by the rule~\eqref{SO_k painted vertex rules}, the first two painted vertices are forced to be $1$ and $2$, since $C_1$ and $C_2$ begin at those two vertices.  
        Since $C_3$ begins at vertex $8$, and since $C_4$ does not exist, the third and fourth painted vertices could have been chosen anywhere to the right of vertex $2$.  
        The degree of the $\SO_4$-invariant equals $18 + 4 = 22$.}
        \label{subfig:SO_k graph example}
    \end{subfigure}
    
    \caption{Example of a basis graph for an $\O(V)$-invariant and an $\O(V)$-semiinvariant, where $n = 8$.  
    Together, these two types of graphs furnish a basis for the $\SO(V)$-invariants.}
    \label{fig:O_k and SO_k graph example}
\end{figure}

\subsection{The special orthogonal group} 
\label{sec:graphs SO}

Just as for $\SL(V) \subset \GL(V)$, for $\SO(V) \subset \O(V)$ we expand our definition of the standard monomials to include the following elements of $\C[z_{ij}, u_{_{\!{I}}}]$:
\begin{align}
    &\s(A_1, B_1, \ldots, A_\ell, B_\ell)\phantom{,} \quad \text{as in \eqref{def:standard monomial O_k}}, \label{basis R0} \\
    u_{_{\!I}} &\s(A_1, B_1, \ldots, A_\ell, B_\ell), \quad I \leq A_1. \label{basis R1}
\end{align}
The set of standard monomials above furnishes a linear basis for the quotient $\C[z_{ij}, u_{_{\!I}}]/\ker \psi^*$, as defined in~\eqref{SO_k phi}; see \cite{DeConciniProcesi}*{Thm.~5.6(ii)} and the discussion in~\cite{Domokos}*{p.~710}. 
In light of~\eqref{decomp SO_k invariants}, the $\psi^*$-images $f_{AB}$ of the standard monomials of type~\eqref{basis R0} form a basis for the $\O(V)$-invariants, while the images $f_{AB} \det_I$ of type~\eqref{basis R1} form a basis for the $\O(V)$-semiinvariants. 
Note that each monomial of type~\eqref{basis R1} can be identified with an SSYT whose rows have odd length, and whose first column is $I$; its degree in $\C[z_{ij}, u_{_{\!I}}]$ equals $1 + |A|$, but the degree of its image under $\psi^*$ in $\C[V^{n}]^{\SO(V)}$ equals $k + 2|A|$.  

We distinguish the graphs of $f_{AB} \det_I$ from the graphs of $f_{AB}$ by painting the vertices corresponding to the indices in $I$; see Figure~\ref{subfig:SO_k graph example}.
One can view these $k$ painted vertices as a single hyperedge.
The condition $I \leq A_1$ translates to the following rule:
\begin{equation}
\label{SO_k painted vertex rules}
    \parbox{0.8\linewidth}{The $i$th painted vertex must lie weakly to the left of the arcs in $C_i$.}
    \end{equation}
For indices $i$ where $C_i$ does not exist, this condition is fulfilled vacuously.  As a non-example of this rule, the graph below is \emph{not} included in the basis of $\SO_3$-invariants on $V^9$:
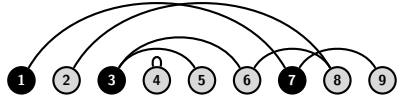
\begin{center}
\scalebox{0.5}{\begin{tikzpicture}[-,auto,node distance=1.2cm,
  ultra thick,plain node/.style={circle,draw,font=\sffamily\bfseries,fill=gray!30}, painted node/.style={circle,draw,font=\sffamily\bfseries,fill=black, text=white}]
\tikzstyle{every loop}=[-, distance = 10]

\node[painted node] (1) {1};
\node[plain node] (2) [right of=1] {2};
\node[painted node] (3) [right of=2] {3};
\node[plain node] (4) [right of=3] {4};
\node[plain node] (5) [right of=4] {5};
\node[plain node] (6) [right of=5] {6};
\node[painted node] (7) [right of=6] {7};
\node[plain node] (8) [right of=7] {8};
\node[plain node] (9) [right of=8] {9};


\draw [bend left = 60] (1) to (7);
\draw [bend left = 60] (3) to (6);
\draw [bend left = 60] (2) to (8);
\draw [bend left = 60] (3) to (5);
\draw [bend left = 60] (6) to (8);
\draw [bend left = 60] (7) to (9);


\path (4) edge [loop above] node {} (1);

\end{tikzpicture}}
\end{center}
This is because the 2nd painted vertex, which is 3, does not lie weakly to the left of $C_2 = \{ (2,8), (6,8) \}$.  If vertex 2 were painted instead of vertex 3, then the graph would obey the rule~\eqref{SO_k painted vertex rules}, since the 3rd painted vertex (7) already lies weakly left of $C_3 = \{(7,9)\}$.

\begin{prop}[Polynomial invariants for $\SO(V)$] Let $\dim V = k$.

\label{prop:SO_k graphs}

\begin{enumerate}
    \item We have $\C[V^{n}]^{\SO(V)} \cong \C[V^{n}]^{\O(V)} \oplus \C[V^{n}]^{\O(V), \: \det}$.
    
    \item A basis for $\C[V^{n}]^{\O(V),\: \det}_d$ is obtained from the arc diagrams in Proposition~\ref{prop:O_k graphs} with total degree $d - k$, by adding exactly one hyperedge (i.e., painting $k$ of the vertices) according to the rule~\eqref{SO_k painted vertex rules}.

    \item The corresponding $\O(V)$-semiinvariant is the product of the $f_{ij}$ corresponding to each arc $(i,j)$, multiplied by $\det(v_{i_1}, \ldots, v_{i_k})$, where $i_1, \ldots, i_k$ are the painted vertices.
 \end{enumerate}
\end{prop}

\begin{prop}[Tensor invariants for $\SO(V)$]
    \label{prop:SO_k tensor invariants}
    A basis for $(V^{\otimes n})^{\SO(V)}$ is given by the set of 1-regular arc diagrams on $n$ vertices which include at most one hyperedge satisfying the rule~\eqref{SO_k painted vertex rules}, and whose arcs can be decomposed into $\leq k$ weak nestings.
    In other words, there are either $0$ or $k$ painted vertices, and each vertex is either painted or else is incident to exactly one arc.
    Assuming $k \leq n$, we have the following possibilities:
    \begin{itemize}
        \item $n,k$ both even $\Longrightarrow$ graphs may have $0$ or $1$ hyperedge;
        \item $n,k$ both odd $\Longrightarrow$ graphs must have $1$ hyperedge;
        \item $n$ even, $k$ odd $\Longrightarrow$ graphs must have $0$ hyperedges;
        \item $n$ odd, $k$ even $\Longrightarrow$ there exist no $\SO(V)$-invariants.
    \end{itemize}
\end{prop}

\begin{corollary}
    We have $\displaystyle\dim (V^{\otimes n})^{\SO(V)} = \sum_{\mathclap{\substack{\mu \vdash n, \\ \ell(\mu) \leq k, \\ \textup{row lengths all even} \\ \textup{or all odd}}}} \#\SYT(\mu)$.
\end{corollary}

\begin{proof}
    By Corollary~\ref{cor:O_k tensor dim}, we already know that the SYT's with even row lengths count the $\O(V)$ tensor invariants, i.e., the graphs with 0 hyperedges.  To enumerate the graphs with 1 hyperedge, consider the upper-triangular part of the adjacency matrix of such a graph (with respect to arcs only).
    This upper-triangular matrix has all zeros in the rows and columns indexed by the $k$ painted vertices, and has exactly one ``1'' in each of the $n-k$ remaining rows and columns.  Therefore this matrix corresponds, via $\RSK_{\O}$, to a tableau of size $n-k$ with at most $k$ rows, all of which have even length, and where the entries are distinct numbers in $[n]$.
    The remaining $k$ numbers (which label the painted vertices) then determine a column which is prepended to this tableau.
    Since the graph obeyed the rule~\eqref{SO_k painted vertex rules}, this new tableau is an SSYT; but in fact, it contains each entry $1, \ldots, n$ exactly once, and thus is an SYT.
    This process is clearly invertible.
    (Note that the four cases in Proposition~\ref{prop:SO_k tensor invariants}, regarding the parity of $n$ and $k$, still hold when we view graphs as SYT's.)
\end{proof}

\subsection{The symplectic group}
\label{sec:graphs Sp}

We define the subposet
\[
\mathcal{E}^n_{\leqslant 2k} \coloneqq \{ A \in \mathcal{C}^n_{\leqslant 2k} \text{ such that $\#A$ is even} \}.
\]
For $A = (a_1, \ldots, a_{2t}) \in \mathcal{E}^n_{\leqslant 2k}$, let $\operatorname{pf}(A)$ denote the $2t$-Pfaffian of the alternating matrix $[z_{ij}]$ determined by the row and column indices $a_1, \ldots, a_{2t}$.  
Then for $A_1 \leq \cdots \leq A_{\ell} \in \mathcal{E}^n_{\leqslant 2k}$, we define the standard monomial
\[
\s(A_1, \ldots, A_\ell) \coloneqq \prod_{i=1}^\ell \operatorname{pf}(A_i).
\]
The set of standard monomials furnishes a basis for $\C[z_{ij}] / \ker \pi^*$; see \cite{DeConciniProcesi}*{\S6}, \cite{DeConciniSymplectic}, \cite{Procesi}*{\S8.4}, or \cite{Lakshmibai}*{Thm.~7.2.6.4}. 
 Again we call $\#A_1$ the \emph{width} of the standard monomial. 
 Viewing the $A_i$ as columns of an SSYT, we have a natural bijection
\begin{equation}
    \label{bijection s-monomials SSYT Sp_2k}
    \{ \text{standard monomials of width $\leq k$}\} \longleftrightarrow \bigcup_{\mathclap{\substack{\mu \in \Par(2k \times \infty)\\ \text{with even column lengths}}}} \: \SSYT(\mu,n).
\end{equation}
We define $\m(A_1, \ldots, A_\ell)$ to be the ordinary monomial whose degree matrix equals $\RSK_{\Sp}(A_1, \ldots, A_\ell)$.  
Hence by Proposition~\ref{prop:RSK_Sp}, we have the analogue of the degree- and width-preserving bijection~\eqref{RSK_GL s to m}.
Let $f_A \coloneqq \pi^*(\m(A_1, \ldots, A_\ell))$.

Viewing $\RSK_{\Sp}(A_1, \ldots, A_\ell)$ as an adjacency matrix, we see that $f_A$ is represented by an arc diagram with multiple edges but \textit{without} loops, on the vertices $1, \ldots, n$; see~\cite{Burge}*{Fig.~1}.  
The relationship between nestings and monomial width is identical to the $\GL(V)$ case above; see Figure~\ref{fig:Sp_2k graph example}. 
The proofs of the propositions and corollary below follow their analogues in the preceding subsections.

\begin{prop}[Polynomial invariants for $\Sp(V)$]
    \label{prop:Sp_2k graphs}
    Let $\dim V = 2k$.
    A basis for $\C[V^{n}]^{\Sp(V)}$ is given by the set of all arc diagrams (with multiple edges but without loops) on $n$ vertices, with total degree $d$, such that no strict nesting contains more than $k$ arcs.  
    In particular, we interpret each such arc diagram as the product in which $f_{ij} $ appears once for each arc $(i,j)$.
\end{prop}

\begin{prop}[Tensor invariants for $\Sp(V)$]
     \label{prop:Sp_2k tensor invariants}
     The space $(V^{\otimes n})^{\Sp(V)}$ is nonzero if and only if $n$ is even.
    In this case, a basis for $(V^{\otimes n})^{\Sp(V)}$ is given by the set of 1-regular arc diagrams on $n$ vertices, such that no strict nesting contains more than $k$ arcs.
     In particular, identifying $V^{\otimes n}$ as the space of multilinear forms on $V^{n}$, we interpret each arc $(i,j)$ as the contraction $f_{ij}$. 
 \end{prop}

\begin{corollary}
    We have $\displaystyle \dim (V^{\otimes n})^{\Sp(V)} = \sum_{\mathclap{\substack{\mu \vdash n, \\ \ell(\mu) \leq 2k, \\ \textup{even column lengths}}}} \#\SYT(\mu)$ if $n$ is even, and $0$ if $n$ is odd.
\end{corollary}

\begin{figure}
    \centering
    \scalebox{0.7}{
\begin{tikzpicture}[-,auto,node distance=1cm,
  very thick,plain node/.style={circle,draw,font=\sffamily\bfseries,fill=gray!30}, bend left = 60]


\node[plain node] (1) {1};
\node[plain node] (2) [right of=1] {2};
\node[plain node] (3) [right of=2] {3};
\node[plain node] (4) [right of=3] {4};
\node[plain node] (5) [right of=4] {5};
\node[plain node] (6) [right of=5] {6};


\draw (1) to (5);
\draw (2) to (6);
\draw [bend left = 70] (2) to (6);
\draw (3) to (4);
\draw (2) to (6);
\draw (2) to (6);
\draw (2) to (6);
\draw (4) to (6);

\end{tikzpicture}
}
    \caption{An example of a basis graph for the polynomial $\Sp(V)$-invariants, where $n = 6$. 
    This graph corresponds to the invariant $f_{15} f_{26}^2 f_{34} f_{46}$.  By inspection, we see that the largest strict nesting $\{(2,6), \: (3,4)\}$ contains $2$ arcs, and so this graph is included in the basis for $\C[V^6]^{\Sp_{2k}}_{10}$ for all $k \geq 2$.  
    We could also arrive at the width $2$ by decomposing into the two chains $C_1 = \{(1,5), (2,6), (2,6), (4,6)\}$ and $C_2 = \{(3,4)\}$.}
    \label{fig:Sp_2k graph example}
\end{figure}
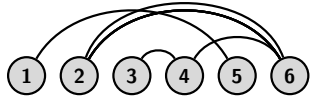

\section{Stanley decompositions and Hilbert--Poincar\'e series}
\label{sec:Stanley decomp and Hilbert series}

Recall from~\eqref{Stanley decomp} that for any of our classical posets $\P$, we have a shelling of $\Delta_k = \Delta_k(\P)$ that induces the Stanley decomposition
\begin{equation}
\label{stanley decomp again}
    \C[\Delta_k] = \bigoplus_{\f \in \scrF_k} z_{\cor(\f)} \C[z_{ij} \mid (i,j) \in \f].
\end{equation}
For $\P = \P_{\GL}$ or $\P_{\O}$ or $\P_{\Sp}$, one has $\C[\Delta_k] \cong \C[\M_{p,q}^{\leqslant k}]$ or $\C[\SM_n^{\leqslant k}]$ or $\C[\AM_n^{\leqslant 2k}]$, respectively.
This fact, when combined with the RSK correspondences between standard monomials~\cite{DeConciniProcesi} and ordinary monomials, was used in the 1990s by~\cite{Sturmfels}, \cite{Conca94}, and \cite{Herzog}, respectively, to describe determinantal rings.
Combining this with the isomorphism $\pi^*: z_{ij} \mapsto f_{ij}$ of Weyl's fundamental theorems, we see also that $\C[\Delta_k] \cong \C[\C[z_{ij}] / \ker \pi^*$.

In this section, we use this Stanley decomposition to write down a combinatorial description of the Hilbert--Poincar\'e series of the invariant ring.  
All notation and terminology pertaining to shellings and corners is from Section~\ref{sec:shellings}.
In the propositions in this section, we write $f_{\cor(\f)} \coloneqq \prod_{(i,j) \in \cor(\f)} f_{ij}$.
We also write $\F{\nu}{n}$ for the irreducible $\gl_n$-module with highest weight $\nu$.

\subsection{The general linear group}

The Stanley decomposition in~\eqref{stanley decomp again} leads immediately to the Hilbert--Poincar\'e series 
 \begin{equation}
    \label{Hilbert series z_ij GL_k}
    P\!\left(\C[z_{ij}] / \ker \pi^*; t\right) = 
    \frac{\sum_{ \f \in \scrF_{k} } t^{\# \cor(\f)}}{(1-t)^{k(p+q-k)}},
 \end{equation}
 since $k(p+q-k)$ is the size of each facet $\f \in \scrF_k$.
The alternative form of this Hilbert--Poincar\'e series in the following proposition can be found in~\cite{EnrightHunziker04}*{Thm.~23} and \cite{EricksonHunziker23}*{Table 3, row I}, in the context of the $k$th Wallach representation of $\mathfrak{su}(p,q)$.  
(See Section~\ref{sub:Wallach}.)
The combinatorial proof for the numerator polynomial is adapted from~\cite{Krattenthaler}*{Prop.~28 and Fig.~8}.

 \begin{prop}
 \label{prop:Hilbert series GL_k}
     Let $\dim V = k$.
     The ring of $\GL(V)$-invariants has the Stanley decomposition
    \[
    \C[V^{*p} \oplus V^q]^{\GL(V)} = \bigoplus_{\f \in \scrF_k} f_{\cor(\f)} \C[f_{ij} \mid (i,j) \in \f].
    \]
    Furthermore, for $k \leq \min\{p,q\}$, the Hilbert--Poincar\'e series is 
     \begin{equation*}
         P\!\left(\C[V^{*p} \oplus V^q]^{\GL(V)}; t\right) = \frac{\sum_\nu (\dim\F{\nu}{p-k} \dim \F{\nu}{q-k}) \: t^{2|\nu|}}{(1-t^2)^{k(p+q-k)}},
     \end{equation*}
     where $\nu \in \Par(\min\{p-k,\:q-k\} \times k)$.
 \end{prop}

\begin{proof}
The Stanley decomposition follows from~\eqref{stanley decomp again} and the fact that $\C[z_{ij}] / \ker \pi^*$ is isomorphic to the $\GL(V)$-invariant ring via $z_{ij} \mapsto f_{ij}$.
We claim that the numerator in~\eqref{Hilbert series z_ij GL_k} equals
\[
\sum_{\nu} \#\SSYT(\nu,\:p-k) \cdot \# \SSYT(\nu, \: q-k) \cdot t^{|\nu|},
\]
with $\nu \in \Par(\min\{p-k,\:q-k\} \times k)$.  
To show this, we will use Proposition~\ref{prop:RSK_GL} to show that $\RSK_{\GL}$ yields a bijection between $\scrF_k$ and the set of tableau pairs above, which converts the number of corners into the size of each tableau.

Let $\f \in \scrF_k$. 
Recall that for each path $\bp_i$ in $\f$, the corners in $\cor(\bp_i)$ form a (possibly empty) strict chain within the $(p-k) \times (q-k)$ region in $\P_{\GL}$ whose northwest vertex is immediately south of the point $s_i$.  
(See the shaded region for $\bp_3$ in Figure~\ref{subfig:facet example GL}.)
By superimposing all of these regions inside a common $(p-k) \times (q-k)$ rectangle, we convert $\cor(\f) = \bigcup_i \cor(\bp_i)$ into a matrix in $\M_{p-k,q-k}(\mathbb{N})$ whose entries sum to $\#\cor(\f)$, and whose multisupport is the union of $k$ strict chains in the smaller poset $\P_{\GL}(p-k,\: q-k)$.  
Consider the vertical reflection of this matrix, in which the $k$ strict chains become $k$ antichains; the multisupport of this reflected matrix therefore has height $\leq k$ in $\P_{\GL}(p-k, \: q-k)$.  
To show surjectivity, we choose any matrix in the latter set, and decompose its multisupport into the antichains $D_i$, defined in~\eqref{D1 for GL_k}; see also Figure~\ref{fig:Ci Di for GL_k}.
Upon vertical reflection, each $D_i$ recovers the corners of the path $\bp_i$.  
For injectivity, note that the inclusions of upper-order ideals in Lemma~\ref{lemma:Ci and Di for GL_k}(3) imply that the paths determined by the $D_i$ do not intersect, and so there is a unique facet $\f = \bigcup_i \bp_i$ determined by the $D_i$.
Hence we have a bijection from $\scrF_k$ to the matrices of height $\leq k$, which (by Proposition~\ref{prop:RSK_GL}) are in bijection with $\bigcup_{\nu} \SSYT(\nu,\:p-k) \times \SSYT(\nu,\:q-k)$, where $\nu$ has at most $k$ columns.  
By part (1) of the same proposition, this bijection $\f \mapsto (T,U)$ has the property that $\#\cor(\f) = |T| = |U|$.  
This proves the claim for the numerator polynomial.  
Finally, since $\pi^* : z_{ij} \longmapsto f_{ij}$ doubles the degree, the proposition follows upon substituting $t^2$ for $t$ in~\eqref{Hilbert series z_ij GL_k}.
\end{proof}

We recall from~\cite{Stanley78}*{Thm.~4.4} the following fact: a Cohen--Macaulay ring is a Gorenstein ring if and only if the numerator polynomial of its reduced Hilbert--Poincar\'e series is palindromic.

\begin{corollary}
    \label{cor:GL_k Gorenstein}
    Suppose that $k < \min\{p,q\}$.  
    Then the ring $\C[V^{*p} \oplus V^q]^{\GL(V)}$ is Gorenstein if and only if $p = q$.
\end{corollary}

\begin{proof}
    Suppose that $\la$ is the shape of a Young diagram, and let $\mu$ be the shape obtained from a 180-degree rotation of the complement of $\la$ inside any rectangular shape.
    It is a general fact that $\F{\mu}{p}$ is the tensor product of $\F{\la}{p}$ with some power of the determinant, and so $\dim\F{\la}{p} = \dim\F{\mu}{p}$.  
    Since the shapes $\nu$ in Proposition~\ref{prop:Hilbert series GL_k} are restricted to a rectangular shape, the numerator polynomial of the Hilbert--Poincar\'e series must be palindromic if $p=q$.  Conversely, suppose (without loss of generality) that $p < q$.  
    Then the constant term in the numerator is $1$, but the leading coefficient is strictly greater than $1$, since the dimension of the $\gl_q$-module whose highest weight corresponds to the full $(p-k) \times k$ rectangle has dimension strictly greater than 1.
\end{proof}

\subsection{The special linear group}
\label{sub:Hilbert series SL_k}

We write $\mathbf{F} \rightarrow I$ to express that $\mathbf F$ is a maximal chain of the form $\mathbf{F}: (1,\ldots, k) \rightarrow I$ in $\Ceq{p}$.  
Likewise, we write $I \rightarrow \f$ to state that $\f$ is the union of non-intersecting paths $\bp_\ell : (i_\ell, 1) \rightarrow (p,\: q+1-\ell)$ in $\P_{\GL}$, for $1 \leq \ell \leq k$.  
The statement $\mathbf{F} \rightarrow \f$ means that there exists an $I \in \Ceq{p}$ such that $\mathbf{F} \rightarrow I$ and $I \rightarrow \f$.
Analogously, $J \downarrow \f$ means that $\f$ is the union of non-intersecting paths $\bp_\ell : (1, j_\ell) \rightarrow (p+1-\ell,\:q)$ in $\P_{\GL}$, for $1 \leq \ell \leq k$.  
We write $\mathbf{F} \downarrow J$ to state that $\mathbf{F}$ is a path of the form $\mathbf{F}:(1,\ldots, k) \rightarrow J$ in $\Ceq{q}$. 
Finally, $\mathbf{F} \downarrow \f$ means that there is some $J \in \Ceq{q}$ such that $\f \downarrow J$ and $J \downarrow \mathbf{F}$.
The direction of our arrows evokes the starting points of the paths $\bp_\ell$ in our diagrams: if $I \rightarrow \f$, then $I$ gives the starting points along the western edge, while if $J \downarrow \f$, then $J$ gives the starting points along the northern edge.

We take as our set of facets
\[
\{ \f \in \scrF_k \} \cup \{ \mathbf{F} \rightarrow \f \} \cup  \{ \f \rightarrow \mathbf{F}\}.
\]
If $I \rightarrow \f$, then we define a corner of $\f$ in the usual way, namely as a \scalebox{2}{$\llcorner$}-pattern.
If $J \downarrow \f$, however, then the situation is rotated by 180 degrees, and so we must define a corner as a \rotatebox[origin=c]{180}{{\scalebox{2}{$\llcorner$}}}-pattern.
In $\Ceq{p}$ or $\Ceq{q}$, we write two consecutive elements of a path with the notation $A \xrightarrow{i} B$ if $B$ is obtained by adding $1$ to the $i$th coordinate of $A$. 
 Then we define a corner of $\mathbf{F}$ to be the set of all $B \in \mathbf{F}$ such that $\mathbf{F}$ contains the pattern
\[
A \xrightarrow{i} B \xrightarrow{j} C, \qquad i < j.
\]
For a subset $S \subseteq \Ceq{p}$, write $\det_{S^*} \coloneqq \prod_{I \in S} \det_{I^*}$.  
Likewise, for $S \subseteq \Ceq{q}$, write $\det_{S} \coloneqq \prod_{J \in S} \det_J$. 


\begin{prop}
    \label{prop:Hilbert series SL_k}
    Let $\dim V = k$.
    The ring of $\SL(V)$-invariants has the Stanley decomposition
    \begin{align*}
    \C[V^{*p} \oplus V^q]^{\SL(V)} = \phantom{\oplus} &\bigoplus_{\f \in \scrF_k} f_{\cor(\f)} \: \C[f_{ij} \mid (i,j) \in \f]\\
    \oplus &\bigoplus_{\mathbf{F} \rightarrow \f} f_{\cor(\f)} \: \C[f_{ij}, {\rm det}_{I^*} \mid (i,j) \in \f, \: I \in \mathbf{F}] \cdot {\rm det}_{\cor(\mathbf{F})^*}\\
    \oplus &\bigoplus_{\mathbf{F} \downarrow \f} f_{\cor(\f)} \: \C[f_{ij}, {\rm det}_{J} \mid (i,j) \in \f, \: J \in \mathbf{F}] \cdot {\rm det}_{\cor(\mathbf{F})}.
    \end{align*}
    Furthermore, for $k \leq \min\{p,q\}$, the Hilbert--Poincar\'e series is
    \begin{align*}
      P(\C[V^{*p} \oplus V^q]^{\SL(V)}; t) = \phantom{+}  &P(\C[V^{*p} \oplus V^q]^{\GL(V)};t)\\
        + & \quad t^k \sum_{I \in \Ceq{p}} \left(\frac{\sum_{\mathbf{F} \rightarrow I} (t^k)^{\#\cor(\mathbf F)}}{(1-t^k)^{\#\mathbf{F}}}\right) \left( \frac{\sum_{I \rightarrow \f} (t^2)^{\#\cor(\f)}}{(1-t^2)^{\#\f}} \right)\\
        + & \quad t^k \sum_{J \in \Ceq{q}} \left(\frac{\sum_{\mathbf{F} \downarrow J} (t^k)^{\#\cor(\mathbf F)}}{(1-t^k)^{\#\mathbf{F}}}\right) \left( \frac{\sum_{J \downarrow \f} (t^2)^{\#\cor(\f)}}{(1-t^2)^{\#\f}} \right),
    \end{align*}
    where in the first sum we have $\#\mathbf{F} = |I| + \frac{2-k(k+1)}{2}$ and $\#\f = k(p+q-\frac{k-1}{2})-|I|$; the same holds in the second sum upon replacing $I$ with $J$.
\end{prop}

\begin{proof}
     To prove the Stanley decomposition, we must show that each monomial of the form~\eqref{fAB three forms SL} lies inside exactly one component.
     By Proposition~\ref{prop:Hilbert series GL_k}, each $f_{AB}$ lies inside a unique component of the first direct sum.
     Hence without loss of generality, consider a monomial $f_{AB} \prod_{\ell = 1}^m \det_{I_\ell^*}$.
    Since the order complex of $\Ceq{p}$ is shellable, the chain $I_1 \leq \cdots \leq I_m$ determines a unique facet $\mathbf{F}$.
     By part (1) of Lemma~\ref{lemma:Ci and Di for GL_k}, the support of $C^*_i$ in $\RSK_{\GL}(A,B)$ equals the $i$th entry in the first column $A_1$ of $A$.
     Since $I_m \leq A_1$, the support of $f_{AB}$ is contained in some family $\f$ such that $I \rightarrow \f$.
     By~\cite{HerzogTrung}*{Thm.~4.9}, the set $\{ \f \mid I \rightarrow \f\}$ is the set of facets of a shellable subcomplex of $\Delta_k$, such that the restrictions are the usual corners in this paper; hence this $\f$ is in fact unique.
     Since $\mathbf{F} \rightarrow I_m \rightarrow \f$, we see that $f_{AB} \prod_{\ell = 1}^m \det_{I_\ell^*}$ lies only in the component labeled by $\mathbf{F} \rightarrow \f$.
     The Stanley decomposition follows upon applying $\psi^*$.  
     
     The Hilbert--Poincar\'e series, in turn, follows directly from the Stanley decomposition, where we recall that $f_{ij}$ has degree 2, while $\det_{I^*}$ and $\det_J$ have degree $k$.  
     Note that we multiply each of the $I$ and $J$ sums by $t^k$, so as not to double-count the monomials $f_{AB}$, which were already counted in the Hilbert--Poincar\'e series of the $\GL_k$-invariants.
     The formulas for $\#\mathbf{F}$ and $\#\f$ are straightforward calculations.
     \end{proof}

\begin{rem}
    There is an explicit formula (involving determinants) for the numerators $\sum_{I\rightarrow \f} t^{\#\cor(\f)}$ and $\sum_{J \downarrow \f} t^{\#\cor(\f)}$ in Proposition~\ref{prop:Hilbert series SL_k}, which is a special case of a general formula of Krattenthaler~\cite{KrattenthalerTurnsPaper}*{Thm.~1} for counting corners in families of non-intersecting lattice paths with arbitrary endpoints.
    Krattenthaler's formula can be seen as a $q$-analogue of Abhyannkar's determinantal formula~\cite{Abhyankar}.
\end{rem}

In the examples below, we indicate each element in $\cor(\f)$ with a solid red square, and we place a red rectangle around each element of $\cor(\mathbf{F})$.

\begin{ex}[$\SL_2$-invariants, for $p=3$ and $q=4$]\
\label{ex:SL invariants k2p3q4}

\input{Hilbert_series_examples/k2p3q4}
    
\end{ex}

\begin{ex}[$\SL_3$-invariants, for $p=q=4$]\

    \tikzstyle{corner}=[rectangle,draw=black,fill=red, minimum size = 5pt, inner sep=0pt]
\tikzstyle{endpt}=[circle,fill=black, text=white, font=\tiny\sffamily\bfseries, minimum size = 8pt, inner sep=0pt]

Below is the Hilbert series for the $\GL_3$-invariants:

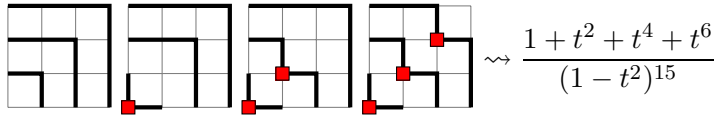
\begin{tikzpicture}[scale=.45, baseline]
\draw[gray] (1,1) grid (4,4);

\draw[ultra thick] (1,2)--(2,2)--(2,1);
\draw[ultra thick] (1,3)--(3,3)--(3,1);
\draw[ultra thick] (1,4)--(4,4)--(4,1);
\end{tikzpicture}
\begin{tikzpicture}[scale=.45, baseline]
\draw[gray] (1,1) grid (4,4);

\draw[ultra thick] (1,2)--(1,1)--(2,1);
\draw[ultra thick] (1,3)--(3,3)--(3,1);
\draw[ultra thick] (1,4)--(4,4)--(4,1);

\node at (1,1) [corner] {};

\end{tikzpicture}
\begin{tikzpicture}[scale=.45, baseline]
\draw[gray] (1,1) grid (4,4);

\draw[ultra thick] (1,2)--(1,1)--(2,1);
\draw[ultra thick] (1,3)--(2,3)--(2,2)--(3,2)--(3,1);
\draw[ultra thick] (1,4)--(4,4)--(4,1);

\node at (1,1) [corner] {};
\node at (2,2) [corner] {};

\end{tikzpicture}
\begin{tikzpicture}[scale=.45, baseline]
\draw[gray] (1,1) grid (4,4);

\draw[ultra thick] (1,2)--(1,1)--(2,1);
\draw[ultra thick] (1,3)--(2,3)--(2,2)--(3,2)--(3,1);
\draw[ultra thick] (1,4)--(3,4)--(3,3)--(4,3)--(4,1);

\node at (1,1) [corner] {};
\node at (2,2) [corner] {};
\node at (3,3) [corner] {};

\node[right=0pt of current bounding box.east,anchor=west
    ]{$\leadsto \dfrac{1 + t^2 + t^4 + t^6}{(1-t^2)^{15}}$};

\end{tikzpicture}

Because $p=q$, the sums for $I$ and $J$ are identical; therefore we need only consider each element $I \in \mathcal{C}^4_3$ once.  (The missing diagrams are obtained by reflection about the main diagonal.)
To compensate, we multiply each summand by $2t^3$ rather than just $t^3$:

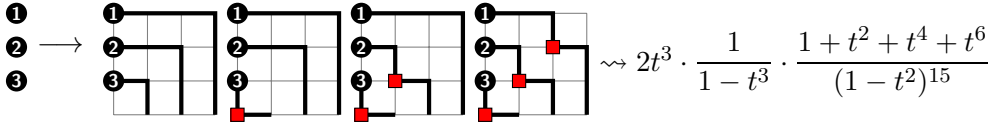
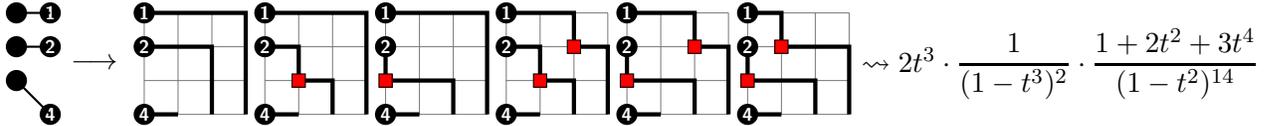
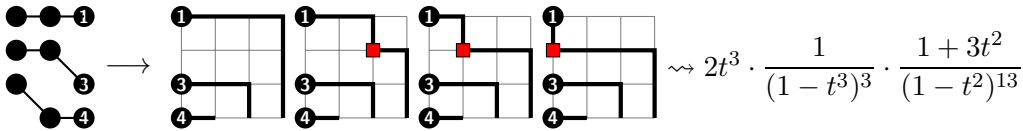
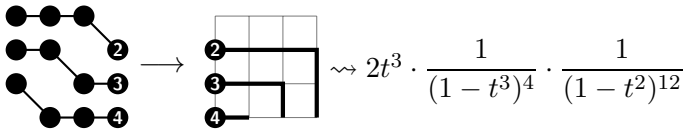
\begin{tikzpicture}[scale=.45,baseline]

\node at (1,4) [endpt] {1};
\node at (1,3) [endpt] {2};
\node at (1,2) [endpt] {3};

\node[right=0pt of current bounding box.east,anchor=west
    ]{$\longrightarrow$};
\end{tikzpicture}
\begin{tikzpicture}[scale=.45, baseline]
\draw[gray] (1,1) grid (4,4);

\draw[ultra thick] (1,2)--(2,2)--(2,1);
\draw[ultra thick] (1,3)--(3,3)--(3,1);
\draw[ultra thick] (1,4)--(4,4)--(4,1);

\node at (1,4) [endpt] {1};
\node at (1,3) [endpt] {2};
\node at (1,2) [endpt] {3};

\end{tikzpicture}
\begin{tikzpicture}[scale=.45, baseline]
\draw[gray] (1,1) grid (4,4);

\draw[ultra thick] (1,2)--(1,1)--(2,1);
\draw[ultra thick] (1,3)--(3,3)--(3,1);
\draw[ultra thick] (1,4)--(4,4)--(4,1);

\node at (1,1) [corner] {};

\node at (1,4) [endpt] {1};
\node at (1,3) [endpt] {2};
\node at (1,2) [endpt] {3};

\end{tikzpicture}
\begin{tikzpicture}[scale=.45, baseline]
\draw[gray] (1,1) grid (4,4);

\draw[ultra thick] (1,2)--(1,1)--(2,1);
\draw[ultra thick] (1,3)--(2,3)--(2,2)--(3,2)--(3,1);
\draw[ultra thick] (1,4)--(4,4)--(4,1);

\node at (1,1) [corner] {};
\node at (2,2) [corner] {};

\node at (1,4) [endpt] {1};
\node at (1,3) [endpt] {2};
\node at (1,2) [endpt] {3};

\end{tikzpicture}
\begin{tikzpicture}[scale=.45, baseline]
\draw[gray] (1,1) grid (4,4);

\draw[ultra thick] (1,2)--(1,1)--(2,1);
\draw[ultra thick] (1,3)--(2,3)--(2,2)--(3,2)--(3,1);
\draw[ultra thick] (1,4)--(3,4)--(3,3)--(4,3)--(4,1);

\node at (1,1) [corner] {};
\node at (2,2) [corner] {};
\node at (3,3) [corner] {};

\node at (1,4) [endpt] {1};
\node at (1,3) [endpt] {2};
\node at (1,2) [endpt] {3};

\node[right=0pt of current bounding box.east,anchor=west
    ]{$\leadsto 2t^3 \cdot \dfrac{1}{1-t^3} \cdot \dfrac{1 + t^2 + t^4 + t^6}{(1-t^2)^{15}}$};

\end{tikzpicture}

\begin{tikzpicture}[scale=.45,baseline]
\node at (1,4) [endpt] {1};
\node at (1,3) [endpt] {2};
\node at (1,1) [endpt] {4};
\node at (0,4) [endpt] {};
\node at (0,3) [endpt] {};
\node at (0,2) [endpt] {};
\draw[thick] (0,4)--(1,4);
\draw[thick] (0,3)--(1,3);
\draw[thick] (0,2)--(1,1);

\node[right=0pt of current bounding box.east,anchor=west
    ]{$\longrightarrow$};
\end{tikzpicture}
\begin{tikzpicture}[scale=.45, baseline]
\draw[gray] (1,1) grid (4,4);

\draw[ultra thick] (1,1)--(2,1);
\draw[ultra thick] (1,3)--(3,3)--(3,1);
\draw[ultra thick] (1,4)--(4,4)--(4,1);

\node at (1,4) [endpt] {1};
\node at (1,3) [endpt] {2};
\node at (1,1) [endpt] {4};
\end{tikzpicture}
\begin{tikzpicture}[scale=.45, baseline]
\draw[gray] (1,1) grid (4,4);

\draw[ultra thick] (1,1)--(2,1);
\draw[ultra thick] (1,3)--(2,3)--(2,2)--(3,2)--(3,1);
\draw[ultra thick] (1,4)--(4,4)--(4,1);

\node at (2,2) [corner] {};

\node at (1,4) [endpt] {1};
\node at (1,3) [endpt] {2};
\node at (1,1) [endpt] {4};
\end{tikzpicture}
\begin{tikzpicture}[scale=.45, baseline]
\draw[gray] (1,1) grid (4,4);

\draw[ultra thick] (1,1)--(2,1);
\draw[ultra thick] (1,3)--(1,2)--(3,2)--(3,1);
\draw[ultra thick] (1,4)--(4,4)--(4,1);

\node at (1,2) [corner] {};

\node at (1,4) [endpt] {1};
\node at (1,3) [endpt] {2};
\node at (1,1) [endpt] {4};
\end{tikzpicture}
\begin{tikzpicture}[scale=.45, baseline]
\draw[gray] (1,1) grid (4,4);

\draw[ultra thick] (1,1)--(2,1);
\draw[ultra thick] (1,3)--(2,3)--(2,2)--(3,2)--(3,1);
\draw[ultra thick] (1,4)--(3,4)--(3,3)--(4,3)--(4,1);

\node at (2,2) [corner] {};
\node at (3,3) [corner] {};

\node at (1,4) [endpt] {1};
\node at (1,3) [endpt] {2};
\node at (1,1) [endpt] {4};
\end{tikzpicture}
\begin{tikzpicture}[scale=.45, baseline]
\draw[gray] (1,1) grid (4,4);

\draw[ultra thick] (1,1)--(2,1);
\draw[ultra thick] (1,3)--(1,2)--(3,2)--(3,1);
\draw[ultra thick] (1,4)--(3,4)--(3,3)--(4,3)--(4,1);

\node at (1,2) [corner] {};
\node at (3,3) [corner] {};

\node at (1,4) [endpt] {1};
\node at (1,3) [endpt] {2};
\node at (1,1) [endpt] {4};
\end{tikzpicture}
\begin{tikzpicture}[scale=.45, baseline]
\draw[gray] (1,1) grid (4,4);

\draw[ultra thick] (1,1)--(2,1);
\draw[ultra thick] (1,3)--(1,2)--(3,2)--(3,1);
\draw[ultra thick] (1,4)--(2,4)--(2,3)--(4,3)--(4,1);

\node at (1,2) [corner] {};
\node at (2,3) [corner] {};

\node at (1,4) [endpt] {1};
\node at (1,3) [endpt] {2};
\node at (1,1) [endpt] {4};

\node[right=0pt of current bounding box.east,anchor=west
    ]{$\leadsto 2t^3 \cdot \dfrac{1}{(1-t^3)^2} \cdot \dfrac{1 + 2t^2 + 3t^4}{(1-t^2)^{14}}$};
    
\end{tikzpicture}

\begin{tikzpicture}[scale=.45,baseline]
\node at (1,4) [endpt] {1};
\node at (1,2) [endpt] {3};
\node at (1,1) [endpt] {4};
\node at (0,4) [endpt] {};
\node at (0,3) [endpt] {};
\node at (0,1) [endpt] {};
\node at (-1,4) [endpt] {};
\node at (-1,3) [endpt] {};
\node at (-1,2) [endpt] {};

\draw[thick] (-1,4)--(0,4)--(1,4);
\draw[thick] (-1,3)--(0,3)--(1,2);
\draw[thick] (-1,2)--(0,1)--(1,1);

\node[right=0pt of current bounding box.east,anchor=west
    ]{$\longrightarrow$};
\end{tikzpicture}
\begin{tikzpicture}[scale=.45, baseline]
\draw[gray] (1,1) grid (4,4);

\draw[ultra thick] (1,1)--(2,1);
\draw[ultra thick] (1,2)--(3,2)--(3,1);
\draw[ultra thick] (1,4)--(4,4)--(4,1);

\node at (1,4) [endpt] {1};
\node at (1,2) [endpt] {3};
\node at (1,1) [endpt] {4};
\end{tikzpicture}
\begin{tikzpicture}[scale=.45, baseline]
\draw[gray] (1,1) grid (4,4);

\draw[ultra thick] (1,1)--(2,1);
\draw[ultra thick] (1,2)--(3,2)--(3,1);
\draw[ultra thick] (1,4)--(3,4)--(3,3)--(4,3)--(4,1);

\node at (3,3) [corner] {};

\node at (1,4) [endpt] {1};
\node at (1,2) [endpt] {3};
\node at (1,1) [endpt] {4};
\end{tikzpicture}
\begin{tikzpicture}[scale=.45, baseline]
\draw[gray] (1,1) grid (4,4);

\draw[ultra thick] (1,1)--(2,1);
\draw[ultra thick] (1,2)--(3,2)--(3,1);
\draw[ultra thick] (1,4)--(2,4)--(2,3)--(4,3)--(4,1);

\node at (2,3) [corner] {};

\node at (1,4) [endpt] {1};
\node at (1,2) [endpt] {3};
\node at (1,1) [endpt] {4};
\end{tikzpicture}
\begin{tikzpicture}[scale=.45, baseline]
\draw[gray] (1,1) grid (4,4);

\draw[ultra thick] (1,1)--(2,1);
\draw[ultra thick] (1,2)--(3,2)--(3,1);
\draw[ultra thick] (1,4)--(1,3)--(4,3)--(4,1);

\node at (1,3) [corner] {};

\node at (1,4) [endpt] {1};
\node at (1,2) [endpt] {3};
\node at (1,1) [endpt] {4};

\node[right=0pt of current bounding box.east,anchor=west
    ]{$\leadsto 2t^3 \cdot \dfrac{1}{(1-t^3)^3} \cdot \dfrac{1 + 3t^2}{(1-t^2)^{13}}$};

\end{tikzpicture}

\begin{tikzpicture}[scale=.45,baseline]
\node at (1,3) [endpt] {2};
\node at (1,2) [endpt] {3};
\node at (1,1) [endpt] {4};
\node at (-2,4) [endpt] {};
\node at (-1,4) [endpt] {};
\node at (0,4) [endpt] {};
\node at (-2,3) [endpt] {};
\node at (-1,3) [endpt] {};
\node at (0,2) [endpt] {};
\node at (-2,2) [endpt] {};
\node at (-1,1) [endpt] {};
\node at (0,1) [endpt] {};

\draw[thick] (-2,4)--(-1,4)--(0,4)--(1,3);
\draw[thick] (-2,3)--(-1,3)--(0,2)--(1,2);
\draw[thick] (-2,2)--(-1,1)--(0,1)--(1,1);

\node[right=0pt of current bounding box.east,anchor=west
    ]{$\longrightarrow$};
\end{tikzpicture}
\begin{tikzpicture}[scale=.45, baseline]
\draw[gray] (1,1) grid (4,4);

\draw[ultra thick] (1,1)--(2,1);
\draw[ultra thick] (1,2)--(3,2)--(3,1);
\draw[ultra thick] (1,3)--(4,3)--(4,1);

\node at (1,3) [endpt] {2};
\node at (1,2) [endpt] {3};
\node at (1,1) [endpt] {4};

\node[right=0pt of current bounding box.east,anchor=west
    ]{$\leadsto 2t^3 \cdot \dfrac{1}{(1-t^3)^4} \cdot \dfrac{1 }{(1-t^2)^{12}}$};

\end{tikzpicture}

Simplifying the sum of the rational functions above, we obtain the Hilbert series 
\[
P(\C[V^{*4} \oplus V^4]^{\SL_3}; t) = \dfrac{1 + 4t^2 + 4t^3 + 10t^4 + 8t^5 + 14t^6 + 8t^7 + 10t^8 + 4t^9 + 4t^{10} + t^{12}}{(1-t^2)^{12} (1-t^3)^4}.
\]
We observe that, unlike the previous example, the numerator polynomial is palindromic but not unimodal.
    
\end{ex}

\begin{ex}[$\SL_3$-invariants, for $p=3$ and $q=4$]
    \label{ex:SL_k invariants k3p3q4}
    A similar enumeration can be carried out by hand, as in the previous examples.
    This leads to the Hilbert--Poincar\'e series
    \begin{align*}
        P&(\C[V^{*3} \oplus V^4]^{\SL_3}; t) \\&= \frac{1}{(1-t^2)^{12}} + \frac{2t^3}{(1-t^3)(1-t^2)^{12}} + \frac{t^3 (1 +t^2 + t^4)}{(1-t^3)^2 (1-t^2)^{11}} + \frac
        {t^3(1+2t^2)}{(1-t^3)^3 (1-t^2)^{10}} + \frac{t^3}{(1-t^3)^4 (1-t^2)^9} \\[2ex]
        &= \frac{1 + 3 t^2 + 2 t^3 + 6 t^4 + 3 t^5 + 8 t^6 + 3 t^7 + 6 t^8 + 2 t^9 +  3 t^{10} + t^{12}}{(1 - t^2)^9 (1 - t^3)^3 (1 - t^6)}.
    \end{align*}
    Once again, the numerator is palindromic but not unimodal.  
    Note also the presence of the factor $(1-t^6)$ in the denominator, since unlike the previous examples, the exponent $6$ is neither $2$ nor $k$.
\end{ex}

\begin{rem}
    It follows from the result~\cite{HochsterRoberts}*{Cor.~1.9} for semisimple connected Lie groups that the ring $\C[V^{*p} \oplus V^q]^{\SL(V)}$ is Gorenstein for all values of $k$, $p$, and $q$.
\end{rem}

\subsection{The orthogonal group}

Recall from Section~\ref{sec:shellings} that $\#\f = k(2n-k+1)/2$ for each facet $\f \in \scrF_k \subset \Delta_k(\P_{\O})$.  
Combined with the Stanley decomposition of $\C[\Delta_k]$ in~\eqref{stanley decomp again}, this yields the Hilbert--Poincar\'e series
\begin{equation*}
    P\!\left(\C[z_{ij}]/\ker \pi^*; t\right) = 
    \frac{\sum_{ \f \in \scrF_{k} } t^{\# \cor(\f)}}{(1-t)^{k(2n-k+1)/2}}.
 \end{equation*}
 The Hilbert--Poincar\'e series in the following proposition is found in~\cite{EnrightHunziker04}*{Thm.~24} and \cite{EricksonHunziker23}*{Table 3, row IIIa}, in the context of the $k$th Wallach representation of $\mathfrak{sp}(n,\mathbb{R})$.  
 The combinatorial proof for the numerator polynomial is implicit in~\cite{Krattenthaler}*{Prop.~33}.

 \begin{prop}
 \label{prop:Hilbert series O_k}
     Let $\dim V = k$.
     The ring of $\O(V)$-invariants has the Stanley decomposition
     \[
     \C[V^{n}]^{\O(V)} = \bigoplus_{\f \in \scrF_k} f_{\cor(\f)} \C[f_{ij} \mid (i,j) \in \f].
     \]
     Furthermore, for $k \leq n$, the Hilbert--Poincar\'e series is
     \begin{equation}
     \label{O_k Hilbert series}
         P\!\left(\C[V^{n}]^{\O(V)}; t\right) = \frac{\sum_\nu (\dim\F{\nu}{n-k+1}) \: t^{|\nu|}}{(1-t^2)^{k(2n-k+1)/2}},
     \end{equation}
     where $\nu \in \Par((n-k+1) \times k)$ with columns of even length.
 \end{prop}

\begin{proof}
The Stanley decomposition again follows from applying $\pi^* : z_{ij} \longmapsto f_{ij}$ to~\eqref{stanley decomp again}.  
As for the numerator of the Hilbert--Poincar\'e series, this time we use the correspondence $\RSK_{\Sp}$ to establish a bijection between $\scrF_k$ and $\bigcup_\nu \SSYT(\nu, \: n-k+1)$, which converts the number of corners into half the size of the corresponding tableau.  
(Note the combinatorial ``switch'' from $\P_{\O}(n)$ to $\P_{\Sp}(n-k+1)$, which will be illuminated in Section~\ref{sub:ES reduction}.)  
The proof mimics that of Proposition~\ref{prop:Hilbert series GL_k}, with the modifications described below.

Recall from Section~\ref{sec:shellings} that for each path $\bp_i$ in $\f$, the corners in $\cor(\bp_i)$ form a strict chain within the $(n-k) \times (n-k)$ right triangle whose northeast vertex is immediately south of the point $s_i$. 
(See the shaded region for $\bp_3$ in Figure~\ref{subfig:facet example O}.)
By superimposing all of these regions inside a common $(n-k+1) \times (n-k+1)$ right triangle (aligned at the northeast), we convert $\cor(\f)$ into a strictly upper-triangular matrix in $\M_{n-k+1}(\mathbb{N})$ whose entries sum to $\#\cor(\f)$.  
Since strict chains in $\P_{\O}$ are antichains in $\P_{\Sp}$, the support of this matrix has height $\leq k$ in the poset $\P_{\Sp}(n-k+1)$. 
We thus have a bijection from $\scrF_k$ to the set of strictly upper-triangular matrices in $\M_{n-k+1}(\mathbb{N})$ whose multisupport has height $\leq k$.  
Again, surjectivity is shown via the antichain decomposition from~\eqref{D1 for O_k}, which decomposes the multisupport of a matrix into antichains $D_i \subset \P_{\Sp}(n-k+1)$, which recover the corners of the paths $\bp_i$.  
By Proposition~\ref{prop:RSK_Sp}, this bijection from $\scrF_k$ to the matrices of height $\leq k$ can be extended to the tableaux with even column lengths and with at most $k$ columns.
\end{proof}

\begin{corollary}
    \label{cor:O_k Gorenstein}
    Suppose that $k < n$.  Then the ring $\C[V^{n}]^{\O(V)}$ is Gorenstein if and only if $n - k$ is odd.
\end{corollary}

\begin{proof}
    If $n - k$ is odd, then each of the shapes $\nu$ in Proposition~\ref{prop:Hilbert series O_k} can be paired with its complement in a $(n-k+1) \times k$ rectangle.
    Therefore, just as in the argument in Corollary~\ref{cor:GL_k Gorenstein}, the numerator polynomial of the Hilbert--Poincar\'e series must be palindromic.
    Conversely, if $n-k$ is even, then each $\nu$ actually fits inside a $(n-k) \times k$ rectangle. Therefore the constant term in the numerator is $1$, while the leading coefficient is strictly greater than 1, since the dimension of the $\gl_{n-k+1}$-module with highest weight given by the rectangle with $n-k$ rows is strictly greater than 1.
\end{proof}

\subsection{The special orthogonal group}
\label{sub:SO_k Hilbert series}

Recall from~\eqref{decomp SO_k invariants} that the $\SO(V)$-invariants decompose into two components, namely the $\O(V)$-invariants and the $\O(V)$-semiinvariants:
\begin{equation}
    \label{SO_k invariant decomp again}
    \C[V^{n}]^{\SO(V)} \cong \C[V^{n}]^{\O(V)} \oplus \C[V^{n}]^{\O(V), \: \det}.
\end{equation}
Thus it remains to determine a Stanley decomposition for the module of $\O(V)$-semiinvariants. 

We convert Proposition~\ref{prop:SO_k graphs} from graphs into lattice paths; thus each basis element determines a pair $(\f',I) \in \scrF_k \times \Ceq{n}$. 
Because of the condition $I \leq A_1$ in~\eqref{basis R1}, and because of part (1) of Lemma~\ref{lemma:Ci and Di for O_k}, we can view $(\f',I)$ as a family $\f$ of lattice paths whose starting points are given by $(i_1,n), \ldots, (i_k,n)$, rather than the usual $(1,n), \ldots, (k,n)$.
In other words, we modify the facets in $\scrF_k$ by allowing the starting points of each path to vary along the right edge of $\P_{\O}$.
We write $\f \leftarrow I$ to express that $\f$ is a family of paths with starting points given by $I$.

 \begin{prop}
 \label{prop:Hilbert series O_k semis nonpure}
     The module of $\O(V)$-semiinvariants has the Stanley decomposition
     \[
     \C[V^{n}]^{\O(V), \det} = \bigoplus_{I \in \Ceq{n}} \bigoplus_{\f \leftarrow I} f_{\cor(\f)} \C[f_{ij} \mid (i,j) \in \f] \cdot {\textstyle \det_{I}},
     \]
     and the Hilbert--Poincar\'e series \begin{equation*}
         P\!\left(\C[V^{n}]^{\O(V), \det}; t\right) = 
         t^k \sum_{I \in \Ceq{n}} \sum_{\f \leftarrow I} \frac{(t^2)^{\#\cor(\f)}}{(1-t^2)^{\#\f}}.
     \end{equation*}
 \end{prop}

The form of the Hilbert--Poincar\'e series above is ``nonpure,'' in the sense that the exponent $\#\f$ in each summand depends on $I$.
This can be remedied by varying the \emph{end}points (along the diagonal), rather than the starting points, of the paths.
To avoid double-counting endpoints, we redefine the set $\cor(\f)$ to contain only non-diagonal corners. 

\begin{prop}
 \label{prop:Hilbert series O_k semis}
     Let $\cor(\f)$ denote the set of corners in $\f$ which are not of the form $(i,i)$.
     Then for $k \leq n$, the module of $\O(V)$-semiinvariants has the pure Hilbert--Poincar\'e series
     \begin{align}
         P\!\left(\C[V^{n}]^{\O(V), \det}; t\right) &= 
         t^k \cdot \frac{\sum_{\f} (t^2)^{\#\cor(\f)}}{(1-t^2)^{k(2n-k+1)/2}} \label{O_k semis Hilbert series paths}\\[2ex]
         &= \frac{\sum_\nu (\dim\F{\nu}{n-k+1}) \: t^{|\nu|}}{(1-t^2)^{k(2n-k+1)/2}}, \label{O_k semis Hilbert series tableaux}
     \end{align}
     where $\nu \in \Par((n-k+1) \times k)$ with columns of odd length.
 \end{prop}

\begin{proof}
The series~\eqref{O_k semis Hilbert series tableaux} is proved in \cite{EricksonHunziker23}*{Table 3, row IIIb}, in the context of the simple $\mathfrak{sp}(n,\mathbb{R})$-module with highest weight $((-\tfrac{k}{2})^{n-k}, (-\tfrac{k}{2})^k)$.  We will show the equality between~\eqref{O_k semis Hilbert series paths} and~\eqref{O_k semis Hilbert series tableaux}.

Let $T$ be a tableau in the alphabet $[n-k+1]$ having $k$ columns, all with odd length.  Then by~\eqref{RSK_GL on twins}, $\RSK_{\GL}(T,T)$ is a symmetric matrix with trace $k$, whose multisupport has height $k$.  
From now on we consider only the upper-triangular part of this matrix.
(Note that the sum of the entries in this upper-triangular part equals $k + |T|/2$.)
We decompose its multisupport into antichains $D_i$, as defined for $\RSK_{\GL}$ in~\ref{D1 for GL_k}; note that each $D_i$ contains exactly one diagonal element.  
Each $D_i$ thus determines a path $\bp_i$ when aligned inside the $(n-k+1) \times (n-k+1)$ right triangle below $s_i$ in the poset $\P_{\O} = \P_{\O}(n)$, as follows.  
(Note that this triangle extends just outside $\P_{\O}$. 
Note also that each $D_i$, in particular its subset of non-diagonal elements, is a strict chain in $\P_{\O}$.)  
Take $\cor(\bp_i)$ to be the non-diagonal elements of $D_i$, and take the endpoint of $\bp_i$ to be the point immediately above the diagonal element of $D_i$.  
In this way, each $T$ determines a unique facet $\f \in \scrF_k$.  
Moreover, this process is clearly invertible, using the same method of superimposing corners (and now endpoints) from our previous proofs.  
Hence by restricting $\RSK_{\GL}$ to tableau pairs $(T,T)$ with $k$ columns, all of odd length, we obtain a bijection between the set of such tableaux and $\scrF_k$.
As observed above, by part (1) of Proposition~\ref{prop:RSK_GL}, if $T$ has shape $\nu$, then $(|\nu|-k)/2$ equals the number of non-diagonal corners in the corresponding facet.
Hence the exponent $k + 2(\#\cor(\f))$ in~\eqref{O_k semis Hilbert series paths} equals the exponent $|\nu|$ in~\eqref{O_k semis Hilbert series tableaux}.
\end{proof}

The corollary below follows immediately from~\eqref{SO_k invariant decomp again}, \eqref{O_k Hilbert series}, and~\eqref{O_k semis Hilbert series tableaux}:

\begin{corollary}
    \label{cor:Hilbert series SO_k invariants}
    For $k \leq n$, the ring of $\SO(V)$-invariants has the Hilbert--Poincar\'e series
    \[
        P\!\left(\C[V^{n}]^{\SO(V)}; t\right) = \frac{\sum_\nu (\dim\F{\nu}{n-k+1}) \: t^{|\nu|}}{(1-t^2)^{k(2n-k+1)/2}},
     \]
     where $\nu \in \Par((n-k+1) \times k)$ has either all columns of odd length or all columns of even length.
\end{corollary}


\begin{rem}
    Just as for $\SL(V)$, it follows from~\cite{HochsterRoberts}*{Cor.~1.9} that the ring $\C[V^{n}]^{\SO(V)}$ is Gorenstein for all values of $\dim V$ and $n$.
    This is also easily seen from the fact that each shape $\nu$ above can be paired with its complement in a $(n - k + 1) \times k$ rectangle.
\end{rem}



\begin{ex}[$\SO_2$-invariants, where $n=4$]
\label{ex:SO invariants}

    \input{Hilbert_series_examples/k2n4}
    
\end{ex}

\subsection{The symplectic group}

Recall from Section~\ref{sec:shellings} that $\#\f = k(2n-2k-1)$ for each facet $\f \in \scrF_k \subset \Delta_k(\P_{\Sp})$.  
Combined with the Stanley decomposition of $\C[\Delta_k]$ in~\eqref{stanley decomp again}, this yields the Hilbert--Poincar\'e series
\begin{equation*}
    P\!\left(\C[z_{ij}]/\ker \pi^*; t\right) = 
    \frac{\sum_{ \f \in \scrF_{k} } t^{\# \cor(\f)}}{(1-t)^{k(2n-2k-1)}}.
 \end{equation*}
 The Hilbert--Poincar\'e series in the following proposition is found in~\cite{EHP}*{Cor.~8.5} and \cite{EricksonHunziker23}*{Table 3, row II}, in the context of the $k$th Wallach representation of $\mathfrak{so}^*(2n)$.  
 The combinatorial proof for the numerator polynomial is adapted from that in~\cite{Krattenthaler}*{Prop.~30 and Fig.~9}.

 \begin{prop}
 \label{prop:Hilbert series Sp_2k}
     The ring of $\Sp(V)$-invariants has the Stanley decomposition
    \[
    \C[V^{n}]^{\Sp(V)} = \bigoplus_{\f \in \scrF_k} f_{\cor(\f)} \C[f_{ij} \mid (i,j) \in \f].
    \]
    Furthermore, for $k \leq \lfloor n/2 \rfloor$, the Hilbert--Poincar\'e series is  
     \begin{equation*}
         P\!\left(\C[V^{n}]^{\Sp(V)}; t\right) = \frac{\sum_\nu (\dim\F{\nu}{n-2k-1}) \: t^{|\nu|}}{(1-t^2)^{k(2n-2k-1)}},
     \end{equation*}
     where $\nu \in \Par((n-2k-1) \times 2k)$ with rows of even length.
 \end{prop}

\begin{proof}
This time, we use the correspondence $\RSK_{\O}$ to establish a bijection between $\scrF_k$ and $\bigcup_\nu \SSYT(\nu, \: n-2k-1)$, which converts the number of corners into half the size of the corresponding tableau.  
(Note this time the combinatorial ``switch'' from $\P_{\Sp}(n)$ to $\P_{\O}(n-2k-1)$; see Section~\ref{sub:ES reduction}.)

Recall from Section~\ref{sec:shellings} that for each path $\bp_i \subseteq \f$, the corners in $\cor(\bp_i)$ form a strict chain within the $(n-2k-1) \times (n-2k-1)$ right triangle whose upper-left vertex is immediately below the point $s_i$.
(See the shaded region for $\bp_1$ in Figure~\ref{subfig:facet example Sp}.)
By superimposing all of these regions inside a common $(n-2k-1) \times (n-2k-1)$ right triangle, we convert $\cor(\f)$ into an upper-triangular matrix in $\M_{n-2k-1}(\mathbb{N})$ whose entries sum to $\#\cor(\f)$.  
Since strict chains in $\P_{\Sp}$ are antichains in $\P_{\O}$, the support of this matrix has height $\leq k$ in the poset $\P_{\O}(n-2k-1)$. 
We thus have a bijection from $\scrF_k$ to the set of upper-triangular matrices in $\M_{n-2k-1}(\mathbb{N})$ with height $\leq k$.  
Again, surjectivity is shown by decomposing the multisupport into the antichains $D_i$ from~\eqref{D1 for GL_k}, which recover the corners of the paths $\bp_i$.  
By Proposition~\ref{prop:RSK_O}, this bijection from $\scrF_k$ to the matrices of height $\leq k$ can be extended to the tableaux with even row lengths and with at most $2k$ columns.
\end{proof}

\begin{rem}
    As for $\SL(V)$ and $\SO(V)$, it follows from~\cite{HochsterRoberts}*{Cor.~1.9} that the ring $\C[V^{n}]^{\Sp(V)}$ is Gorenstein for all values of $\dim V$ and $n$.
    This is also easily seen from the fact that each shape $\nu$ above can be paired with its complement in a $(n-2k-1) \times 2k$ rectangle.
\end{rem}

\section{Howe duality and modules of covariants}
\label{sec:Howe duality and Enright reduction}

\subsection{Hermitian symmetric pairs}
\label{sub:Hermitian symmetric}

Suppose $G_\R$ is a connected real reductive group 
and $K_\R\subset G_\R$ is a maximal compact subgroup such that $G_\R/K_\R$ is an irreducible Hermitian symmetric space of the non-compact type. 
Let $\g_{\R} = \k_{\R} \oplus \p_{\R}$ be a 
Cartan decomposition of the Lie algebra of $G_\R$ and let 
$\g = \k\oplus \p$ be the corresponding decomposition
of the complexified Lie algebra. 
 From the general theory, there exists a distinguished element $h_0 \in \mathfrak{z}(\k)$ such that $\operatorname{ad} h_0$ acts on $\g$ with eigenvalues $0$ and $\pm 1$.  
 We thus have a triangular decomposition $\g = \p^- \oplus \k \oplus \p^+$, where $\p^{\pm} = \{ x \in \g \mid [h_0, x] = \pm x\}$.  
 The subalgebra $\q = \k \oplus \p^+$ is a maximal parabolic subalgebra of $\g$, with Levi subalgebra $\k$ and abelian nilradical $\p^+$.  
 Parabolic subalgebras of complex simple Lie algebras that arise in this way are called parabolic subalgebras of \emph{Hermitian type}, and $(\g,\k)$ is called a (complexified) \emph{Hermitian symmetric pair}.

Suppose $(\g,\k)$ is a Hermitian symmetric pair, and let $\h$ be a Cartan subalgebra of both $\g$ and $\k$.  
Let $\Phi$ be the root system of the pair $(\g,\h)$, and $\g_\al$ the root space corresponding to $\al \in \Phi$.  
Then put $\Phi(\k) \coloneqq \{ \al \in \Phi \mid \g_\al \subseteq \k\}$ and $\Phi(\p^+) = \{ \al \in \Phi \mid \g_\al \subseteq \p^+\}$.  
 Choose a set $\Phi^+$ of positive roots so that $\Phi(\p^+)\subseteq \Phi^+$, and let $\Phi^- = -\Phi^+$ denote the negative roots.  
 We write $\Phi^{\!+\!}(\k)$ for $\Phi^+ \cap \Phi(\k)$.  Let $\Pi 
 \Phi^+$ denote the set of simple roots $\al_i$.   
 We write $\langle \; , \; \rangle$ to denote the non-degenerate bilinear form on $\h^*$ induced from the Killing form of $\g$.  
 For $\al \in \Phi$, we write $\al^\vee \coloneqq 2\al / \langle \al,\al \rangle$.  
 We define the fundamental roots $\omega_i$ such that they form a basis of $\h^*$ dual to the $\al_i^\vee$, i.e., $\langle \omega_i,\al_j^\vee\rangle = \delta_{ij}$.  
 As usual, we let $\rho \coloneqq \frac{1}{2} \sum_{\al \in \Phi^+} \al$.  For the specific pairs $(\g,\k)$ of classical type, we choose the roots and weights spelled out explicitly in~\cite{EricksonHunziker23}*{\S3.3}.
Let 
\[
    \Lplusk \coloneqq \{ \la \in \h^* \mid \langle \la+\rho, \: \al^\vee\rangle \in \mathbb{Z}_{>0} \text{ for all } \al \in \Phi^{\!+\!}(\k)\}
\]
denote the set of dominant integral weights with respect to $\Phi^{\!+\!}(\k)$.
For $\la \in \Lambda^{{+}}(\k)$, let $F_\la$ denote the simple $\k$-module with highest weight $\la$.

Recall that two roots $\al, \be$ are called \emph{strongly orthogonal} if neither $\al + \be$ nor $\al - \be$ is a root.
We adopt Harish-Chandra's maximal set $\{\gamma_1, \ldots, \gamma_r\}$ of strongly orthogonal roots in $\Phi(\p^+)$, which are defined recursively as follows.
Let $\gamma_1$ be the lowest root in $\Phi(\p^+)$.
(Hence $\gamma_1 \in \Pi$.)
Then for $i > 1$, let $\gamma_i$ be the lowest root in $\Phi(\p^+)$ that is strongly orthogonal to each of $\gamma_1, \ldots, \gamma_{i-1}$.  For any partition $n_1 \geq \cdots \geq n_r \geq 0$, the weight $-\sum_i n_i \gamma_i$ is $\k$-dominant and integral.  
We have the following multiplicity-free decomposition due to Schmid~\cite{Schmid}:
\begin{equation}
    \label{Schmid decomp}
    \C[\p^+] \cong \bigoplus_{\mathclap{n_1 \geq \cdots \geq n_r \geq 0}} F_{-n_1 \gamma_1 -\cdots - n_r\gamma_r.}
\end{equation}
In particular, we have $\p^- \cong F_{-\gamma_1}$ as a $\k$-module.

Let $K$ be the complexification of $K_\R$ as described above.
Then $\k$ is the Lie algebra of $K$, and the adjoint action of $\k$ on
$\g$ exponentiates to a $K$-action.
In particular, $K$ acts on $\p^+$.
The $K$-orbits in $\p^+$ are  
$\scr{O}_0:=\{0\}$ and
$\scrO_k:=K\cdot(e_{\gamma_1}+\dots +e_{\gamma_k})$ for $1\leq k\leq r$, where $e_{\gamma_i}\in \g_{\gamma_i}$ denotes a root vector corresponding to $\gamma_i$.
The closures of the $K$-orbits in $\p^+$ form a chain of algebraic varieties
\[
\{0\} = \overline{\scrO}_0 \subset \overline{\scrO}_1 \subset \cdots \subset \overline{\scrO}_r = \p^+.
\]
For each $0\leq k \leq r$, the coordinate ring of $\overline{\scrO}_k$ can be viewed as a truncation of the Schmid decomposition~\cite{EHW}:
\begin{equation}
    \C[\overline{\scrO}_k] \cong \bigoplus_{\mathclap{n_1 \geq \cdots \geq n_k \geq 0}} F_{-n_1 \gamma_1 -\cdots - n_r\gamma_k.}
\end{equation}
If $(\g,\k)$ is one of the types in Table~\ref{table:Howe}, then $\C[\overline{\scrO}_k]$ is isomorphic to the coordinate ring of a determinantal variety $\M^{\leqslant k}_{p,q}$, $\SM^{\leqslant k}_{n}$, or $\AM^{\leqslant 2k}_{n}$, which appeared in Section~\ref{sec:invariant theory} above.


\subsection{Howe duality and modules of covariants}

For full details behind the facts presented in this section, we refer the reader to Roger Howe's paper~\cite{Howe89}.
In each of the three Howe duality settings we will describe, the story begins with a real Lie group $G_{\R}$, defined below:
\begin{align}
\U(p,q)&:=\left\{g\in \GL(p+q,\C) \mid g \begin{pmatrix}  I_p &0 \\   0 & -I_q\end{pmatrix}g^* =  \begin{pmatrix}  I_p &0 \\   0 & I_q\end{pmatrix}\right\},\\[5pt] \label{eq: Sp(2n,R)}
\Sp(2n,\R)&:=\left\{g\in \GL(2n,\C) \mid g \begin{pmatrix}  0 & I_n \\  -I_n& 0\end{pmatrix}g^t =  \begin{pmatrix}  0 & I_n \\  -I_n & 0\end{pmatrix}\right\}\cap\U(n,n),\\[5pt]
\O^*(2n) &:= \left\{g\in \GL(2n,\C)\mid g \begin{pmatrix}  0 & I_n \\  I_n & 0\end{pmatrix}g^t =  \begin{pmatrix}  0 & I_n \\   I_n & 0\end{pmatrix}\right\}\cap\U(n,n).
\end{align}
Note that 
$\U(p,q) \cap \U(p+q) \cong \U(p)\times \U(q)$,
$\Sp(2n,\R)\cap \U(2n) \cong \U(n)$ and 
$\O^*(2n)\cap \U(2n)\cong \U(n)$,
where in last two cases $\U(n)$ 
is embedded block-diagonally as follows: 
\[
\left\{\begin{pmatrix}
    a & 0 \\
    0 & (a^{-1})^{t}
\end{pmatrix}
\mid a\in \U(n)\right\}\cong \U(n).
\]
It follows that 
$\Sp(2n,\R)\subset \SL(2n,\C)$ 
and $\O^*(2n)\subset \SL(2n,\C)$. 
For this reason, many authors write $\SO^*(2n)$ to denote $\O^*(2n)$.
We also point out that our definition of $\Sp(2n,\R)$ given by \eqref{eq: Sp(2n,R)}
differs from the more standard definition 
by a Cayley transform as follows.
Let $\mathbf{c}\in \SL(2n,\C)$ be the matrix given by
\[
\mathbf{c}:=\frac{1}{\sqrt{2}}
\begin{pmatrix}I_n &\sqrt{-1}I_n\\\sqrt{-1}I_n& I_n
\end{pmatrix}.
\]
Then $\mathbf{c} \Sp(2n,\R)\mathbf{c}^{-1}\subseteq \SL(2n,\R)$ with equality if and only if $n=1$.

Let $\g_\R$ be the Lie algebra of of one of the Lie groups above, with complexification $\g$.
In each of the three cases above, ``Howe duality'' is a phenomenon by which $\g$ is naturally paired with one of the classical groups $H$ from the previous sections, in the following manner.
(See Table~\ref{table:Howe} for a summary of the data for each case of Howe duality.)

\begin{table}[ht]
    \centering
    \begin{tblr}{colspec={|Q[m,c]|Q[m,c]|Q[m,c]|Q[m,c]|Q[m,c]|Q[m,c]|},stretch=1.5}

\hline

$G_\R$& $(\g, K)$ & $H$ &   $\C[W]$ & $\sigma \in \Sigma$  &$\la \in \Lplusk$ \\

\hline[2pt]

$\operatorname{U}(p,q)$ & $(\gl_{p+q}, \GL_p \times \GL_q)$ & $\GL_k$ & $\C[\M_{p,k}\oplus\M_{q,k}]$ &  {$(\sigma^+, \sigma^-)$: \\ $\ell(\sigma^-) \leq p,$ \\ $\ell(\sigma^+) \leq q$} & $(-k-\overleftarrow{\sigma^{-}}; \: \overrightarrow{\sigma^+})$\\

\hline

$\operatorname{Mp}(2n,\R)$ & $(\sp_{2n}, \widetilde{\GL}_n)$ & $\O_k$ & $\C[\M_{n,k}]$ & {$\sigma'_1 + \sigma'_2 \leq k$, \\ $\ell(\sigma) \leq n$} & $-\dfrac{k}{2} - \overleftarrow{\sigma}$\\

\hline

$\O^*(2n)$ & $(\so_{2n}, \GL_n)$ & $\Sp_{2k}$ & $\C[\M_{n,2k}]$ & {$\ell(\sigma) \leq k$, \\ $\ell(\sigma) \leq n$} & $-k - \overleftarrow{\sigma}$ \\

\hline

\end{tblr}
    \caption{Data for Howe duality.  
    For $H = \GL_k$, we write $\sigma = (\sigma^+, \sigma^-) \coloneqq (\sigma^+_1, \sigma^+_2\ldots, 0, \ldots, 0, \ldots, -\sigma^-_2,-\sigma^-_1)$, where $\sigma^+$ and $\sigma^-$ are the partitions consisting of the positive and negative parts of $\sigma$, respectively.  
    We reverse a partition $\sigma = (\sigma_1, \ldots, \sigma_k)$ by writing $\overleftarrow{\sigma} \coloneqq (\sigma_k, \ldots, \sigma_1)$; in this setting, for the sake of contrast we write $\overrightarrow{\sigma} \coloneqq \sigma$.}
    \label{table:Howe}
\end{table}

Let $H$ be one of the three complex classical groups in Table~\ref{table:Howe}.  Then $H$ acts naturally on the space $W$, which is the same $W$ as in Section~\ref{sec:invariant theory}.  
Let $\mathcal{D}(W)$ be the Weyl algebra of polynomial-coefficient differential operators on $W$, and denote the subalgebra of $H$-invariant operators by $\mathcal{D}(W)^H$.  
For each of the groups $H$ in the table, there is a Lie algebra $\g_{\R}$ whose complexification $\g$ is isomorphic to a generating set for the algebra $\mathcal{D}(W)^H$.
Let $\omega : \g \rightarrow \mathcal{D}(W)^H$ denote this Lie algebra homomorphism whose image generates $\mathcal{D}(W)^H$.  
Following Howe, we call $(H, \g)$ a \emph{dual pair}.
Let $\g = \p^- \oplus \k \oplus \p^+$ be the triangular decomposition described above.
In each case, define an $H$-invariant polynomial map $\pi: W \longrightarrow \p^+$ as follows:
\begin{alignat*}{3}
    &(H = \GL_k) \qquad && \pi: \M_{p,k} \oplus \M_{q,k} \longrightarrow \M_{p,q}, \qquad && (X,Y) \longmapsto XY^t,\\
    &(H = \O_k) \qquad && \pi: \M_{n,k} \longrightarrow \SM_n, \qquad && X \longmapsto X X^t,\\
    &(H = \Sp_{2k}) \qquad && \pi: \M_{n,2k} \longrightarrow \AM_n, \qquad && X \longmapsto X \left(\begin{smallmatrix}0 & I_k \\ -I_k & 0 \end{smallmatrix}\right) X^t.
    \end{alignat*}
Then $\pi(W) \subset \p^+$ is either all of $\p^+$ (if $k \geq r$), or else is the determinantal variety $\M^{\leqslant k}_{p,q}$, $\SM^{\leqslant k}_{n}$, or $\AM^{\leqslant 2k}_{n}$, respectively.
This induces the comorphism $\pi^* : \C[\p^+] \longrightarrow \C[W]$ given by $z_{ij} \mapsto f_{ij}$, where the $z_{ij}$ are the standard coordinates on $\p^+$.  
As in Section~\ref{sec:invariant theory}, the SFT states that $\ker \pi^*$ is generated by the determinants of minors or Pfaffians.

It follows from all this that $\g$ acts on $\C[W]$ via its image $\omega(\g) \subset \mathcal{D}(W)^H$.  
Similarly, if $U$ is any representation of $H$,
then $\mathcal{D}(W)^H$ and hence $\g$
acts on the module of covariants 
$(\C[W]\otimes U)^H$.
Explicitly, the action of an invariant  differential operator $D\in \mathcal{D}(W)^H$ on $(\C[W]\otimes U)^H$ is given by
$D\cdot(\sum_i f_i\otimes u_i)=\sum_{i} (Df_i)\otimes u_i$. 
Furthermore, the $\k$-action integrates to an $K$-action and hence 
$(\C[W]\otimes U)^H$ can be viewed as a $(\g,K)$-module.

\begin{theorem}[Howe, Kashiwara--Vergne]
\label{thm:Howe duality}
We have the following multiplicity-free decomposition of $\C[W]$ as a $(\g,K) \times H$-module:
\begin{equation}
    \label{Howe decomp}
    \C[W] \cong \bigoplus_{\sigma \in \Sigma} 
    (\C[W]\otimes U_\sigma)^H
   \otimes U_{\sigma}^*,
\end{equation}
where $\Sigma \coloneqq \{ \sigma \in \widehat{H} \mid (\C[W]\otimes U_\sigma)^H \neq 0\}$.
Furthermore, there is an injective map $\Sigma \longrightarrow \Lplusk$, $\sigma \longmapsto \lambda$, such that as a $(\g,K)$-module
\begin{equation}
\label{covariants = L lambda}
 (\C[W]\otimes U_\sigma)^H \cong L_\la
 :=\mbox{simple $\g$-module with highest weight $\lambda$}
\end{equation}
and $L_\lambda$ is unitarizable with respect to $\g_\R$.
\end{theorem}

See the appendix for further details, where we give the action of $K \times H$ on $\C[W]$.
We also write down explicitly the Lie algebra homomorphism $\omega: \g \rightarrow \mathcal{D}(W)^H$, using a block matrix form that we find especially helpful in understanding the differential operators by which $\g$ acts on $\C[W]$.

\subsection{Stanley decompositions for modules of $\GL(V)$-covariants}

Here we treat the modules of $\GL(V)$-covariants in the case where $U_\sigma$ is a polynomial representation or its dual, i.e., where $\sigma^- = 0$ or $\sigma^+=0$.
(For much of the discussion we assume $\sigma^- = 0$, without loss of generality.)
By generalizing our combinatorial approach to the $\SL(V)$-invariants in Section~\ref{sub:Hilbert series SL_k}, we obtain a similar interpretation of the Hilbert--Poincar\'e series, as positive combinations taken over certain ``hybrid'' families of paths living in both $\Cleq{p}$ (or $\Cleq{q}$) and in $\P_{\GL}$.

\subsubsection{Two relations on $\Cleq{p}$ and $\Cleq{q}$}

We begin with a note aimed at unifying different ordering conventions in the literature.
Note that $p$ can just as well be replaced by $q$ throughout the following discussion.

For $A \in \Cleq{p}$, we write $\widetilde{A}$ to denote the element of $\Cleq{p}$ obtained from $A$ by replacing each entry $i$ with $p-i+1$.
Then define a new partial order $\preceq$ such that
\[
A \preceq B \Longleftrightarrow \widetilde{B} \leq \widetilde{A}.
\]
In other words, viewing $A$ and $B$ as adjacent \emph{bottom}-justified columns, we have $A \preceq B$ whenever we obtain an SSYT by rotating 180 degrees and substituting $i \mapsto p-i+1$.
Hence in general, $\leq$ and $\preceq$ do not coincide: for example, when $A$ is shorter than $B$, it is possible that $A \preceq B$, but impossible that $A \leq B$.

We introduce an important relation $A \rightarrow B$ which relaxes the ordering in $\Cleq{p}$, and which turns out to be the key difference in generalizing from invariants to covariants:

\begin{dfn}
\label{def:arrow}
    Let $A = (a_1, \ldots, a_r), \: B = (b_1, \ldots, b_s) \in \Cleq{p}$.
    Then we write $A \rightarrow B$ if and only if $a_i \leq b_{k-r+i}$ for all $i = 1, \ldots, r+s-k$.
\end{dfn}

More intuitively, adjoin $A$ and $B$ as columns and then slide $A$ downward until its bottom row lines up with the (possibly empty) $k$th row of $B$.
Then $A \rightarrow B$ if and only if the resulting skew tableau is semistandard.
For example, if $k=5$, then we have $(3,4,5) \rightarrow (1,2,4,4)$, since the skew tableau below is semistandard:
\[
k=5 \left\{ \; \ytableausetup{smalltableaux}
\ytableaushort{\none1,\none2,34,44,5}\right.
\]
If $k=4$, however, then the relation fails since the first column now slides up by one box.
Note that $A \leq B$ implies $A \rightarrow B$, although the example above shows that the converse is false.
If $\#A + \#B \leq k$, then $A \rightarrow B$ is vacuously true.
More importantly, note that $A \rightarrow B \Longleftrightarrow \widetilde{B} \rightarrow \widetilde{A}$.
Therefore we have
\begin{equation}
\label{two SM sequences}
X_1 \preceq \cdots \preceq X_m \rightarrow A_1 \leq \cdots \leq A_\ell \quad \text{ if and only if} \quad \widetilde{A}_\ell \preceq \cdots \preceq \widetilde{A}_1 \rightarrow \widetilde{X}_m \leq \cdots \leq \widetilde{X}_1.
\end{equation}
We will choose the form on the left-hand side to define our standard monomials in the upcoming section, in order to remain consistent with the first half of the paper, but we could just as well have used the form on the right-hand side (which actually lines up more nicely with our main reference~\cite{Jackson}).
Likewise, in our construction~\eqref{s-monomials SL_k} of the $\SL(V)$ standard monomials, we could have chosen to substitute $A \mapsto \widetilde{A}$ and reversed the inequalities everywhere.
But although this would indeed yield a different basis of standard monomials, nonetheless it is not hard to see that the effect on our lattice path diagrams would simply be a rotation by 180 degrees.

\begin{rem}
    The relations $\preceq$ and $\rightarrow$ were notably absent from our $\SL(V)$ constructions earlier in the paper, for the following reason.
    When we restrict our attention to the subposet $\Ceq{p}$ where all elements have length $k$, both relations $A \preceq B$ and $A \rightarrow B$ are equivalent to $A \leq B$ (in fact, $A \rightarrow B$ is equivalent to $A \leq B$ as long as \emph{one} of the two elements lies in $\Ceq{p}$).
    Hence for the $\SL(V)$-invariants, where the $X$'s (which were $I$'s in that setting) were elements of $\Ceq{p}$, there was no need to think about the relations $\prec$ or $\rightarrow$, since they could all be replaced by $\leq$ on either side of~\eqref{two SM sequences}.
\end{rem}

\subsubsection{Synopsis of the $\SL(V)$-invariants}

We begin by revisiting our results for the $\SL(V)$-invariants (Proposition~\ref{prop:Hilbert series SL_k}), this time focusing on the $\GL(V)$-stable components in the decomposition~\eqref{decomp SL invariants}.
These components are the modules of $\GL(V)$-semiinvariants $\C[V^{*p} \oplus V^q]^{\GL(V), \; \det^m}$, for each $m \in \mathbb{Z}$.
But the semiinvariants are a special case of covariants: as a $\GL(V)$-module, the character $\det^m$ is equivalent to the representation $U_{\sigma}$ where $\sigma = (m^k) = (m, \ldots, m)$.
Thus for fixed integer $m \geq 0$, we have $\sigma = (m^k)$, and one can summarize our previous results in terms of individual modules of covariants: we obtain the Stanley decomposition 
\begin{equation}
    \label{SL Stanley paths}
    (\C[V^{*p} \oplus V^q] \otimes U_\sigma)^{\GL(V)} =  \bigoplus_{\substack{\mathbf{C} \rightarrow \f: \\ 
    \#\mathbf{C} = m}} f_{\cor(\f)} \C[f_{ij} \mid (i,j) \in \f] \cdot \textstyle{\det_{\mathbf{C}^*}}
\end{equation}
which, in turn, yields the following Hilbert--Poincar\'e series (where our annotations foreshadow future notation):
\begin{equation}
\label{SL invariants as paths}
(t^k)^m \sum_{J \in \Ceq{q}} \underbrace{(\text{\# chains $\mathbf C \rightarrow J$ in $\Ceq{q}$ of size $m$})}_{\substack{\#\sigma_{J} \\ \text{(see Def.~\ref{def:bins})}}} \cdot \underbrace{\frac{\sum_{J \downarrow \f} (t^2)^{\# \cor(\f)}}{(1-t^2)^{\# \f}}}_{\substack{Q_J(t) \\ \text{(see~\eqref{P_I})}}}.
\end{equation}

Note that in each summand above, instead of the facets $\mathbf{F} \subset \Ceq{q}$ that appeared in the Stanley decomposition of the $\SL(V)$-invariants, we must instead consider chains $\mathbf{C} \subset \Ceq{q}$ of size $m$.
We briefly recall these combinatorial objects in the formulas above.
Each chain $\mathbf{C}$ in $\Ceq{p}$ (when $m > 0$; resp., $\Ceq{q}$ when $m < 0$) of size $|m|$ was depicted as a non-intersecting family of $k$ ``trails,'' all of length $|m|$, approaching $\P_{\GL}$ from the west (resp., north), with each trail allowed to move weakly south (resp., east) with each step toward the east (resp., south).
(Contrast this with the upcoming Figure~\ref{fig:mascot}, where we will allow fewer than $k$ trails, of different lengths.)
The $k$ trails end at distinct ``interpolation points'' along the western (resp., northern) edge of $\P_{\GL}$, and these points are encoded by an element $I \in \Ceq{p}$ (resp., $J \in \Ceq{q}$).
Our shorthand for this was $\mathbf{C} \rightarrow I$ (resp., $\mathbf{C} \downarrow J$).
We then take $I$ (or $J$) as the \emph{starting} points of the $k$ non-intersecting lattice paths in some family $\f \subset \P_{\GL}$, with the endpoints being the $k$ easternmost points along the southern edge (resp., the $k$ southernmost points along the eastern edge).
Our shorthand for this was $I \rightarrow \f$ (resp., $J \downarrow \f$).
Finally, we write $\mathbf{C} \rightarrow \f$ (resp., $\mathbf{C} \downarrow \f$) if there exists an $I$ such that $\mathbf{C} \rightarrow I \rightarrow \f$.


\subsubsection{Main result for $\GL(V)$-covariants of polynomial type}

We now turn toward generalizing the setting above, from the case $\sigma = (m^k)$ to the case where $\sigma$ is an arbitrary partition.

\begin{dfn}
    Let $\sigma = (\sigma^+, \sigma^-)$ as in Table~\ref{table:Howe}.
    If $\sigma^- = 0$ (resp., $\sigma^+ = 0$), then a \emph{$\sigma$-chain} is a chain $X_1 \preceq X_2 \preceq \cdots \preceq X_{m}$ in $\Cleq{q}$ (resp., $\Cleq{p}$) such that the SSYT with columns $\widetilde{X}_{m}, \ldots, \widetilde{X}_1$ has shape $\sigma^+$ (resp., $\sigma^-$).
    
    
\end{dfn}

As before, one can visualize a $\sigma$-chain as a family of $\ell(\sigma^{\pm})$ many ``trails'' whose lengths are the parts $\sigma^{\pm}_i$.
Note the obvious bijection between $\SSYT(\sigma,q)$ and the set of $\sigma$-chains in $\Cleq{q}$, given by applying $X \mapsto \widetilde{X}$ to each tableau column and then reversing the order.
Hence we can say more about the equivalence in~\eqref{two SM sequences}:
\begin{equation}
\label{two MS with tableau fact}
    \text{$X_1 \preceq \cdots \preceq X_m$ is a $\sigma$-chain in $\Cleq{q}$} \Longleftrightarrow \text{$\widetilde{X}_m \leq \cdots \leq \widetilde{X}_1$ is a tableau in $\SSYT(\sigma, q)$.}
\end{equation}
\begin{ex}
\label{ex:sigma chain}
In Figure~\ref{subfig:Mascot sigma minus zero}, the lengths of the trails in the $\sigma$-chain $\mathbf{C}$ reveal that $\sigma = (5,5,3,3,2)$.
Explicitly, regarding each row in the trails (from north to south) as an element of $\Cleq{q}$, we see that $\mathbf{C}$ is the $\sigma$-chain
\[
(2,6) \preceq (3,6) \preceq (1,3,5,7) \preceq (1,2,3,6,7) \preceq (1,2,4,6,7).
\]
Note that we can combine these chain elements as columns of a single skew tableau; then upon rotating and replacing each entry $i$ by $q-i+1$, we obtain an element of $\SSYT(\sigma,q)$.
Since $q=8$ in Figure~\ref{fig:mascot}, this construction is given as follows:
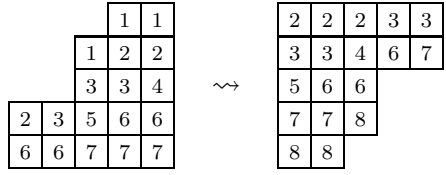
\[\ytableausetup{boxsize=1.1em}
\ytableaushort[\scriptstyle]{\none\none\none11,\none\none122,\none\none334,23566,66777}\quad \leadsto \quad
\ytableaushort[\scriptstyle]{22233,33467,566,778,88}
\]
In Figure~\ref{subfig:Mascot sigma plus zero}, the position of the trails (on the western edge) informs us that $\sigma^+ = 0$, and in particular the lengths of the trails reveal that $\sigma^- = (9,6,4)$.
As before, taking columns of dots (west to east) as elements of $\Cleq{p}$, we see that $\mathbf{C}$ is the $\sigma$-chain
\[
(1) \preceq (2) \preceq (5) \preceq (3,6) \preceq (4,7) \preceq (1,4,7) \preceq (1, 5,8) \preceq (2, 5,8) \preceq (2,6,9).
\]
Since $p=10$ in Figure~\ref{fig:mascot}, the tableau construction is given as follows:
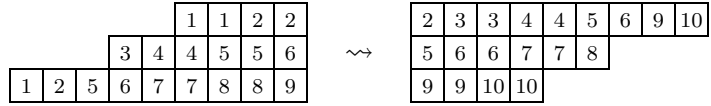
\[\ytableausetup{boxsize=1.1em}
\ytableaushort[\scriptstyle]{\none\none\none\none\none1122,\none\none\none344556,125677889}
\quad \leadsto \quad
\ytableaushort[\scriptstyle]{23344569{10},566778,99{10}{10}}
\]
\end{ex}

Next we will partition the set of $\sigma$-chains into subsets determined by their maximal elements.
Equivalently, one can imagine separating elements of $\SSYT(\sigma^-, p)$ into ``bins'' $\sigma_I$, where $T$ belongs to the bin $\sigma_I$ if and only if $T$ remains semistandard after prepending the column $\widetilde{I}$.
The pairs $\mathbf{C} \rightarrow \f$ defined below will ultimately index the Stanley decomposition of the covariants. 

\begin{dfn}
\label{def:bins}
If $\sigma^+ = 0$, let $\mathbf{C}$ be a $\sigma$-chain in $\Cleq{p}$, and let $I \in \Ceq{q}$.
We write $\mathbf{C} \rightarrow I$ if $I$ is the minimal element of $\Ceq{p}$ containing the maximal element of $\mathbf{C}$.
We define the set
\[
\sigma_I \coloneqq \{ \text{$\sigma$-chains $\mathbf{C}$} \mid \mathbf{C} \rightarrow I \}.
\]
We write $\mathbf{C} \rightarrow \f$ if there exists a $I$ (necessarily unique) such that $\mathbf{C} \rightarrow I \rightarrow \f$.
(This definition can be adapted for the case $\sigma^- = 0$ by replacing each instance of $p$ with $q$, and $I$ with $J$.)
\end{dfn}

    

In practice, one can easily determine set $\sigma_I$ by considering the \emph{initial string} in $I$, meaning the longest string of consecutive numbers starting with 1 (which may be empty).
We can now restate Definition~\ref{def:bins} as follows: $\mathbf{C} \rightarrow I$ if the maximal element of $\mathbf{C}$ is obtained by deleting (possibly zero) entries from the initial string of $I$.
For example, if $I = (1,2,3,5,6)$, then $\sigma_I$ contains precisely those $\sigma$-chains with the following maximal element(s):
\begin{align*}
\ell(\sigma) = 5 &\leadsto (1,2,3,5,6),\\
\ell(\sigma) = 4 &\leadsto (1,2,5,6) \text{ or } (1,3,5,6) \text{ or } (2,3,5,6),\\
\ell(\sigma) = 3 &\leadsto (1,5,6) \text{ or } (2,5,6) \text{ or } (3,5,6),\\
\ell(\sigma) = 2 &\leadsto (5,6),\\
\ell(\sigma) = 1 &\leadsto \sigma_I = \varnothing.
\end{align*}
Conversely, given a $\sigma$-chain $\mathbf{C} \subset \Cleq{p}$, we can determine the unique $I$ such that $\mathbf{C} \rightarrow I$ by repeatedly inserting the smallest unused entry into the maximal element of $\mathbf{C}$ until we obtain a sequence $I$ of length $k$.  
For example, in Figure~\ref{subfig:Mascot sigma plus zero}, the maximal element of $\mathbf{C}$ is $(2,6,9)$, and so we insert $1$ and $3$ to obtain $I = (1,2,3,6,9)$.

Since the union of the disjoint sets $\sigma_J$ (resp., $\sigma_I$) is the entire set of $\sigma$-chains in $\Cleq{q}$ (resp., $\Cleq{p}$), we have
\begin{align}
    \label{sum sigma_I dim}
    \begin{split}
    \sum_{J \in \Ceq{q}} \#\sigma_J &= \#\SSYT(\sigma^+, q) = \dim \F{\sigma^+}{q}, \qquad \sigma^- = 0,\\
    \sum_{I \in \Ceq{p}} \#\sigma_I &= \#\SSYT(\sigma^-, p) = \dim \F{\sigma^-}{p}, \qquad \sigma^+ = 0.
    \end{split}
\end{align}
We observe that
\begin{equation}
    \label{I J dim U_sigma}
    \#\sigma_{(1, \ldots, k)} = \#\SSYT(\sigma^{\pm},k) = \dim U_\sigma.
\end{equation}
In the case $\sigma^- = 0$, this is because $\mathbf{C} \rightarrow (1, \ldots, k)$ if and only if the SSYT obtained by substituting its elements $X \mapsto \widetilde{X}$ has all its entries in $\{q-k+1, \ldots, q-1, q\}$.
Upon subtracting $q-k$ from each entry, we obtain an element of $\SSYT(\sigma, k)$.

To summarize, each pair $(\mathbf{C}, \f)$ such that $\mathbf{C} \rightarrow \f$ can be viewed as a ``hybrid'' family of paths that begins as a $\sigma$-chain outside $\P_{\GL}$, and then becomes a family of $k$ lattice paths inside $\P_{\GL}$.
Specifically, we have $\mathbf{C} \rightarrow \f$ if and only if:
\begin{itemize}
    \item every trail in $\mathbf{C}$ joins the starting point of some lattice path in $\f$; and
    \item given the starting point of any unjoined lattice path in $\f$, every point between it and $(1,1)$ belongs to $\f$.
    (It suffices to check this for the unjoined path in $\f$ which starts furthest from $(1,1)$.)  
\end{itemize}
See Figure~\ref{fig:mascot} for an example.

Now that we have defined the required combinatorial objects, it remains to specify the covariants which they represent.
It will most convenient for us to view covariants as certain functions in $\C[W] \coloneqq \C[V^{*p} \oplus V^q]$, via the following identification.
Letting $N$ denote the maximal unipotent subgroup of $\GL(V)$, and considering the $N$-invariants in the decomposition~\eqref{Howe decomp}, we have
\begin{align*}
\C[W]^N &\cong \bigoplus_{\sigma \in \Sigma} (\C[W] \otimes U_\sigma)^{\GL(V)} \otimes (U^*_\sigma)^N \\
\C[W]^N_{\sigma} & \cong  (\C[W] \otimes U_\sigma)^{\GL(V)} \otimes \C u^*_{\sigma},
\end{align*}
where $\C[W]^N_\sigma$ is the subspace of $\sigma^*$-weight vectors for the maximal torus in $\GL(V)$, and $u^*_{\sigma}$ is a highest weight vector of $U_\sigma^*$.
This yields the isomorphisms
\begin{align*}
   (\C[W] \otimes U_\sigma)^{\GL(V)} \cong {\rm Hom}_{\GL(V)}(U_\sigma^*, \C[W]) &\cong \C[W]^N_{\sigma},\\
    \varphi &\mapsto \varphi(u^*_\sigma).
\end{align*}
Hence our main problem reduces to writing down a Stanley decomposition and Hilbert series for $\C[W]^N_\sigma$.

Fixing a basis for $V$ and the dual basis for $V^*$, let $v^*_{ij}$ denote the $i$th coordinate of $v^*_j$, and let $v_{ij}$ denote the $i$th coordinate of $v_j$.
We define the following $N$-invariant functions:
\begin{align*}
    \textstyle{\det_{X^*}}(v^*_1, \ldots, v^*_p, v_1, \ldots, v_q) &\coloneqq \det(v^*_{ij})_{i = 1, \ldots, \#X; \; j \in X,} && X \in \Cleq{p},\\
    \textstyle{\det_{Y}}(v^*_1, \ldots, v^*_p, v_1, \ldots, v_q) &\coloneqq \det(v_{ij})_{i = k+1-\#Y, \ldots, k; \; j \in Y,} && Y \in \Cleq{q}.
\end{align*}
Together with the quadratics $f_{ij}$ defined in~\eqref{f_ij for GL_k}, the functions $\det_{X^*}$ and $\det_Y$ generate the ring $\C[W]^N$; see~\cite{Jackson}*{Cor.~3.4.4}).

Now suppose that $\sigma^- = 0$ (resp., $\sigma^+ = 0$).
If $\mathbf{C}$ is a $\sigma$-chain in $\Cleq{q}$ (resp., $\Cleq{p}$), then we use the familiar shorthand
\begin{equation}
    \label{alpha beta}
    \displaystyle\textstyle \det_{\mathbf{C}} \coloneqq \displaystyle \prod_{Y \in \mathbf{C}} \textstyle{\det_Y}
    \qquad \text{or, respectively,} 
    \qquad
    \textstyle\det_{\mathbf{C}^*} \displaystyle \coloneqq \prod_{X \in \mathbf{C}} \textstyle{\det_{X^*}}.
\end{equation}
Note that in either case, $\det_{\mathbf{C}^*}$ or $\det_{\mathbf{C}}$ is a polynomial function of degree $|\sigma^{\pm}|$.

By~\cite{Jackson}*{Prop.~3.3.6}, a function $f \in \C[W]$ lies in $\C[W]_\sigma$ if and only if
\begin{equation}
    \label{sigma weight}
    \sigma_i =  (\text{degree of $f$ in the variables $v_{i\bullet}$}) - (\text{degree of $f$ in the variables $v^*_{i\bullet}$}), \qquad 1 \leq i \leq k.
\end{equation}
In this case we say that $f$ has weight $\sigma^*$.
It is straightforward to check that if $\mathbf{C}$ is a $\sigma$-chain, then $\det_{\mathbf{C}}$ has weight $\sigma^*$ (as does $\det_{\mathbf{C}^*}$ in the case $\sigma^+ = 0$).

For ease of notation, we  abbreviate the following rational expressions appearing earlier in our Hilbert--Poincar\'e series of the $\SL(V)$-invariants, for $I \in \Ceq{p}$ or $J \in \Ceq{q}$:
\begin{equation}
    \label{P_I}
P_I(t) \coloneqq \sum_{I \rightarrow \f} \frac{(t^2)^{\#\cor(\f)}}{(1-t^2)^{\#\f}} \qquad \text{and} \qquad Q_J(t) \coloneqq \sum_{J \downarrow \f} \frac{(t^2)^{\#\cor(\f)}}{(1-t^2)^{\#\f}}.
\end{equation}
Recall from Proposition~\ref{prop:Hilbert series SL_k} that $\#\f$ depends only on the sum $|I|$ of the elements of $I$, and is given by
\begin{equation}
    \label{size of f}
    \#\f = k\left(p+q-\frac{k-1}{2}\right)-|I|.
\end{equation}
The same holds in $Q_J(t)$ upon replacing $|I|$ by $|J|$.

\begin{rem}
    We observe a connection to earlier work by Herzog and Trung~\cites{HerzogTrung,Herzog}, in a different setting.
    In particular, our rational expressions $P_I(t)$ and $Q_J(t)$ appeared there as the Hilbert series of certain determinantal rings, where the defining ideal is generated by the minors of an order ideal of the set of bitableaux.
    In particular, $P_I(t)$ and $Q_J(t)$ are the Hilbert series of the order ideals generated by bitableaux $(I, (1,\ldots, k))$ and $((1, \ldots, k), J)$, respectively.
\end{rem}

\begin{theorem}
    \label{thm:Hilbert series GL_k covariants}
Let $U_\sigma$ be a polynomial representation of $\GL(V)$, or the dual representation thereof.
    Then the module of $\GL(V)$-covariants of type $U_\sigma$ has the following Stanley decomposition (where the sums range over the hybrid path families $\mathbf{C} \rightarrow \f$ defined in Definition~\ref{def:bins}):
    \begin{equation}
    \label{Stanley decomp GL covariants}
    (\C[V^{*p} \oplus V^q] \otimes U_\sigma)^{\GL(V)} \cong \begin{cases}
    \displaystyle
        \bigoplus_{\mathbf{C} \downarrow \f}  f_{\cor(\f)} \C[f_{ij} \mid (i,j) \in \f] \cdot \textstyle\det_{\mathbf{C}},& \sigma^- = 0,\\[4ex]
        \displaystyle
        \bigoplus_{\mathbf{C} \rightarrow \f} f_{\cor(\f)} \C[f_{ij} \mid (i,j) \in \f] \cdot \textstyle \det_{\mathbf{C}^*}, & \sigma^+ = 0.
    \end{cases}
\end{equation}
Furthermore, we have the following Hilbert--Poincar\'e series, with $\sigma_I$ and $\sigma_J$ as in Definition~\ref{def:bins}, and with $P_I(t)$ and $Q_J(t)$ as in~\eqref{P_I}:
\begin{equation}
    \label{GL covariants Hilbert series}
    P\Big((\C[V^{*p} \oplus V^q] \otimes U_\sigma)^{\GL(V)}; t\Big) = \begin{cases}
        \displaystyle
        t^{|\sigma^+|} \sum_{J \in \Ceq{q}} \#\sigma_J \cdot Q_J(t), & \sigma^- = 0,\\[4ex]
        \displaystyle
        t^{|\sigma^-|} \sum_{I \in \Ceq{p}} \#\sigma_I \cdot P_I(t), & \sigma^+ = 0.
    \end{cases}
\end{equation}
In either case, the sum of coefficients $\sum_I  \#\sigma_I$ or $\sum_J  \#\sigma_J$ equals the dimension of the simple $\k$-module $F_\la$, where $\sigma \mapsto \la$ is the map described in Theorem~\ref{thm:Howe duality}.
\end{theorem}

We defer the proof of Theorem~\ref{thm:Hilbert series GL_k covariants} to the end of the next subsection.
The following corollary is a special case of Corollary~9.2 in~\cite{NOT}, which follows immediately from the formulas in our theorem.
 
\begin{corollary}
\label{cor:NOT result}
    Let $(H, \g) = (\GL_k, \gl_{p+q})$, and suppose that $\sigma^- = 0$ or $\sigma^+ = 0$.
    Let $\sigma \mapsto \la$ be the map described in Theorem~\ref{thm:Howe duality}.
    Let \textup{``Deg''} denote the Bernstein degree.  
    Then 
    \[
    \operatorname{Deg} L_\la = \dim U_\sigma \cdot \operatorname{deg} \overline{\scrO}_k.
    \]
\end{corollary}

\begin{proof}
    Suppose, without loss of generality, that $\sigma^- = 0$.
    Recall that $Q_{(1, \ldots, k)}(t)$ is the Hilbert--Poincar\'e series of $\overline{\scrO}_k$.
    By~\eqref{size of f}, the exponent $\#\f$ is strictly greater in $Q_{(1,\ldots,k)}(t)$ than it is in $Q_J(t)$ for any $J \neq (1, \ldots k)$.
    Hence in the reduced form of the Hilbert series~\eqref{GL covariants Hilbert series}, the numerator equals 
    \[
    \text{$\#\sigma_{(1,\ldots, k)} \,\cdot$ numerator of $Q_{(1,\ldots, k)}(t) + $ terms containing the factor $(1-t^2)$}.
    \]
    Evaluating this numerator at $t=1$, we obtain $\#\sigma_{(1, \ldots, k)}$ times the Bernstein degree of $\overline{\scrO}_k$.
    The result follows from~\eqref{I J dim U_sigma}.
\end{proof}

Note that Corollary~\ref{cor:NOT result} gives the following combinatorial interpretation: the Bernstein degree of $L_\la$ equals the number of families of hybrid paths $\mathbf{C} \rightarrow (1, \ldots, k)\rightarrow \f$ such that $\mathbf{C}$ is a $\sigma$-chain. 

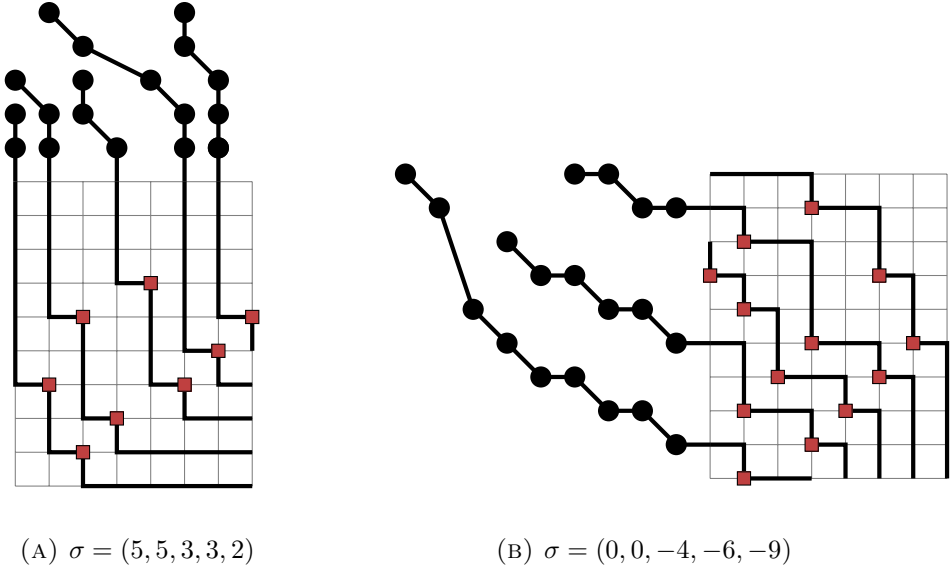
\begin{figure}
    \centering
        \begin{subfigure}[b]{0.4\textwidth}
         \centering
         \tikzstyle{corner}=[rectangle,draw=black,fill=red!50!gray, minimum size = 5pt, inner sep=0pt]
\tikzstyle{endpt}=[circle,fill=black, text=white, font=\tiny\sffamily\bfseries, minimum size = 8pt, inner sep=0pt]
\tikzstyle{blankpoint}=[circle,draw, minimum size = 8pt, inner sep=0pt]

\begin{tikzpicture}[scale=.45,baseline]
\draw[gray] (1,1) grid (8, 10);

\draw[ultra thick] (7,10) -- (7,6) -- (8,6) -- (8,5);
\draw[ultra thick] (6,10) -- (6,5) -- (7,5) -- (7,4) -- (8,4);
\draw[ultra thick] (4,10) -- (4,7) -- (5,7) -- (5,4) --(6,4) -- (6,3) -- (8,3);
\draw[ultra thick] (2,10) -- (2,6) -- (3,6) -- (3,3) -- (4,3) -- (4,2) -- (8,2);
\draw[ultra thick] (1,10) -- (1,4) -- (2,4) -- (2,2) -- (3,2) -- (3,1) -- (8,1);

\node at (1,11) [endpt] {};
\node at (2,11) [endpt] {};
\node at (4,11) [endpt] {};
\node at (6,11) [endpt] {};
\node at (7,11) [endpt] {};

\node at (2,4) [corner] {};
\node at (3,2) [corner] {};
\node at (3,6) [corner] {};
\node at (4,3) [corner] {};
\node at (5,7) [corner] {};
\node at (6,4) [corner] {};
\node at (7,5) [corner] {};
\node at (8,6) [corner] {};

\draw[ultra thick] (1,10) -- (1,11) -- (1,12) ;
\node at (1,12) [endpt] {};

\draw[ultra thick] (2,10) -- (2,11) -- (2,12) -- (1,13) ;
\node at (2,12) [endpt] {};
\node at (1,13) [endpt] {};

\draw[ultra thick] (4,10) -- (4,11) -- (3,12) -- (3,13);
\node at (3,12) [endpt] {};
\node at (3,13) [endpt] {};

\draw[ultra thick] (6,10) -- (6,11) -- (6,12) -- (5,13) -- (3,14) -- (2,15);
\node at (6,12) [endpt] {};
\node at (5,13) [endpt] {};
\node at (3,14) [endpt] {};
\node at (2,15) [endpt] {};

\draw[ultra thick] (7,10) -- (7,11) -- (7,13) -- (6,14) -- (6,15);
\node at (7,11) [endpt] {};
\node at (7,12) [endpt] {};
\node at (7,13) [endpt] {};
\node at (6,14) [endpt] {};
\node at (6,15) [endpt] {};

\end{tikzpicture}
         \caption{$\sigma = (5,5,3,3,2)$}
         \label{subfig:Mascot sigma minus zero}
     \end{subfigure}
    \begin{subfigure}[b]{0.4\textwidth}
         \centering
         \tikzstyle{corner}=[rectangle,draw=black,fill=red!50!gray, minimum size = 5pt, inner sep=0pt]
\tikzstyle{endpt}=[circle,fill=black, text=white, font=\tiny\sffamily\bfseries, minimum size = 8pt, inner sep=0pt]
\tikzstyle{blankpoint}=[circle,draw, minimum size = 8pt, inner sep=0pt]

\begin{tikzpicture}[scale=.45, baseline=(current bounding box.south)]
\draw[gray] (1,1) grid (8, 10);

\draw[ultra thick] (1,10) -- (4,10) -- (4,9) -- (6,9) -- (6,7) -- (7,7) -- (7,5) -- (8,5) -- (8,1);
\draw[ultra thick] (1,9) -- (2,9) -- (2,8) -- (4,8) -- (4,5) -- (6,5) -- (6,4) -- (7,4) -- (7,1);
\draw[ultra thick] (1,8) -- (1,7) -- (2,7) -- (2,6) --(3,6) -- (3,4) -- (5,4) -- (5,3) -- (6,3) -- (6,1);
\draw[ultra thick] (1,5) -- (2,5) -- (2,3) -- (3,3) -- (4,3) -- (4,2) -- (5,2) -- (5,1);
\draw[ultra thick] (1,2) -- (2,2) -- (2,1) -- (4,1);

\node at (0,9) [endpt] {};
\node at (0,5) [endpt] {};
\node at (0,2) [endpt] {};

\node at (4,9) [corner] {};
\node at (6,7) [corner] {};
\node at (7,5) [corner] {};
\node at (2,8) [corner] {};
\node at (4,5) [corner] {};
\node at (6,4) [corner] {};
\node at (1,7) [corner] {};
\node at (2,6) [corner] {};
\node at (3,4) [corner] {};
\node at (5,3) [corner] {};
\node at (2,3) [corner] {};
\node at (4,2) [corner] {};
\node at (2,1) [corner] {};

\draw[ultra thick] (1,9) -- (0,9) -- (-1,9) -- (-2,10) -- (-3,10);
\node at (-1,9) [endpt] {};
\node at (-2,10) [endpt] {};
\node at (-3,10) [endpt] {};

\draw[ultra thick] (1,5) -- (0,5) -- (-1,6) -- (-2,6) -- (-3, 7) -- (-4, 7) -- (-5, 8);
\node at (-1,6) [endpt] {};
\node at (-2,6) [endpt] {};
\node at (-3,7) [endpt] {};
\node at (-4, 7) [endpt] {};
\node at (-5,8) [endpt] {};

\draw[ultra thick] (1,2) -- (0,2) -- (-1,3) -- (-2,3) -- (-3,4) -- (-4,4) -- (-5,5) -- (-6,6) -- (-7,9) -- (-8, 10);
\node at (-1,3) [endpt] {};
\node at (-2,3) [endpt] {};
\node at (-3,4) [endpt] {};
\node at (-4, 4) [endpt] {};
\node at (-5, 5) [endpt] {};
\node at (-6,6) [endpt] {};
\node at (-7,9) [endpt] {};
\node at (-8,10) [endpt] {};

\end{tikzpicture}
         \caption{$\sigma = (0, 0, -4, -6, -9)$}
         \label{subfig:Mascot sigma plus zero}
     \end{subfigure}
    
    \caption{Typical hybrid path families $\mathbf{C} \rightarrow \f$ indexing the components of the Stanley decomposition~\eqref{Stanley decomp GL covariants}.
    Figure~\ref{subfig:Mascot sigma minus zero} shows the case where $\sigma^- = 0$, and Figure~\ref{subfig:Mascot sigma plus zero} where $\sigma^+ = 0$.
    In each case, the $\sigma$-chain $\mathbf{C}$ is depicted as the family of trails outside the $p \times q$ rectangular grid (which represents $\P_{\GL}$), and $\f$ is depicted by the family of lattice paths inside the grid.
    Note that the lengths of the trails give the coordinates of $\sigma$.
    We observe that $k=5$ (the number of lattice paths in $\f$), and $p=10$ and $q=8$ (the dimensions of the grid).
    Note that in Figure~\ref{subfig:Mascot sigma plus zero}, $\ell(\sigma^-) = 3$ is strictly less than $k$, and so two of the paths in $\f$ are not joined to a trail in $\mathbf{C}$.
    To verify that $\mathbf{C} \rightarrow \f$, we check that $\f$ contains every point on the western edge of $\P_{\GL}$ lying north of the lower unjoined path.}
    \label{fig:mascot}
\end{figure}

\subsubsection{A linear basis of standard monomials}

In this subsection, we translate certain results of~\cite{Jackson} into a form that is amenable to our lattice path approach.
We then use this framework to prove Theorem~\ref{thm:Hilbert series GL_k covariants}.

The main result of~\cite{Jackson} realizes $\C[W]^N$ as a \emph{Gr\"obner algebra}, in the following way.
The polynomial ring $\C[\{f_{ij}\}, \{\det_{X^*}\}, \{\det_{Y}\}]$ is equipped with a monomial ordering (see~\cite{Jackson}*{Def.~3.8.22}).
There exists a canonical projection $\C[\{f_{ij}\}, \{\det_{X^*}\}, \{\det_{Y}\}] \longrightarrow \C[W]^N$, such that the images of the monomials lying outside the kernel $\mathcal{I}$ are linearly independent.
These images are called \emph{standard monomials}, and form a linear basis for $\C[W]^N \cong \C[\{f_{ij}\}, \{\det_{X^*}\}, \{\det_{Y}\}] / \mathcal{I}$.
Moreover, there is a \emph{straightening law}, i.e., a relation for each minimal generator of $\mathcal{I}$ expressing it as a linear combination of strictly lesser monomials~\cite{Jackson}*{eqn.~(3.8.21) and (3.8.26)}.
The ideal $\mathcal{I}$ is generated by the finite set containing the following monomials~\cite{Jackson}*{Thm.~3.6.17}:
\begin{itemize}
    \item $\det_{X^*}\det_{X'^*}$, where $X$ and $X'$ are incomparable in $\Cleq{p}$;
    \item $\det_{Y}\det_{Y'}$, where $Y$ and $Y'$ are incomparable in $\Cleq{q}$;
    \item $\det_{X^*} \det_Y$, where $\#X+\#Y > k$;
    \item the monomials of \emph{$\GL_k$-splits}, which we define below.
\end{itemize}

\begin{dfn}[see \cite{Jackson}*{Def.~3.6.15}]
\label{def:split}
    A \emph{$\GL_k$-split} is an arrangement of positive integers taking the form
\[
\arraycolsep=1.5pt
\begin{array}{rcl}
   & a^*_1 & \\
& \vee & \\
    & \vdots & \\
   & \vee & \\
& a^*_s & \\
   & \vee & \\
   b_r > \cdots > b_1 \geq c^*_1 > \cdots > c^*_r > & d^*_1 & > \cdots > d^*_t
\end{array}
\hspace{1cm}
\begin{array}{rcl}
   & a_1 & \\
& \vee & \\
    & \vdots & \\
   & \vee & \\
& a_v & \\
   & \vee & \\
   b^*_u > \cdots > b^*_1 > c_1 > \cdots > c_u > & d_1 & > \cdots > d_w,
   \end{array}
\]
and subject to the following conditions:
\begin{enumerate}
    \item the starred entries are elements of $[p]$, and the unstarred entries are elements of $[q]$;
    \item $r \geq s$ and $u \geq v$;
    \item $s + t \leq k$ and $v + w \leq k$;
    \item \label{sum k+1} $r+t+u+w = k+1$;
    \item $r + u > 0$;
    \item \label{split condition} $b_1 > c_1$ and $b^*_1 > c^*_1$.
\end{enumerate}
(The parameters $r,s,t,u,v,w$ are allowed to be zero, meaning that the corresponding sequences are empty.)
The \emph{monomial} of a split is the following product in $\C[\{f_{ij}\}, \{\det_{X^*}\}, \{\det_{Y}\}]$:
 \begin{equation}
    \label{monomial of split}
 \prod_{i=1}^r f_{c^*_i, b_i} \prod_{i=1}^u f_{b^*_i, c_i} \cdot \textstyle \det_{(d^*_t, \ldots, d^*_1, a^*s, \ldots, a^*_1)^*} \det_{(d_w, \ldots, d_1, a_v, \ldots, a_1)}.
 \end{equation}
\end{dfn}
In this paper, we say that a $\GL_k$-split is of \emph{negative type} if $v=w=0$, since the weight $\sigma$ of its associated monomial is such that $\sigma^+ = 0$.
\begin{lemma}
    \label{lemma: split to tableaux} 
Each $\GL_k$-split of negative type gives rise to a triple of columns of strictly increasing integers (with the first two columns joined together), taking one of the two forms below:
\[
\ytableausetup{boxsize = normal}
\textup{length = $k+1$} \; \left\{\begin{ytableau}
c^*_r & a^*_s \\
\raisebox{-2pt}{\vdots} & \raisebox{-2pt}{\vdots} \\
\raisebox{-2pt}{\vdots} & \raisebox{-2pt}{\vdots} \\
\raisebox{-2pt}{\vdots} & \raisebox{-2pt}{\vdots} \\
\raisebox{-2pt}{\vdots} & \raisebox{-2pt}{\vdots} \\
\raisebox{-2pt}{\vdots} & a^*_1 \\
\raisebox{-2pt}{\vdots} \\
c^*_1\\
b^*_1\\
\raisebox{-2pt}{\vdots}\\
b^*_u
\end{ytableau}\right.
\qquad
\begin{ytableau}
    c_u\\
    \raisebox{-2pt}{\vdots}\\
    \raisebox{-2pt}{\vdots}\\
    \raisebox{-2pt}{\vdots}\\
    c_1\\
    b_1\\
    \raisebox{-2pt}{\vdots}\\
    \raisebox{-2pt}{\vdots}\\
    \raisebox{-2pt}{\vdots}\\
    \raisebox{-2pt}{\vdots}\\
    b_r
\end{ytableau}
\qquad
\text{or}
\qquad
\textup{length = $k$} \; \left\{\begin{ytableau}
\none & d^*_t \\
\none & \raisebox{-2pt}{\vdots} \\
c^*_r & d^*_1 \\
\raisebox{-2pt}{\vdots} & a^*_s \\
\raisebox{-2pt}{\vdots} & \raisebox{-2pt}{\vdots} \\
\raisebox{-2pt}{\vdots} & a^*_1 \\
\raisebox{-2pt}{\vdots} \\
c^*_1\\
b^*_1\\
\raisebox{-2pt}{\vdots}\\
b^*_u
\end{ytableau}\right.
\qquad
\begin{ytableau}
    \none\\
    \none\\
    c_u\\
    \raisebox{-2pt}{\vdots}\\
    \raisebox{-2pt}{\vdots}\\
    c_1\\
    b_1\\
    \raisebox{-2pt}{\vdots}\\
    \raisebox{-2pt}{\vdots}\\
    \raisebox{-2pt}{\vdots}\\
    b_r
\end{ytableau}
\]
In the left-hand case, the second column may be empty.
In the right-hand case, the topmost two-box row in the skew shape necessarily exists and violates the semistandard condition.

Conversely, every column triple with one of these two forms --- where the first two columns have entries in $[p]$ and the third column has entries in $[q]$ --- arises from a unique $\GL_k$-split of negative type.
\end{lemma}

\begin{proof}
    The fact that each column is strictly increasing, along with the fact that $c^*_r > d^*_1$, follows directly from the defining inequalities and condition~(\ref{split condition}) of Definition~\ref{def:split}.
    Since $w=0$, condition~(\ref{sum k+1}) of Definition~\ref{def:split} forces $r+t+u=k+1$.
    Hence if $t = 0$, then $r+u = k+1$, and we obtain the non-skew two-column shape shown above on the left-hand side; the second column is empty if $s=0$, but the first and third columns necessarily have length $k+1$.
    If $t >0$, then we obtain a two-column skew shape of length $k$, shown above on the right-hand side; in this case, the two columns necessarily overlap in at least one row, namely the row $c^*_r > d^*_1$ (or, if $r=0$, the row $b_1 > d^*_1$).
    
    For the converse, first suppose that we are given a pair of columns, each strictly increasing with entries in $[p]$, where the first column has length $k+1$ and the second has length between $0$ and $k+1$; suppose we also are given a third column with entries in $[q]$ and with length $k+1$.
    Then we recover the $\GL_k$-split as follows.
    The $a^*_i$ are the entries in the second column, and there are no $d^*_i$.
    To recover the rest of the elements in the split, it suffices to determine the parameter $r$, which is the unique number such that the $r$th entry from the top of the first column is greater than the $r$th entry from the bottom of the third column (if none exists, then $r=0$), and such that the $(r+1)$th entry from the top of the first column is less than the $(r+1)$th entry from the bottom of the third column (if none exists, then $u = 0$).

    Finally, suppose we are given a two-column skew shape with length $k$ and entries in $[p]$, with at least one two-box row, such that the topmost two-box row violates semistandardness; suppose we also are given a third column of length $\leq k$ with entries in $[q]$.
    It is clear from the picture above how to recover the $d^*_i$ and $a^*_i$ from the second column of the skew shape.
    Likewise, we recover the parameter $r$ as in the case above.
    The violation of semistandardness ensures that $d^*_1$ will be less than the element to its left in Definition~\ref{def:split}.
\end{proof}

The straightening law~\cite{Jackson}*{(3.8.21)} allows one to write the monomial of a $\GL_k$-split as a signed sum of monomials which are stictly lesser in the monomial order.
In the case of a $\GL_k$-split of negative type, we can picture this straightening law in terms of the right-hand side of the picture in Lemma~\ref{lemma: split to tableaux}: these lesser monomials are obtained from the split by all possible ways of exchanging any number of entries in the second column with entries in the first column.
Note that in each resulting configuration, upon reordering the entries in each column to strictly increase,  semistandardness is restored in the topmost two-box row (although the rows beneath may not be semistandard).
By removing all corresponding box-pairs in the first and third columns which lie above the topmost semistandard violation, we divide by a product of $f_{ij}$'s to obtain the monomial of yet another split.
By repeating this straightening procedure a finite number of times, we eventually express the original monomial of a $\GL_k$-split as a combination of monomials which are not divisible by any split monomial, and hence lie outside $\mathcal{I}$.

If, in the right-hand side of the picture in Lemma~\ref{lemma: split to tableaux}, we call the first and second columns $A$ and $X$, respectively, then we observe that the straightening procedure ultimately yields pairs $(A', X')$ whose overlap is semistandard.
But by Definition~\ref{def:arrow}, this just means that $A' \rightarrow X'$.
This leads to the following characterization of a basis of standard monomials for $\C[V^{*p} \oplus V^q]^N_\sigma$:

\begin{lemma}
    \label{lemma: SM basis for GL covariants}

Suppose $\sigma^+ = 0$.  Let $f_{AB}$ be as defined in~\eqref{f_AB}.
Then a basis of standard monomials for $\C[V^{*p} \oplus V^q]^N_{\sigma}$ is given by monomials of the form
\[
f_{AB} \cdot \textstyle\det_{\mathbf{C}^*},
\]
where $\mathbf{C}$ is a $\sigma$-chain in $\Cleq{p}$ such that
\begin{equation}
    \label{SM form in GL lemma}
\underbrace{X_1 \preceq \cdots \preceq X_m}_{\mathbf{C}} \rightarrow \underbrace{A_1 \leq \cdots \leq A_\ell}_A \qquad \text{and} \qquad \underbrace{B_1 \leq \cdots \leq B_\ell}_B,
\end{equation}
and where $A_i \in \Cleq{p}$ and $B_i \in \Cleq{q}$.  
In the case $\sigma^- = 0$, standard monomials take the form
\[
f_{AB} \cdot \textstyle\det_{\mathbf{C}},
\]
where $\mathbf{C}$ is a $\sigma$-chain in $\Cleq{q}$ such that
\[
\underbrace{A_1 \leq \cdots \leq A_\ell}_A \qquad \text{and} \qquad \underbrace{Y_1 \preceq \cdots \preceq Y_m}_{\mathbf{C}} \rightarrow \underbrace{B_1 \leq \cdots \leq B_\ell}_B.
\]
\end{lemma}

\begin{proof}
    Both cases are proved identically, so we will prove the case $\sigma^+ = 0$.
    We will actually show that the conventions in the discussion above yield a different basis, consisting of monomials of the form
    \[
    f_{AB} \cdot \prod_i \textstyle\det_{X_i^*},
    \]
    such that the $X_i$ form the columns of an element of $\SSYT(\sigma,p)$, and
    \[
    A_1 \preceq \cdots \preceq A_\ell \rightarrow X_1 \leq \cdots \leq X_m \qquad \text{and} \qquad B_1 \preceq \cdots \preceq B_\ell.
    \]
    But by~\eqref{two SM sequences} and~\eqref{two MS with tableau fact}, this fact implies the statement given in the lemma.

    We recall the four types of generators of $\mathcal{I}$ given earlier.
    The inequalities $X_1 \leq \cdots \leq X_m$ follow from the usual straightening laws for the generators $\det_{X^*} \det_{X'^*}$ where $X$ and $X'$ are incomparable in $\Cleq{p}$.
    The generators $\det_{Y^*} \det_{Y'^*}$ and $\det_{X^*} \det_{Y^*}$ ($\# X + \#Y > k$) are not relevant here, except to force $\#X \leq k$ in the quotient by $\mathcal{I}$, so that each $X_i \in \Cleq{p}$.  By~\eqref{sigma weight}, in order for $\prod_i \det_{A_i^*}$ to lie in $\C[V^{*p} \oplus V^q]_\sigma$, the chain $A_i$ must in fact form a $\sigma$-chain.
    The only remaining generators of $\mathcal{I}$ to consider are the monomials of $\GL_k$-splits of negative type.
    From the first picture in Lemma~\ref{lemma: split to tableaux}, we see that in the quotient we must have each $A_i \in \Cleq{p}$ and $B_i \in \Cleq{q}$.
    Again, the usual straightening relations in those posets yield the stated inequalities among the $A_i$ and $B_i$.
    From the discussion above on the straightening law for $\GL_k$-split monomials, we must have $A_\ell \rightarrow X_1$.    
\end{proof}

\begin{rem}
    Although this lemma is somewhat implicit in Jackson's work, he instead indexes the standard monomials by a construction he calls ``$\GL_k$-sequences''~\cite{Jackson}*{Def.~3.6.8}.
    These sequences, defined by filling and bumping entries in certain oscillating tableaux, have a much different flavor than our own sequences defined in terms of the three relations $\preceq$, $\rightarrow$, and $\leq$.
    More importantly, as we are about to prove, our approach allows us to index the Stanley decomposition~\eqref{Stanley decomp GL covariants} by the combinatorial objects $\mathbf{C} \rightarrow \f$.
\end{rem}

\begin{proof}[Proof of Theorem~\ref{thm:Hilbert series GL_k covariants}]
As above, we assume that $\sigma^+ = 0$.
In order to prove the Stanley decomposition, we need to show that each standard monomial $f_{AB} \cdot \det_{\mathbf{C}^*}$ satisfying the conditions~\eqref{SM form in GL lemma} lies in a unique component of~\eqref{Stanley decomp GL covariants} indexed by some $\mathbf{C} \rightarrow \f$.
By comparing~\eqref{SM form in GL lemma} and~\eqref{Stanley decomp GL covariants}, we see that the $\mathbf{C}$'s must coincide.
Let $I$ be the unique element of $\Cleq{p}$ such that $\mathbf{C} \rightarrow I$.
By part (1) of Lemma~\ref{lemma:Ci and Di for GL_k}, the support of the degree matrix of $f_{AB}$ decomposes into chains $C^*_i$ with the following property: the lowest row index of $C^*_i$ equals the $i$th entry of $A_1$.
Since $I \leq A_1$, the support of $f_{AB}$ is contained in some family $\f$ such that $I \rightarrow \f$.
By~\cite{HerzogTrung}*{Thm.~4.9}, the set $\{ \f \mid I \rightarrow \f\}$ is the set of facets of a shellable subcomplex of $\Delta_k$, such that the restrictions are the usual corners in this paper; hence this $\f$ is in fact unique.
Therefore $\f$ is the unique family such that $I \rightarrow \f$ and such that $f_{AB}$ lies in $f_{\cor(\f)}\C[f_{ij} \mid (i,j) \in \f]$.
This proves the decomposition~\eqref{Stanley decomp GL covariants}.

The Hilbert--Poincar\'e series~\eqref{GL covariants Hilbert series} follows immediately from this Stanley decomposition, since $\det_{\mathbf{C}^*}$ has degree $|\sigma^-|$, and since by definition $\sigma_I = \{ \mathbf{C} \mid \mathbf{C} \rightarrow I \}$.
Finally, by~\eqref{sum sigma_I dim}, we have $\sum_I \#\sigma_I = \dim \F{\sigma}{p}$.
By Table~\ref{table:Howe}, the corresponding weight for $\k = \gl_p \oplus \gl_q$ is $\lambda = (-k, \ldots, -k ; \sigma)$, which is the highest weight of the $\k$-module $\det^{-k} \otimes \F{\sigma}{p}$, which has the same dimension as $\F{\sigma}{p}$.
\end{proof}

\begin{ex}[$k=3$, $p=q=4$]
\label{ex:GL covariants}
    Suppose that $\sigma^- = 0$.
    In order to use Theorem~\ref{thm:Hilbert series GL_k covariants}, we simply need to know $\#\sigma_I$ and $P_I(t)$ for each $I \in \mathcal{C}^4_3 = \{(1,2,3), (1,2,4), (1,3,4), (2,3,4)\}$.
    Note that we have already found each $P_I(t)$ in Example~\ref{ex:SL_k invariants k3p3q4}, namely, the last rational expression in the contribution from each $I$.
    (Recall that each $P_I(t)$ is easily computed by counting corners of facets.)
    We thus obtain
    \begin{equation}
        \label{Hilbert series in GL cov example}
        t^{|\sigma|}\left(\#\sigma_{(1,2,3)}\frac{1+t^2+t^4+t^6}{(1-t^2)^{15}} + \#\sigma_{(1,2,4)} \frac{1+2t^2 + 3t^4}{(1-t^2)^{14}} + \#\sigma_{(1,3,4)} \frac{1 + 3t^2}{(1-t^2)^{13}} + \#\sigma_{(2,3,4)} \frac{1}{(1-t^2)^{12}}\right).
    \end{equation}
    To make this even more concrete, we specify $\sigma = (2,1,0)$.
    We then determine each $\sigma_I$, using Definition~\ref{def:bins} and the discussion thereafter.
    Below, we depict the $\sigma$-chains in each $\sigma_I$.
    Beneath each $\sigma$-chain we give its corresponding tableaux in $\SSYT((2,1,0), \: 4)$, just as in~\eqref{two MS with tableau fact} and Example~\ref{ex:sigma chain}:
\tikzstyle{endpt}=[circle,fill=black, text=white, font=\tiny\sffamily\bfseries, minimum size = 8pt, inner sep=0pt]
\tikzstyle{smallpt}=[circle,fill=lightgray, minimum size = 5pt, inner sep=0pt]
\ytableausetup{smalltableaux}
\begin{equation*}
\resizebox{\linewidth}{!}{$\begin{split}
\sigma_{(1,2,3)} &= \left\{
\begin{tikzpicture}[scale=.4, baseline=(current bounding box.center)]
\draw[ultra thick,lightgray] (1,1)--(4,1);
\node at (1,1) [smallpt] {};
\node at (2,1) [smallpt] {};
\node at (3,1) [smallpt] {};
\node at (4,1) [smallpt] {};
\draw[ultra thick] (1,0.5) -- (1,2) (2,0.5) -- (2,2) -- (1,3) (3,1) -- (3,0.5);
\node at (1,2) [endpt] {1};
\node at (1,3) [endpt] {1};
\node at (2,2) [endpt] {2};
\node at (2.5,-.75) {\ytableaushort{34,4}};
\end{tikzpicture},
\quad
\begin{tikzpicture}[scale=.4, baseline=(current bounding box.center)]
\draw[ultra thick,lightgray] (1,1)--(4,1);
\node at (1,1) [smallpt] {};
\node at (2,1) [smallpt] {};
\node at (3,1) [smallpt] {};
\node at (4,1) [smallpt] {};
\draw[ultra thick] (1,0.5) -- (1,2) (2,0.5) -- (2,2) -- (2,3) (3,1) -- (3,0.5);
\node at (1,2) [endpt] {1};
\node at (2,3) [endpt] {2};
\node at (2,2) [endpt] {2};
\node at (2.5,-.75)  {\ytableaushort{33,4}};
\end{tikzpicture},
\quad
\begin{tikzpicture}[scale=.4, baseline=(current bounding box.center)]
\draw[ultra thick,lightgray] (1,1)--(4,1);
\node at (1,1) [smallpt] {};
\node at (2,1) [smallpt] {};
\node at (3,1) [smallpt] {};
\node at (4,1) [smallpt] {};
\draw[ultra thick] (1,0.5) -- (1,2) (3,0.5) -- (3,2) -- (1,3) (2,0.5) -- (2,1);
\node at (1,2) [endpt] {1};
\node at (1,3) [endpt] {1};
\node at (3,2) [endpt] {3};
\node at (2.5,-.75) {\ytableaushort{24,4}};
\end{tikzpicture},
\quad
\begin{tikzpicture}[scale=.4, baseline=(current bounding box.center)]
\draw[ultra thick,lightgray] (1,1)--(4,1);
\node at (1,1) [smallpt] {};
\node at (2,1) [smallpt] {};
\node at (3,1) [smallpt] {};
\node at (4,1) [smallpt] {};
\draw[ultra thick] (1,0.5) -- (1,2) (3,0.5) -- (3,2) -- (2,3) (2,0.5) -- (2,1);
\node at (1,2) [endpt] {1};
\node at (2,3) [endpt] {2};
\node at (3,2) [endpt] {3};
\node at (2.5,-.75) {\ytableaushort{23,4}};
\end{tikzpicture},
\quad
\begin{tikzpicture}[scale=.4, baseline=(current bounding box.center)]
\draw[ultra thick,lightgray] (1,1)--(4,1);
\node at (1,1) [smallpt] {};
\node at (2,1) [smallpt] {};
\node at (3,1) [smallpt] {};
\node at (4,1) [smallpt] {};
\draw[ultra thick] (1,0.5) -- (1,2) (3,0.5) -- (3,2) -- (3,3) (2,0.5) -- (2,1);
\node at (1,2) [endpt] {1};
\node at (3,3) [endpt] {3};
\node at (3,2) [endpt] {3};
\node at (2.5,-.75) {\ytableaushort{22,4}};
\end{tikzpicture},
\quad
\begin{tikzpicture}[scale=.4, baseline=(current bounding box.center)]
\draw[ultra thick,lightgray] (1,1)--(4,1);
\node at (1,1) [smallpt] {};
\node at (2,1) [smallpt] {};
\node at (3,1) [smallpt] {};
\node at (4,1) [smallpt] {};
\draw[ultra thick] (2,0.5) -- (2,2) (3,0.5) -- (3,2) -- (1,3) (1,0.5) -- (1,1);
\node at (2,2) [endpt] {2};
\node at (1,3) [endpt] {1};
\node at (3,2) [endpt] {3};
\node at (2.5,-.75) {\ytableaushort{24,3}};
\end{tikzpicture},
\quad
\begin{tikzpicture}[scale=.4, baseline=(current bounding box.center)]
\draw[ultra thick,lightgray] (1,1)--(4,1);
\node at (1,1) [smallpt] {};
\node at (2,1) [smallpt] {};
\node at (3,1) [smallpt] {};
\node at (4,1) [smallpt] {};
\draw[ultra thick] (2,0.5) -- (2,2) (3,0.5) -- (3,2) -- (2,3) (1,0.5) -- (1,1);
\node at (2,2) [endpt] {2};
\node at (2,3) [endpt] {2};
\node at (3,2) [endpt] {3};
\node at (2.5,-.75) {\ytableaushort{23,3}};
\end{tikzpicture},
\quad
\begin{tikzpicture}[scale=.4, baseline=(current bounding box.center)]
\draw[ultra thick,lightgray] (1,1)--(4,1);
\node at (1,1) [smallpt] {};
\node at (2,1) [smallpt] {};
\node at (3,1) [smallpt] {};
\node at (4,1) [smallpt] {};
\draw[ultra thick] (2,0.5) -- (2,2) (3,0.5) -- (3,2) -- (3,3) (1,0.5) -- (1,1);
\node at (2,2) [endpt] {2};
\node at (3,3) [endpt] {3};
\node at (3,2) [endpt] {3};
\node at (2.5,-.75) {\ytableaushort{22,3}};
\end{tikzpicture}
\right\}\\
\sigma_{(1,2,4)} &= \left\{
\begin{tikzpicture}[scale=.4, baseline=(current bounding box.center)]
\draw[ultra thick,lightgray] (1,1)--(4,1);
\node at (1,1) [smallpt] {};
\node at (2,1) [smallpt] {};
\node at (3,1) [smallpt] {};
\node at (4,1) [smallpt] {};
\draw[ultra thick] (1,0.5) -- (1,2) (4,0.5) -- (4,2) -- (1,3) (2,1) -- (2,0.5);
\node at (1,2) [endpt] {1};
\node at (1,3) [endpt] {1};
\node at (4,2) [endpt] {4};
\node at (2.5,-.75) {\ytableaushort{14,4}};
\end{tikzpicture},
\quad
\begin{tikzpicture}[scale=.4, baseline=(current bounding box.center)]
\draw[ultra thick,lightgray] (1,1)--(4,1);
\node at (1,1) [smallpt] {};
\node at (2,1) [smallpt] {};
\node at (3,1) [smallpt] {};
\node at (4,1) [smallpt] {};
\draw[ultra thick] (1,0.5) -- (1,2) (4,0.5) -- (4,2) -- (2,3) (2,1) -- (2,0.5);
\node at (1,2) [endpt] {1};
\node at (2,3) [endpt] {2};
\node at (4,2) [endpt] {4};
\node at (2.5,-.75) {\ytableaushort{13,4}};
\end{tikzpicture},
\quad
\begin{tikzpicture}[scale=.4, baseline=(current bounding box.center)]
\draw[ultra thick,lightgray] (1,1)--(4,1);
\node at (1,1) [smallpt] {};
\node at (2,1) [smallpt] {};
\node at (3,1) [smallpt] {};
\node at (4,1) [smallpt] {};
\draw[ultra thick] (1,0.5) -- (1,2) (4,0.5) -- (4,2) -- (3,3) (2,1) -- (2,0.5);
\node at (1,2) [endpt] {1};
\node at (3,3) [endpt] {3};
\node at (4,2) [endpt] {4};
\node at (2.5,-.75) {\ytableaushort{12,4}};
\end{tikzpicture},
\quad
\begin{tikzpicture}[scale=.4, baseline=(current bounding box.center)]
\draw[ultra thick,lightgray] (1,1)--(4,1);
\node at (1,1) [smallpt] {};
\node at (2,1) [smallpt] {};
\node at (3,1) [smallpt] {};
\node at (4,1) [smallpt] {};
\draw[ultra thick] (1,0.5) -- (1,2) (4,0.5) -- (4,2) -- (4,3) (2,1) -- (2,0.5);
\node at (1,2) [endpt] {1};
\node at (4,3) [endpt] {4};
\node at (4,2) [endpt] {4};
\node at (2.5,-.75) {\ytableaushort{11,4}};
\end{tikzpicture},
\quad
\begin{tikzpicture}[scale=.4, baseline=(current bounding box.center)]
\draw[ultra thick,lightgray] (1,1)--(4,1);
\node at (1,1) [smallpt] {};
\node at (2,1) [smallpt] {};
\node at (3,1) [smallpt] {};
\node at (4,1) [smallpt] {};
\draw[ultra thick] (2,0.5) -- (2,2) (4,0.5) -- (4,2) -- (1,3) (1,1) -- (1,0.5);
\node at (2,2) [endpt] {2};
\node at (1,3) [endpt] {1};
\node at (4,2) [endpt] {4};
\node at (2.5,-.75) {\ytableaushort{14,3}};
\end{tikzpicture},
\quad
\begin{tikzpicture}[scale=.4, baseline=(current bounding box.center)]
\draw[ultra thick,lightgray] (1,1)--(4,1);
\node at (1,1) [smallpt] {};
\node at (2,1) [smallpt] {};
\node at (3,1) [smallpt] {};
\node at (4,1) [smallpt] {};
\draw[ultra thick] (2,0.5) -- (2,2) (4,0.5) -- (4,2) -- (2,3) (1,1) -- (1,0.5);
\node at (2,2) [endpt] {2};
\node at (2,3) [endpt] {2};
\node at (4,2) [endpt] {4};
\node at (2.5,-.75) {\ytableaushort{23,3}};
\end{tikzpicture},
\quad
\begin{tikzpicture}[scale=.4, baseline=(current bounding box.center)]
\draw[ultra thick,lightgray] (1,1)--(4,1);
\node at (1,1) [smallpt] {};
\node at (2,1) [smallpt] {};
\node at (3,1) [smallpt] {};
\node at (4,1) [smallpt] {};
\draw[ultra thick] (2,0.5) -- (2,2) (4,0.5) -- (4,2) -- (3,3) (1,1) -- (1,0.5);
\node at (2,2) [endpt] {2};
\node at (3,3) [endpt] {3};
\node at (4,2) [endpt] {4};
\node at (2.5,-.75) {\ytableaushort{12,3}};
\end{tikzpicture},
\quad
\begin{tikzpicture}[scale=.4, baseline=(current bounding box.center)]
\draw[ultra thick,lightgray] (1,1)--(4,1);
\node at (1,1) [smallpt] {};
\node at (2,1) [smallpt] {};
\node at (3,1) [smallpt] {};
\node at (4,1) [smallpt] {};
\draw[ultra thick] (2,0.5) -- (2,2) (4,0.5) -- (4,2) -- (4,3) (1,1) -- (1,0.5);
\node at (2,2) [endpt] {2};
\node at (4,3) [endpt] {4};
\node at (4,2) [endpt] {4};
\node at (2.5,-.75) {\ytableaushort{11,3}};
\end{tikzpicture}\right\}\\
\sigma_{(1,3,4)} &= \left\{
\begin{tikzpicture}[scale=.4, baseline=(current bounding box.center)]
\draw[ultra thick,lightgray] (1,1)--(4,1);
\node at (1,1) [smallpt] {};
\node at (2,1) [smallpt] {};
\node at (3,1) [smallpt] {};
\node at (4,1) [smallpt] {};
\draw[ultra thick] (3,0.5) -- (3,2) (4,0.5) -- (4,2) -- (1,3) (1,1) -- (1,0.5);
\node at (3,2) [endpt] {3};
\node at (1,3) [endpt] {1};
\node at (4,2) [endpt] {4};
\node at (2.5,-.75) {\ytableaushort{14,2}};
\end{tikzpicture},
\quad
\begin{tikzpicture}[scale=.4, baseline=(current bounding box.center)]
\draw[ultra thick,lightgray] (1,1)--(4,1);
\node at (1,1) [smallpt] {};
\node at (2,1) [smallpt] {};
\node at (3,1) [smallpt] {};
\node at (4,1) [smallpt] {};
\draw[ultra thick] (3,0.5) -- (3,2) (4,0.5) -- (4,2) -- (2,3) (1,1) -- (1,0.5);
\node at (3,2) [endpt] {3};
\node at (2,3) [endpt] {2};
\node at (4,2) [endpt] {4};
\node at (2.5,-.75) {\ytableaushort{13,2}};
\end{tikzpicture},
\quad
\begin{tikzpicture}[scale=.4, baseline=(current bounding box.center)]
\draw[ultra thick,lightgray] (1,1)--(4,1);
\node at (1,1) [smallpt] {};
\node at (2,1) [smallpt] {};
\node at (3,1) [smallpt] {};
\node at (4,1) [smallpt] {};
\draw[ultra thick] (3,0.5) -- (3,2) (4,0.5) -- (4,2) -- (3,3) (1,1) -- (1,0.5);
\node at (3,2) [endpt] {3};
\node at (3,3) [endpt] {3};
\node at (4,2) [endpt] {4};
\node at (2.5,-.75) {\ytableaushort{12,2}};
\end{tikzpicture},
\quad
\begin{tikzpicture}[scale=.4, baseline=(current bounding box.center)]
\draw[ultra thick,lightgray] (1,1)--(4,1);
\node at (1,1) [smallpt] {};
\node at (2,1) [smallpt] {};
\node at (3,1) [smallpt] {};
\node at (4,1) [smallpt] {};
\draw[ultra thick] (3,0.5) -- (3,2) (4,0.5) -- (4,2) -- (4,3) (1,1) -- (1,0.5);
\node at (3,2) [endpt] {3};
\node at (4,3) [endpt] {4};
\node at (4,2) [endpt] {4};
\node at (2.5,-.75) {\ytableaushort{11,2}};
\end{tikzpicture}
\right\}\\
\sigma_{(2,3,4)} &= \varnothing.
\end{split}$}
\end{equation*}
    
Hence the Hilbert--Poincar\'e series of the $\GL_3$-covariants of $V^{*4} \oplus V^4$ of type $U_{(2,1,0)}$ equals
    \begin{align*}
    & t^3 \left(8 \cdot\frac{1+t^2+t^4+t^6}{(1-t^2)^{15}} + 8 \cdot \frac{1+2t^2 + 3t^4}{(1-t^2)^{14}} + 4 \cdot \frac{1 + 3t^2}{(1-t^2)^{13}} + 0 \cdot \frac{1}{(1-t^2)^{12}}\right) \\
    = \hspace{1ex} & \frac{20t^3 + 20 t^5 - 4 t^7 - 4 t^9}{(1 - t^2)^{15}}.
    \end{align*}
As the form of this series suggests, the module of covariants in this example is not Cohen--Macaulay; indeed, this example violates the Cohen--Macaulay criterion $\sigma_1 < p-k$ in~\cite{Armour}*{Table~4.2}.
We find it striking that nevertheless, the Hilbert series can be expressed as a positive combination of individual rational functions $P_I(t)$ which have such a nice combinatorial interpretation in terms of lattice paths.

To conclude the example, we give explicit formulas to compute $\#\sigma_I$ for generic $\sigma = (\sigma_1, \sigma_2, \sigma_3)$:
\begin{align*}
    \#\sigma_{(1,2,3)} &= (\sigma_1 - \sigma_2 + 1) (\sigma_1 - \sigma_3 + 2) (\sigma_2 - \sigma_3 + 1)/2,\\
    \#\sigma_{(1,2,4)} &= (\sigma_1 - \sigma_2 + 1) (\sigma_1 - \sigma_3 + 1) (\sigma_2 - \sigma_3 + 1) (\sigma_1 + \sigma_2 + \sigma_3)/6,\\
    \#\sigma_{(1,3,4)} &= (\sigma_1 - \sigma_2 + 1) (\sigma_1 - \sigma_3 + 2) (\sigma_2 - \sigma_3 + 1) (\sigma_2 + \sigma_1 \sigma_2 + 
   2 \sigma_3 + \sigma_1 \sigma_3 + \sigma_2 \sigma_3)/12,\\
   \#\sigma_{(2,3,4)} &= \sigma_3(\sigma_1 + 2) (\sigma_1 - \sigma_2 +1) (\sigma_2 + 1) (\sigma_1 - \sigma_3 + 2) (\sigma_2 - \sigma_3 + 1)/12.
\end{align*}
We obtained these as follows.
By~\eqref{covariants = L lambda}, the Hilbert series of the covariants equals that of $L_\la$, the simple $\mathfrak{su}(4,4)$-module with highest weight $\la = (-3,-3,-3,-3; \sigma_1, \sigma_2, \sigma_3,0)$.
By calculating the generalized BGG resolution of $L_\la$ using the methods in~\cites{EW,EnrightHunziker04}, we can write down its Hilbert series $P(L_\la; t)$ in terms of $\sigma_1$, $\sigma_2$, and $\sigma_3$.
Upon setting $P(L_\la; t)$ equal to~\eqref{Hilbert series in GL cov example}, and solving the resulting system of equations for each $\#\sigma_I$, we obtain the formulas above.
\end{ex}

\subsection{Modules of covariants for the orthogonal group}

Here we consider modules of $\O(V)$-covariants in the special case where $\sigma = (1^m)$ for some $0 \leq m \leq k$.
(For generic $\sigma$, the details are more subtle, and we leave a full exposition for future work.)
Then $U_\sigma = \Wedge^m V$.
Note that the $\O(V)$-invariants and semiinvariants were the special cases $m=0$ and $m=k$, respectively.
For $X = (x_1, \ldots, x_m) \in \mathcal{C}^n_m$, define the $\O(V)$-covariant function 
\begin{align*}
    \phi_X: V^n &\longrightarrow \Wedge^m V,\\
    (v_1, \ldots, v_n) & \longmapsto v_{x_1} \wedge \cdots \wedge v_{x_m}.
\end{align*}
As a simplified analogue of Definition~\ref{def:bins}, for $\sigma = (1^m)$ and $I \in \Ceq{n}$ we define
\[
    \sigma_I \coloneqq \{X \in \mathcal{C}^n_{m} \mid \text{$I$ is the minimal element of $\Ceq{n}$ containing $X$}\}.
\]
If $X \in \sigma_I$ and if $\f$ is a family of $k$ non-intersecting lattice paths in $\P_{\O}$ with starting points given by $I$, then we write $\f \leftarrow X$.
We have the following Stanley decomposition, which is the analogue of~\eqref{Stanley decomp GL covariants}:
\[
(\C[V^n] \otimes \Wedge^m V)^{\O(V)} = \bigoplus_{\f \leftarrow X} f_{\cor(\f)} \C[f_{ij} \mid (i,j) \in \f] \cdot \phi_X.
\]
Then, abbreviating (as before)
\[
P_I(t) \coloneqq \sum_{\f \leftarrow I} \frac{(t^2)^{\#\cor(\f)}}{(1-t^2)^{\#\f}},
\]
we have the following Hilbert--Poincar\'e series:
\begin{equation}
    \label{Hilbert series O_k covariants nonpure}
P\!\Big((\C[V^n] \otimes \Wedge^m V)^{\O(V)}; t\Big) = t^m \sum_{I \in \Ceq{n}} \#\sigma_I \cdot P_I(t).
\end{equation}

To obtain a pure Hilbert--Poincar\'e series, let $\scrF_{k,m}$ be the set of all families of $k$ maximal non-intersecting paths in $\P_{\O}$, exactly $m$ of which have their endpoints on the diagonal ``painted.''
Hence $\scrF_{k,m} \cong \scrF_k \times \mathcal{C}^k_m$.
Painted endpoints are \emph{not} counted now as corners in $\cor(\f)$.
We then have the pure form
\begin{equation}
    \label{Hilbert series O_k covariants pure}
    P\!\left((\C[V^n] \otimes \Wedge^{\!m} V)^{\O(V)}; t\right) = t^m \cdot \frac{\sum_{\f \in \scrF_{k,m}} (t^2)^{\#\cor(\f)}}{(1-t^2)^{k(2n-k+1)/2}}.
\end{equation}


\begin{ex}[$k=3$, $n=4$, $m=2$]
\label{ex:O covariants}
    \input{Hilbert_series_examples/O-covariants_k3n4m2}
\end{ex}

\subsection{Modules of covariants for the symplectic group}

Let $H = \Sp(V)$ with $\dim V = 2k$.
In this subsection, we let $\f$ denote a family of $k$ non-intersecting lattice paths in $\P_{\Sp}$ whose endpoints are $(n-2i+1, n)$ for $i = 1, \ldots, k$.  
We make analogous definitions to those in the case of $\GL(V)$:

\begin{dfn}
    Let $\sigma$ be a partition with $\ell(\sigma) \leq k$.
    A \emph{$\sigma$-chain} is a chain $X_1 \preceq \cdots \preceq X_{\sigma_1}$ in $\Cleq{n}$ such that the SSYT with columns $\widetilde{X}_{\sigma_1}, \ldots, \widetilde{X}_1$ has shape $\sigma$.
\end{dfn}

\begin{dfn}
\label{C arrow f}
    Let $\mathbf{C}$ be a $\sigma$-chain in $\Cleq{n}$, with maximal element $X$.
    We write $\mathbf{C} \rightarrow I$ if $I$ is the minimal element in $I \in \Ceq{n}$ such that $X \rightarrow I$ and $I \geq (2,4,6,\ldots,2k)$.
    We define the set
    \[
    \sigma_I \coloneqq \{\text{$\sigma$-chains $\mathbf{C}$} \mid \mathbf{C} \rightarrow I \}.
    \]
    We write $\mathbf{C} \rightarrow \f$ if there exists an $I$ (necessarily unique) such that $\mathbf{C} \rightarrow I$ and $I = (i_1, \ldots, i_k)$ gives the starting points $(1, i_\ell)$, for $\ell = 1, \ldots, k$, of the paths in $\f$.
\end{dfn}

\begin{dfn}
\label{def:shadow}
    Given a family $I \rightarrow \f$ with $I = (i_1, \ldots, i_k)$, we say that the $\ell$th starting point $(1, i_\ell)$ \emph{shadows} the points $(j+1, i_\ell - j)$, for $j=1, \ldots, \ell-1$.
    Likewise, we say the $\ell$th endpoint $(n-2\ell+1, \: n)$ \emph{shadows} the points $(n-2\ell+1-j, \: n-j)$, for $j = 1, \ldots, \ell-1$.
    Then we define $\cor(\f)$ to be the set of all \scalebox{2}{$\llcorner$}-patterns in $\f$ which are not shadowed.
    (Intuitively, the $\ell$th start/end point casts a shadow of size $\ell$ to its southwest; see the figures in Example~\ref{ex:Sp covariants k2n6}.)
\end{dfn}

Fixing a basis for $V$, let $v_{ij}$ denote the $i$th coordinate of $v_j$.
Then we define the following $N$-invariant functions:
\[
    \textstyle{\det_{X}}(v_1, \ldots, v_n) \coloneqq \det(v_{ij})_{i = 2k+1-\#X, \ldots, 2k; \; j \in X,} \qquad X \in \mathcal{C}^n_{\leqslant 2k}.
\]
Together with the quadratics $f_{ij}$ defined in~\eqref{f_ij for Sp_2k}, the functions $\det_X$ generate $\C[W]^N$; see~\cite{Jackson}*{Cor.~3.4.4}.
Given a $\sigma$-chain $\mathbf{C}$, we define $\det_{\mathbf{C}}$ just as in~\eqref{alpha beta}.
By~\cite{Jackson}*{Prop.~3.3.7}, a function $f \in \C[W]$ lies has weight $\sigma$ under the action of the maximal torus of $H$ if and only if
\[
\sigma_i = \text{(degree of $f$ in the variables $v_{2k-i+1,\bullet}$)} - \text{(degree of $f$ in the variables $v_{i,\bullet}$)}, \qquad 1 \leq i \leq k.
\]
We define the following rational expression for each $I \in \Cleq{n}$ such that $I \geq (2, 4, 6, \ldots, 2k)$:
\[
P_I(t) \coloneqq \frac{\sum_{I \rightarrow \f} (t^2)^{\#\cor(\f)}}{(1-t^2)^{d_I}},
\]
where $d_I = \#\f = k(2n-k)-|I|$.

\begin{theorem}
    \label{thm:Sp covariants}
    Let $U_\sigma$ be the irreducible representation of $\Sp(V)$ with highest weight $\sigma$.
    Then the module of $\Sp(V)$-covariants of $V^n$, of type $U_\sigma$, has the following Stanley decomposition (where the sums range over the hybrid path families $\mathbf{C} \rightarrow \f$ in Definition~\ref{C arrow f}):
    \begin{equation}
        \label{Stanley decomp Sp covariants}
        (\C[V^n] \otimes U_\sigma)^{\Sp(V)} \cong \bigoplus_{\mathbf{C} \rightarrow \f} f_{\cor(\f)} \C[f_{ij} \mid (i,j) \in \f] \cdot \textstyle \det_{\mathbf{C}}.
    \end{equation}
    Furthermore, we have the following Hilbert--Poincar\'e series:
    \begin{equation}
        \label{Hilbert series Sp covariants}
        P\Big( (\C[V^n] \otimes U_\sigma)^{\Sp(V)};t \Big) = t^{|\sigma|} \sum_{\mathclap{\substack{I \in \Ceq{n}: \\ I \geq (2,4, \ldots, 2k)}}} \#\sigma_I \cdot P_I(t),
    \end{equation}
    and $\sum_I \#\sigma_I = \dim \F{\lambda}{n}$.
\end{theorem}

The proof of Theorem~\ref{thm:Sp covariants} proceeds similarly to that of the $\GL(V)$ case in Theorem~\ref{thm:Hilbert series GL_k covariants}, \emph{mutatis mutandis}.
In particular, one can reinterpret Jackson's basis of standard monomials~\cite{Jackson}*{Def.~3.6.10} using our notation $f_{A} \cdot \det_{\mathbf{C}}$, just as in Lemma~\ref{lemma: SM basis for GL covariants}.
In this case, where $H = \Sp(V)$, the semistandard tableau $A$ has entries in $[n]$ and has even column lengths, and $f_A$ denotes the ordinary monomial in the $f_{ij}$'s whose degree matrix is given by $\RSK_{\Sp}(A)$.
As before, the key is to translate Jackson's $\Sp_{2k}$-splits~\cite{Jackson}*{3.6.16} into a condition on neighboring tableau columns (just as in Lemma~\ref{lemma: split to tableaux} above).
Then part (1) of Lemma~\ref{lemma:Ci and Di for Sp_2k} establishes the decomposition~\eqref{Stanley decomp Sp covariants}.

\begin{corollary}[Special case of \cite{NOT}*{Cor.~9.2}]
\label{cor:NOT result Sp}
    Let $(H, \g) = (\Sp_{2k}, \so_{n})$.
    Let $\sigma \mapsto \la$ be the map described in Theorem~\ref{thm:Howe duality}, and let \textup{``Deg''} denote the Bernstein degree.  
    Then 
    \[
    \operatorname{Deg} L_\la = \dim U_\sigma \cdot \operatorname{deg} \overline{\scrO}_k.
    \]
\end{corollary}

\begin{proof}
    The proof is the same as in the $\GL(V)$ case above (Corollary~\ref{cor:NOT result}), except that we must show that $\#\sigma_{(2,4,\ldots, 2k)} = \dim U_\sigma$.
    To this end, note that $\mathbf{C} = X_1 \preceq \cdots \preceq X_m$ is an element of $\sigma_{(2,4,\ldots,2k)}$ if and only if $X_m \rightarrow (2,4,\ldots,2k)$, if and only if $(1,3,\ldots,2k-1) \rightarrow \widetilde{X}_m$.
    This is equivalent to the condition that the SSYT with columns $\widetilde{X}_m, \ldots, \widetilde{X}_1$ has the following property: the $i$th element of its first column must be no less than $2i-1$.
    This is precisely the definition of a \emph{semistandard symplectic tableau} in~\cite{CampbellStokke}*{p.~464}, which follows King's convention~\cite{King75}*{p.~496}.
    (In this paper, we fill our tableaux with the numbers $1,2,\ldots,2k$ rather than the symbols $1, \Bar{1}, \ldots, k, \Bar{k}$ used by the authors cited above.)
    Since these works give a bijection between a weight basis for $U_\sigma$ and the set of symplectic tableaux of shape $\sigma$, we have $\#\sigma_{(2,4,\ldots,2k)} = \dim U_\sigma$, as desired.  
\end{proof}

As in the case of $\GL(V)$, Corollary~\ref{cor:NOT result Sp} allows us to interpret the Bernstein degree of $L_\la$ as the number of families of hybrid paths $\mathbf{C} \rightarrow (2, 4, \ldots, 2k)\rightarrow \f$ such that $\mathbf{C}$ is a $\sigma$-chain. 

\begin{ex}[$k=2$, $n=6$, $\sigma = (1,1)$]
\label{ex:Sp covariants k2n6}

    \input{Hilbert_series_examples/Sp-covariants_k2n6}
    
\end{ex}

\section{Hilbert--Poincar\'e series of the Wallach representations of Type ADE}
\label{sec:ADE}

\subsection{The Wallach representations and rings of invariants}
\label{sub:Wallach}

We return to the general Hermitian symmetric setting described in Section~\ref{sub:Hermitian symmetric}.
Let $c \coloneqq \frac{1}{2}(\rho, \: \gamma_2^\vee - \gamma_1^\vee)$, and let $\zeta$ be the unique fundamental
weight of $\g$ that is orthogonal to $\Phi(\k)$.
Then for any Hermitian symmetric pair $(\g,\k)$, and $k \leq r$, the \emph{$k$th Wallach representation} is the simple $\g$-module $L_{-kc\zeta}$.

When $\g$ arises in one of the three dual pairs $(H,\g)$ described above, the weight $\lambda = -kc\zeta$ is the image of $\sigma = 0$ under the map $\Sigma \rightarrow \Lplusk$.
In this case, $U_\sigma$ is the trivial representation of $H$, and it follows from~\eqref{covariants = L lambda} that the $k$th Wallach representation is isomorphic to $\C[W]^H$ as a $(\g,K)$-module.
Hence, in the dual pair setting, we have already seen how to interpret the Hilbert--Poincar\'e series of $L_{-kc\zeta}$ in terms of lattice paths.  
In this section, we extend this lattice path interpretation to the Wallach representations of all Hermitian symmetric pairs of simply laced type (i.e., where $\g$ is one of the Cartan types $\ssA$, $\ssD$, or $\ssE$).

In the dual pair setting, the two pairs $(H,\g)$ with $\g$ simply laced were $(\GL_k, \gl_{p+q})$ and $(\Sp_{2k}, \so_{2n})$.
We observe that in these two cases, we have an isomorphism of posets
$\P_H \cong \Phi(\p^+)$, given explicitly by
\begin{align*}
    \P_{\GL} \longrightarrow \Phi(\p^+), \qquad & (i,j) \longmapsto \ep_{p+1-i} - \ep_j, \\
    \P_{\Sp} \longrightarrow \Phi(\p^+), \qquad & (i,j) \longmapsto \ep_{n+1-j} + \ep_{n+1-i}.
\end{align*}
In other words, the minimal element $(1,1) \in \P_H$, which we depicted in the upper-left corner, corresponds to the minimal noncompact root $\gamma_1 \in \Phi(\p^+)$, while the maximal element in the lower-right corner of $\P_H$ corresponds to the highest root in $\Phi(\p^+)$.
In this light, one can view each facet in $\scrF_k$ as a maximal width-$k$ subset of $\Phi(\p^+)$.
It is a general fact for Hermitian symmetric pairs that the Hasse diagram of $\Phi(\p^+)$ is a planar lattice, and thus the $k$th order complex is shellable; this allows us to generalize the construction
\[
P(L_{-kc\zeta};t) = (1-t)^{-d} \sum_{\f \in \scrF_k} t^{\#\cor(\f)},
\]
where $d$ is the common size of all facets in $\scrF_k$.
We carry this out below to reinterpret the Hilbert--Poincar\'e series of the Wallach representations for the Hermitian symmetric pairs of simply laced type.
These series were written down in in~\cite{EnrightHunziker04}*{\S\S6.6--6.8}, to which we refer the reader for further details; see also~\cite{EH04exceptional}.

\subsection{First Wallach representation of $(\ssD_n, \ssD_{n-1})$}
\label{sub:Wallach Dn}

For $(\g,\k) = (\ssD_n, \ssD_{n-1})$, 
we depict the poset $\Phi(\p^+)$ as two rows of $n-1$ points each, with the minimal element in the upper-left and the maximal element in the lower-right.
Below is the Hasse diagram for $n=6$ (rotated by 45 degrees for typographical reasons):
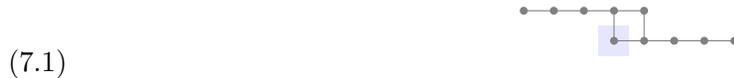
\tikzstyle{dot}=[circle,fill=gray, minimum size = 3pt, inner sep=0pt]\begin{equation}
\label{Hasse D6}
\begin{tikzpicture}[scale=.4, baseline]
\fill[blue, very nearly transparent] (3.5,0.5) rectangle (4.5,1.5);
\draw[gray] (1,2) --(5,2) -- (5,1) -- (8,1);
\draw[gray] (4,2) --(4,1) -- (5,1);
\node at (1,2) [dot] {};
\node at (2,2) [dot] {};
\node at (3,2) [dot] {};
\node at (4,2) [dot] {};
\node at (5,2) [dot] {};
\node at (4,1) [dot] {};
\node at (5,1) [dot] {};
\node at (6,1) [dot] {};
\node at (7,1) [dot] {};
\node at (8,1) [dot] {};
\end{tikzpicture}
\end{equation}
\noindent There is only one Wallach representation $L_\la$, corresponding to $k=1$, where $\la = -(n+2)\omega_1$.
Since $k=1$, the facets $\f$ are just maximal chains in $\Phi(\p^+)$, i.e., paths with steps to the south or to the east.
The corners are those elements in a path at which an \scalebox{2}{$\llcorner$}-pattern occurs; in the diagram above, we have shaded the only possible corner.
Clearly, there are only two maximal paths:
\begin{center}

\tikzstyle{corner}=[rectangle,draw=black,fill=red, minimum size = 5pt, inner sep=0pt]
\begin{tikzpicture}[scale=.4, baseline]
\draw[gray] (1,2) --(5,2) -- (5,1) -- (8,1);
\draw[gray] (4,2) --(4,1) -- (5,1);
\draw[ultra thick] (1,2) -- (5,2) -- (5,1) -- (8,1);

\end{tikzpicture}
\begin{tikzpicture}[scale=.4, baseline]
\draw[gray] (1,2) --(5,2) -- (5,1) -- (8,1);
\draw[gray] (4,2) --(4,1) -- (5,1);
\draw[ultra thick] (1,2) -- (4,2) -- (4,1) -- (8,1);

\node at (4,1) [corner] {};

\end{tikzpicture}

\end{center}
Note that each facet has size $2n-3$.
From this we obtain the following Hilbert--Poincar\'e series, which coincides with the calculation in~\cite{EnrightHunziker04}*{Thm.~26}:
\[
P(L_{\la};t) = \frac{1+t}{(1-t)^{2n-3}}.
\]

\subsection{First Wallach representation of $(\ssE_6, \ssD_5)$}
\label{sub:Wallach E6}

For $(\g,\k) = (\ssE_6, \ssD_5)$, 
there is only one Wallach representation $L_\la$, corresponding to $k=1$, where $\la = -3\zeta$.
Since $k=1$, the facets $\f$ are just maximal chains in $\Phi(\p^+)$.
Note that each facet has size $11$:
\begin{center}

\tikzstyle{corner}=[rectangle,draw=black,fill=red, minimum size = 5pt, inner sep=0pt]
\begin{tikzpicture}[scale=.3, baseline]
\fill[blue, very nearly transparent] (2.5,2.5) rectangle (4.5,3.5);
\fill[blue, very nearly transparent] (3.5,0.5) rectangle (4.5,2.5);
\fill[blue, very nearly transparent] (4.5,0.5) rectangle (5.5,1.5);
\draw[gray] (1,4) --(5,4) -- (5,1);
\draw[gray] (3,4) --(3,3) -- (5,3);
\draw[gray] (4,4) --(4,1) -- (8,1);
\draw[gray] (4,2) --(6,2) -- (6,1);
\draw[ultra thick] (1,4)--(4,4)--(4,3)--(5,3)--(5,2)--(6,2)--(6,1)--(8,1);
\node at (4,3) [corner] {};
\end{tikzpicture}

\end{center}
\noindent It remains to count paths with respect to the number of corners, which is easily done below:

\tikzstyle{corner}=[rectangle,draw=black,fill=red, minimum size = 5pt, inner sep=0pt]
\begin{tikzpicture}[scale=.3, baseline]
\draw[gray] (1,4) --(5,4) -- (5,1);
\draw[gray] (3,4) --(3,3) -- (5,3);
\draw[gray] (4,4) --(4,1) -- (8,1);
\draw[gray] (4,2) --(6,2) -- (6,1);
\draw[ultra thick] (1,4)--(5,4)--(5,2)--(6,2)--(6,1)--(8,1);

\node[right=0pt of current bounding box.east,anchor=west
    ]{$\leadsto \dfrac{1}{(1-t)^{11}}$};

\end{tikzpicture}

\begin{tikzpicture}[scale=.3, baseline]
\draw[gray] (1,4) --(5,4) -- (5,1);
\draw[gray] (3,4) --(3,3) -- (5,3);
\draw[gray] (4,4) --(4,1) -- (8,1);
\draw[gray] (4,2) --(6,2) -- (6,1);
\draw[ultra thick] (1,4)--(5,4)--(5,1)--(8,1);
\node at (5,1) [corner] {};
\end{tikzpicture}
\begin{tikzpicture}[scale=.3, baseline]
\draw[gray] (1,4) --(5,4) -- (5,1);
\draw[gray] (3,4) --(3,3) -- (5,3);
\draw[gray] (4,4) --(4,1) -- (8,1);
\draw[gray] (4,2) --(6,2) -- (6,1);
\draw[ultra thick] (1,4)--(4,4)--(4,3)--(5,3)--(5,2)--(6,2)--(6,1)--(8,1);
\node at (4,3) [corner] {};
\end{tikzpicture}
\begin{tikzpicture}[scale=.3, baseline]
\draw[gray] (1,4) --(5,4) -- (5,1);
\draw[gray] (3,4) --(3,3) -- (5,3);
\draw[gray] (4,4) --(4,1) -- (8,1);
\draw[gray] (4,2) --(6,2) -- (6,1);
\draw[ultra thick] (1,4)--(4,4)--(4,2)--(6,2)--(6,1)--(8,1);
\node at (4,2) [corner] {};
\end{tikzpicture}
\begin{tikzpicture}[scale=.3, baseline]
\draw[gray] (1,4) --(5,4) -- (5,1);
\draw[gray] (3,4) --(3,3) -- (5,3);
\draw[gray] (4,4) --(4,1) -- (8,1);
\draw[gray] (4,2) --(6,2) -- (6,1);
\draw[ultra thick] (1,4)--(4,4)--(4,1)--(8,1);
\node at (4,1) [corner] {};
\end{tikzpicture}
\begin{tikzpicture}[scale=.3, baseline]
\draw[gray] (1,4) --(5,4) -- (5,1);
\draw[gray] (3,4) --(3,3) -- (5,3);
\draw[gray] (4,4) --(4,1) -- (8,1);
\draw[gray] (4,2) --(6,2) -- (6,1);
\draw[ultra thick] (1,4)--(3,4)--(3,3)--(5,3)--(5,2)--(6,2)--(6,1)--(8,1);
\node at (3,3) [corner] {};

\node[right=0pt of current bounding box.east,anchor=west
    ]{$\leadsto \dfrac{5t}{(1-t)^{11}}$};
    
\end{tikzpicture}

\begin{tikzpicture}[scale=.3, baseline]
\draw[gray] (1,4) --(5,4) -- (5,1);
\draw[gray] (3,4) --(3,3) -- (5,3);
\draw[gray] (4,4) --(4,1) -- (8,1);
\draw[gray] (4,2) --(6,2) -- (6,1);
\draw[ultra thick] (1,4)--(4,4)--(4,3)--(5,3)--(5,1)--(8,1);
\node at (4,3) [corner] {};
\node at (5,1) [corner] {};
\end{tikzpicture}
\begin{tikzpicture}[scale=.3, baseline]
\draw[gray] (1,4) --(5,4) -- (5,1);
\draw[gray] (3,4) --(3,3) -- (5,3);
\draw[gray] (4,4) --(4,1) -- (8,1);
\draw[gray] (4,2) --(6,2) -- (6,1);
\draw[ultra thick] (1,4)--(4,4)--(4,2)--(5,2)--(5,1)--(8,1);
\node at (4,2) [corner] {};
\node at (5,1) [corner] {};
\end{tikzpicture}
\begin{tikzpicture}[scale=.3, baseline]
\draw[gray] (1,4) --(5,4) -- (5,1);
\draw[gray] (3,4) --(3,3) -- (5,3);
\draw[gray] (4,4) --(4,1) -- (8,1);
\draw[gray] (4,2) --(6,2) -- (6,1);
\draw[ultra thick] (1,4)--(3,4)--(3,3)--(4,3)--(4,2)--(6,2)--(6,1)--(8,1);
\node at (3,3) [corner] {};
\node at (4,2) [corner] {};
\end{tikzpicture}
\begin{tikzpicture}[scale=.3, baseline]
\draw[gray] (1,4) --(5,4) -- (5,1);
\draw[gray] (3,4) --(3,3) -- (5,3);
\draw[gray] (4,4) --(4,1) -- (8,1);
\draw[gray] (4,2) --(6,2) -- (6,1);
\draw[ultra thick] (1,4)--(3,4)--(3,3)--(5,3)--(5,1)--(8,1);
\node at (3,3) [corner] {};
\node at (5,1) [corner] {};
\end{tikzpicture}
\begin{tikzpicture}[scale=.3, baseline]
\draw[gray] (1,4) --(5,4) -- (5,1);
\draw[gray] (3,4) --(3,3) -- (5,3);
\draw[gray] (4,4) --(4,1) -- (8,1);
\draw[gray] (4,2) --(6,2) -- (6,1);
\draw[ultra thick] (1,4)--(3,4)--(3,3)--(4,3)--(4,1)--(8,1);
\node at (3,3) [corner] {};
\node at (4,1) [corner] {};

\node[right=0pt of current bounding box.east,anchor=west
    ]{$\leadsto \dfrac{5t^2}{(1-t)^{11}}$};

\end{tikzpicture}

\begin{tikzpicture}[scale=.3, baseline]
\draw[gray] (1,4) --(5,4) -- (5,1);
\draw[gray] (3,4) --(3,3) -- (5,3);
\draw[gray] (4,4) --(4,1) -- (8,1);
\draw[gray] (4,2) --(6,2) -- (6,1);
\draw[ultra thick] (1,4)--(3,4)--(3,3)--(4,3)--(4,2)--(5,2)--(5,1)--(8,1);
\node at (3,3) [corner] {};
\node at (4,2) [corner] {};
\node at (5,1) [corner] {};

\node[right=0pt of current bounding box.east,anchor=west
    ]{$\leadsto \dfrac{t^3}{(1-t)^{11}}$};

\end{tikzpicture}

\noindent From this we obtain the following Hilbert--Poincar\'e series, which coincides with~\cite{EnrightHunziker04}*{Thm.~28}:
\[
P(L_{\la};t) = \frac{1+5t+5t^2+t^3}{(1-t)^{11}}.
\]

\subsection{First Wallach representation of $(\ssE_7, \ssE_6)$}
\label{sub:Wallach E7 k1}

For $(\g, \k) = (\ssE_7, \ssE_6)$ there are two Wallach representations.
For $k=1$, we have $\la = -4\zeta$, and the facets are again the maximal chains in $\Phi(\p^+)$:
\begin{center}

\tikzstyle{corner}=[rectangle,draw=black,fill=red, minimum size = 5pt, inner sep=0pt]

\begin{tikzpicture}[scale=.3, baseline]
\begin{scope}[shift={(0,-4)}]
\fill[blue, very nearly transparent] (3.5,7.5) rectangle (5.5,8.5);
\fill[blue, very nearly transparent] (4.5,4.5) rectangle (5.5,7.5);
\fill[blue, very nearly transparent] (5.5,4.5) rectangle (6.5,6.5);
\fill[blue, very nearly transparent] (6.5,4.5) rectangle (8.5,5.5);
\fill[blue, very nearly transparent] (7.5,3.5) rectangle (8.5,4.5);
\draw[gray] (1,9) --(6,9) -- (6,5);
\draw[gray] (5,9) --(5,5) -- (9,5);
\draw[gray] (4,9) --(4,8)--(6,8) ;
\draw[gray] (5,7) --(7,7) -- (7,5);
\draw[gray] (5,6) --(9,6) -- (9,4);
\draw[gray] (8,6) --(8,4) -- (12,4);
\draw[ultra thick] (1,9)--(5,9)--(5,7)--(7,7)--(7,6)--(8,6)--(8,4)--(12,4);
\node at (5,7) [corner] {};
\node at (8,4) [corner] {};
\end{scope}
\end{tikzpicture}

\end{center}
Note that the size of each facet is $17$, and the maximum number of corners is 5.
Counting all maximal paths and their corners, we obtain
\[
P(L_\la; t) = \frac{ 1+10t+28t^2 +28t^3 +10t^4 +t^5}{(1-t)^{17}},
\]
which agrees with~\cite{EnrightHunziker04}*{Thm.~29}.

\subsection{Second Wallach representation of $(\ssE_7, \ssE_6)$}
\label{sub:Wallach E7 k2}

For $k=2$, we have $\la = -8\zeta$.
Each facet can be depicted as a family of $k=2$ non-intersecting paths in $\Phi(\p^+)$; each family necessarily contains $26$ elements.
Note that each path in a facet can contain at most one corner.
In fact, there are only three facets:
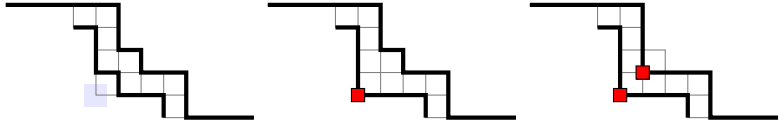
\begin{center}

\tikzstyle{corner}=[rectangle,draw=black,fill=red, minimum size = 5pt, inner sep=0pt]

\begin{tikzpicture}[scale=.3, baseline]
\begin{scope}[shift={(0,-4)}]
\fill[blue, very nearly transparent] (4.5,4.5) rectangle (5.5,5.5);
\draw[gray] (1,9) --(6,9) -- (6,5);
\draw[gray] (5,9) --(5,5) -- (9,5);
\draw[gray] (4,9) --(4,8)--(6,8) ;
\draw[gray] (5,7) --(7,7) -- (7,5);
\draw[gray] (5,6) --(9,6) -- (9,4);
\draw[gray] (8,6) --(8,4) -- (12,4);
\draw[ultra thick] (1,9)--(6,9)--(6,7)--(7,7)--(7,6)--(9,6)--(9,4)--(12,4);
\draw[ultra thick] (4,8)--(5,8)--(5,6)--(6,6)--(6,5)--(8,5)--(8,4);
\end{scope}
\end{tikzpicture}
\begin{tikzpicture}[scale=.3, baseline]
\begin{scope}[shift={(0,-4)}]
\draw[gray] (1,9) --(6,9) -- (6,5);
\draw[gray] (5,9) --(5,5) -- (9,5);
\draw[gray] (4,9) --(4,8)--(6,8) ;
\draw[gray] (5,7) --(7,7) -- (7,5);
\draw[gray] (5,6) --(9,6) -- (9,4);
\draw[gray] (8,6) --(8,4) -- (12,4);
\draw[ultra thick] (1,9)--(6,9)--(6,7)--(7,7)--(7,6)--(9,6)--(9,4)--(12,4);
\draw[ultra thick] (4,8)--(5,8)--(5,5)--(8,5)--(8,4);
\node at (5,5) [corner] {};
\end{scope}
\end{tikzpicture}
\begin{tikzpicture}[scale=.3, baseline]
\begin{scope}[shift={(0,-4)}]
\draw[gray] (1,9) --(6,9) -- (6,5);
\draw[gray] (5,9) --(5,5) -- (9,5);
\draw[gray] (4,9) --(4,8)--(6,8) ;
\draw[gray] (5,7) --(7,7) -- (7,5);
\draw[gray] (5,6) --(9,6) -- (9,4);
\draw[gray] (8,6) --(8,4) -- (12,4);
\draw[ultra thick] (1,9)--(6,9)--(6,6)--(9,6)--(9,4)--(12,4);
\draw[ultra thick] (4,8)--(5,8)--(5,5)--(8,5)--(8,4);
\node at (5,5) [corner] {};
\node at (6,6) [corner] {};
\end{scope}
\end{tikzpicture}
\end{center}
Hence the Hilbert--Poincar\'e series is
\[
P(L_\la; t) = \frac{1+t+t^2}{(1-t)^{26}},
\]
as given in~\cite{EnrightHunziker04}*{Thm.~30}.

\subsection{Enright--Shelton reduction and corner posets}
\label{sub:ES reduction}


Recall that Figure~\ref{fig:facet examples} illustrates a typical facet $\f = \bp_1 \sqcup \cdots \sqcup \bp_k$ of the $k$th order complex $\Delta_k(\P)$, where $\P$ is one of the three classical posets.
In particular, for a fixed index $i = 1, \ldots, k$, we shaded the region containing all possible corners of a path $\bp_i$.
Recall also that these ``corner regions'' are translates of each other as $i$ ranges from $1$ to $k$;
hence without ambiguity we now define ${\rm Cor}_k(\P) \subset \P$ to be the corner region induced by $\Delta_k(\P)$.

We will view ${\rm Cor}_k(\P)$ as a subset of $\P$, but not as a subposet, for the following reason.
Recall that each set $\cor(\bp_i)$ is a strict chain with respect to the partial order on $\P$.
In order to exploit the RSK correspondence, however, we imagined  the vertical or horizontal reflection so that $\cor(\bp_i)$ would be an \emph{anti}chain; see the proofs of Propositions~\ref{prop:Hilbert series GL_k}, \ref{prop:Hilbert series O_k}, and \ref{prop:Hilbert series Sp_2k}.
In other words, we actually viewed ${\rm Cor}_k(\P)$ as the poset whose Hasse diagram was a 90-degree rotation of the Hasse diagram it inherited from $\P$.
Upon endowing ${\rm Cor}_k(\P)$ with this partial order, and calling it the \emph{corner poset}, we observe the following poset isomorphisms:
\begin{align*}
    {\rm Cor}_k(\P_{\GL}(p,q)) & \cong \P_{\GL}(p-k, \: q-k),\\
    {\rm Cor}_k(\P_{\O}(n)) & \cong \P_{\Sp}(n-k+1),\\
    {\rm Cor}_k(\P_{\Sp}(n)) & \cong \P_{\O}(n - 2k - 1).
\end{align*}
Somewhat surprisingly, these isomorphisms encode a phenomenon called \emph{Enright--Shelton reduction}~\cites{ES87,ES89}, which we summarize briefly below.
In order to make the connection more transparent, we first rename each classical poset $\P_H$ as $\P(\g,\k)$, where $(\g,\k)$ is the pair corresponding to $H$ in Table~\ref{table:Howe}.
Labeling $(\g,\k)$ by its Killing--Cartan classification, we rewrite the poset isomorphisms above as follows:
\begin{align}
\label{Cor isos g}
\begin{split}
    \operatorname{Cor}_k \P(\ssA_{p+q-1}, \; \ssA_{p-1} \times \ssA_{q-1}) & \cong \P(\ssA_{p+q-2k-1}, \; \ssA_{p-k-1} \times \ssA_{q-k-1}),\\
    \operatorname{Cor}_k \P(\ssC_n, \ssA_n) & \cong \P(\ssD_{n-k+1}, \ssA_{n-k+1}),\\
    \operatorname{Cor}_k \P(\ssD_n, \ssA_n) & \cong \P(\ssC_{n - 2k - 1}, \ssA_{n-2k-1}).
    \end{split}
\end{align}

Enright--Shelton reduction is a process that converts a singular weight $\Lambda^{+}(\k)$ into a regular weight in $\Lambda^{+}(\k')$, where $(\g', \k')$ is a certain Hermitian symmetric pair whose rank is less than that of $(\g, \k)$.
More specifically, for the three families of Hermitian symmetric pairs $(\g, \k)$ arising in the Howe duality setting, where $\lambda = -kc\zeta$ (i.e., where $L_\lambda$ is the $k$th Wallach representation), there exists an associated pair $(\g',\k')$ such that the Enright--Shelton reduction $\lambda \mapsto \lambda' \in \Lambda^+(\k')$ induces a \emph{congruence of blocks} $\mathcal{B}_\la \cong \mathcal{B}_{\la'}$ in parabolic category $\mathcal{O}$.
(See~\cite{EricksonHunziker23}*{\S5} for details.)
The transfer theorem~\cite{EnrightHunziker04}*{p.~623} relates the Hilbert series of $L_\la$ to Enright--Shelton reduction, by interpreting the numerator as the Hilbert series of a finite-dimensional $\g'$-module determined by the weight $\la'$.
In~\cite{Armour}*{Table 4.3} and~\cite{EricksonHunziker23}*{Table 2, entries I--IIIa}, the data for Enright--Shelton reduction (where $\lambda = -kc\zeta$) is given in the context of the dual pairs in Table~\ref{table:Howe}.
Upon inspecting each pair $(\g, \g')$ in the tables cited above, we observe that the corner poset isomorphisms~\eqref{Cor isos g} can be described uniformly in terms of Enright--Shelton reduction, as follows:
\begin{equation}
    \label{reduction isos}
    \operatorname{Cor}_k \P(\g,\k) \cong \P(\g',\k').
\end{equation}
Nor is this phenomenon limited to the classical groups; indeed, at least on the level of covering graphs (although not necessarily Hasse diagrams), it is exhibited by every instance of the Wallach representations shown above in Sections~\ref{sub:Wallach Dn}--\ref{sub:Wallach E7 k2}.
The pairs $(\g', \k')$ given below are all listed in~\cite{EnrightHunziker04}*{\S6}:

\begin{itemize}
\item 
For the first Wallach representation of $(\g,\k) = (\ssD_n, \ssD_{n-1})$, we have $\g' = \ssA_1$, and thus $\P(\g') \cong \P_{\GL}(1,1)$ is a singleton.
Likewise, we see from the figure in Section~\ref{sub:Wallach Dn} that $\operatorname{Cor}_1 \P(\g,\k)$ is a single point.

\item For the first Wallach representation of $(\g,\k) = (\ssE_6, \ssD_5)$, we have $(\g', \k') = (\ssA_5, \ssA_4)$, and thus $\P(\g') \cong \P_{\GL}(4,1)$ is a chain of size 4.
This is true also of $\operatorname{Cor}_1 \P(\g)$ as shown in the figure in Section~\ref{sub:Wallach E6}.

\item For the first Wallach representation of $(\g,\k) = (\ssE_7, \ssE_6)$, we have $(\g', \k') = (\ssD_6, \ssD_5)$, and thus $\P(\g',\k')$ is the poset shown in~\eqref{Hasse D6}.
Clearly this has the same covering graph as the shaded corner poset $\operatorname{Cor}_1 \P(\g,\k)$ in the figure in Section~\ref{sub:Wallach E7 k1}.

\item For the second Wallach representation of $(\ssE_7, \ssE_6)$, we have $\g' = \ssA_1$, and thus $\P(\g',\k')$ is a singleton, as is $\operatorname{Cor}_2 \P(\g,\k)$ in the figure in Section~\ref{sub:Wallach E7 k2}.

\end{itemize}
For the present, we leave~\eqref{Cor isos g} as a mere observation, but it seems that there is a deeper relationship between Enright--Shelton reduction and $\operatorname{Cor}_k \P(\g,\k)$ which has yet to be fully understood.

\section*{Appendix: explicit maps in Howe duality settings}

\subsection*{$\operatorname{GL}_k\!-\mathfrak{gl}_{p+q}$ Howe duality:}\

The Lie algebra homomorphism $\omega: \mathfrak{gl}_{p+q} \rightarrow \mathcal{D}(\operatorname{M}_{p,k}\oplus\operatorname{M}_{q,k})^{\operatorname{GL}_k}$ is given by

\begin{equation*}
\begin{split}
\begin{pmatrix}
A & B\\
C & D
\end{pmatrix}\mapsto
&\ \sum_{i=1}^p\sum_{j=1}^p a_{ij} \left(-\sum_{\ind=1}^k x_{j\ind} \frac{\partial}{\partial x_{i\ind}}-k \delta_{ij}\right)
+\sum_{i=1}^q\sum_{j=1}^q d_{ij}\left(\sum_{\ind=1}^k y_{i\ind} \frac{\partial}{\partial y_{j\ind}}\right)\\
&\quad + \sum_{i=1}^p\sum_{j=1}^qb_{ij}\bigg( - \frac{1}{\eta}\underbrace{\sum_{\ind=1}^k  \frac{\partial^2}{\partial x_{i\ind}\partial y_{j\ind}}}_{\Delta_{ij}}\bigg)
 +\sum_{i=1}^q\sum_{j=1}^p c_{ij} \bigg(\eta \underbrace{\sum_{\ind=1}^k x_{j\ind} y_{i\ind}}_{f_{ji}}\bigg). 
\end{split}
\end{equation*}

The action of $\operatorname{GL}_p\times \operatorname{GL}_q\times \operatorname{GL}_k$ on $\mathbb{C}[\operatorname{M}_{p,k}\oplus\operatorname{M}_{q,k}]$ is given by 
\[
(g_1,g_2,h).f(X,Y)=\det(g_1)^{-k} f(g_1^{-1}Xh, g_2^{t}\,Y(h^t)^{-1}).
\]

\subsection*{$\operatorname{O}_k\!-\mathfrak{sp}_{2n}$ Howe duality:}\

The Lie algebra homomorphism  $\omega: \mathfrak{sp}_{2n}\rightarrow \mathcal{D}(\operatorname{M}_{n,k})^{\operatorname{O}_k}$ is given by
\begin{equation*}
\begin{split}
\begin{pmatrix}
A & B\\
C & -A^T
\end{pmatrix}\mapsto
&\ \sum_{i=1}^n\sum_{j=1}^n a_{ij} \bigg(-\sum_{\ind=1}^k x_{j\ind} \frac{\partial}{\partial x_{i\ind}}-\frac{k}{2} \delta_{ij}\bigg)\\
&\quad +\sum_{i=1}^n\sum_{j=1}^nb_{ij}\bigg( -\frac{1}{2\eta}\underbrace{\sum_{\ind=1}^k  \frac{\partial^2}{\partial x_{i\ind}\partial x_{j\ind}}}_{\Delta_{ij}}\bigg)
+\sum_{i=1}^n\sum_{j=1}^n c_{ij} \bigg(\frac{\eta}{2}\underbrace{\sum_{\ind=1}^k  x_{j\ind} x_{i\ind}}_{f_{ji}}\bigg).  
\end{split}
\end{equation*}

The action of $\widetilde{\operatorname{GL}}_n\times \operatorname{O}_k$ on $\mathbb{C}[\operatorname{M}_{n,k}]$ is given by 
\[((g,s),h).f(X)=\det(g)^{-k/2}f(g^{-1}Xh):=s^{-k}f(g^{-1}Xh),\]
where $\widetilde{\operatorname{GL}}_n:=\{(g,s)\in \operatorname{GL}_n\times \mathbb{C}^{\times} \mid \det(g)=s^2\}$.

\subsection{$\operatorname{Sp}_{2k}\!-\mathfrak{so}_{2n}$ Howe duality:}\

The Lie algebra homomorphism  $\omega: \mathfrak{so}_{2n}\rightarrow \mathcal{D}(\operatorname{M}_{n,2k})^{\operatorname{Sp}_{2k}}$ is given by
\begin{equation*}
\begin{split}
\begin{pmatrix}
A & B\\
C & -A^T
\end{pmatrix}\mapsto
&\ \sum_{i=1}^n\sum_{j=1}^n a_{ij} \bigg(-\sum_{\ind=1}^{2k} x_{j\ind} \frac{\partial}{\partial x_{i\ind}}-k \delta_{ij}\bigg)\\
&\quad +\sum_{i=1}^n\sum_{j=1}^nb_{ij}\bigg(-\frac{1}{2\eta}\underbrace{\sum_{\ind=1}^k \Big( \frac{\partial^2}{\partial x_{i\ind}\partial x_{j,\ind+k}}
- \frac{\partial^2}{\partial x_{i,\ind+k}\partial x_{j,\ind}}\Big)}_{\Delta_{ij}}\bigg) \\[-10pt]
&\quad +\sum_{i=1}^n\sum_{j=1}^n c_{ij} \bigg(\frac{\eta}{2}\underbrace{\sum_{\ind=1}^k  \Big(x_{j\ind} x_{i,\ind+k}-x_{j,\ind+k} x_{i\ind}\Big)}_{f_{ji}}\bigg).  
\end{split}
\end{equation*}

The action of $\operatorname{GL}_n\times \operatorname{Sp}_{2k}$ on $\mathbb{C}[\operatorname{M}_{n,2k}]$ is given by
\[
(g,h).f(X)=\det(g)^{-k} f(g^{-1}Xh).
\]

\subsection*{Maps between coordinate functions:}\

In each of the three cases above, identify $\mathfrak{p}^-$ with $(\mathfrak{p}^+)^*$ via the $K$-equivariant linear isomorphism
\[
\begin{pmatrix}
0 & 0\\
C & 0
\end{pmatrix}
\mapsto 
\left[
\begin{pmatrix}
0 & B\\
0 & 0
\end{pmatrix}
\mapsto 
\operatorname{trace}(CB)
\right].
\]
Then define the matrix coordinate functions $z_{ij}\in (\mathfrak{p}^+)^*$  by
\[
z_{ij} :\begin{pmatrix}
0 & B\\
0 & 0
\end{pmatrix}\mapsto b_{ij}.
\]
Via the inverse isomorphism  $(\mathfrak{p}^+)^*\rightarrow \mathfrak{p}^-$,
we have
\[
z_{ij}\mapsto 
\begin{cases}
\begin{pmatrix}
0 & 0\\
E_{ji} & 0
\end{pmatrix}
 & \mbox{if $\mathfrak{g}=\mathfrak{gl}_{p+q}$ and $1\leq i\leq p$, $1\leq j\leq q$;}\\[20pt]
 \begin{pmatrix}
0 & 0\\
\frac{1}{2}(E_{ji}+E_{ij})  & 0
\end{pmatrix}
& \mbox{if $\mathfrak{g}=\mathfrak{sp}_{2n}$ and $1\leq i\leq  j\leq n$;}\\[20pt]
 \begin{pmatrix}
0 & 0\\
\frac{1}{2}(E_{ji}-E_{ij})  & 0
\end{pmatrix}&  \mbox{if $\mathfrak{g}=\mathfrak{so}_{2n}$ and $1\leq i<  j\leq n$.}\\
\end{cases} 
\]
It follows that by choosing $\eta=1$ if $\mathfrak{g}=\mathfrak{gl}_{p+q}$ and $\eta =2$ if $\mathfrak{g}=\mathfrak{sp}_{2n}$ or $\mathfrak{so}_{2n}$, respectively, we have 
\[
z_{ij} \mapsto f_{ij}
\]
via $\mathbb{C}[z_{ij}]=\mathbb{C}[\mathfrak{p}^+]\cong \mathcal{U}(\mathfrak{p}^-)\subset \mathcal{U}(\mathfrak{g}) \rightarrow \mathcal{D}(W)^{H}$.

 \bibliographystyle{alpha}
 \bibliography{references}

\end{document}